\documentclass[
reqno]{amsart} 
\usepackage[T1]{fontenc}
\usepackage[english]{babel}
\usepackage{amssymb,bbm,enumerate}
\usepackage{color}
\usepackage[linktocpage=true,colorlinks=true, linkcolor=blue, citecolor=red, urlcolor=green]{hyperref}

\usepackage{xypic}

\usepackage{tikz-cd}

\renewcommand{\phi}{\varphi}
\renewcommand{\ker}{\Ker}
\def\ch{\mathrm{ch}}
\def\chb{\mathrm{ch,b}}
\def\cl{\mathrm{cl}}

\def\Har{\mathrm{Har}}
\def\dd{\partial}
\def\al{\alpha}
\def\be{\beta}
\def\si{\sigma}

\def\ga{\gamma}
\def\Ga{\Gamma}
\def\la{\lambda}

\def\De{\Delta}

\newcommand{\mc}[1]{\mathcal{#1}}
\newcommand{\mf}[1]{\mathfrak{#1}}
\newcommand{\mb}[1]{\mathbb{#1}}

\newcommand{\tint}{{\textstyle\int}}

\DeclareMathOperator{\Hom}{Hom}
\DeclareMathOperator{\End}{End}

\DeclareMathOperator{\ad}{ad}

\DeclareMathOperator{\Der}{Der}
\DeclareMathOperator{\Inder}{Inder}
\DeclareMathOperator{\Cas}{Cas}
\DeclareMathOperator{\Ker}{Ker}
\DeclareMathOperator{\Span}{span}

\DeclareMathOperator{\fil}{F}
\DeclareMathOperator{\gr}{gr}
\DeclareMathOperator{\cur}{Cur}
\DeclareMathOperator{\Vir}{Vir}

\newcommand{\vir}{\mathrm{Vir}}

\newcommand{\PV}{\mathop{\rm PV }}

\newcommand{\vac}{|0\rangle}

\makeatletter
\def\smallunderbrace#1{\mathop{\vtop{\m@th\ialign{##\crcr
   $\hfil\displaystyle{#1}\hfil$\crcr
   \noalign{\kern3\p@\nointerlineskip}%
   \tiny\upbracefill\crcr\noalign{\kern3\p@}}}}\limits}

\theoremstyle{plain}
\newtheorem{theorem}{Theorem}[section]
\newtheorem{lemma}[theorem]{Lemma}
\newtheorem{proposition}[theorem]{Proposition}
\newtheorem{corollary}[theorem]{Corollary}
\newtheorem{conjecture}[theorem]{Conjecture}

\theoremstyle{definition}
\newtheorem{definition}[theorem]{Definition}
\newtheorem{example}[theorem]{Example}

\theoremstyle{remark}
\newtheorem{remark}[theorem]{Remark}

\numberwithin{equation}{section}

\definecolor{light}{gray}{.9}

\begin{document}

\title{Computation of cohomology of vertex algebras}

\author{Bojko Bakalov}
\address{Department of Mathematics, North Carolina State University,
Raleigh, NC 27695, USA}
\email{bojko\_bakalov@ncsu.edu}

\author{Alberto De Sole}
\address{Dipartimento di Matematica, Sapienza Universit\`a di Roma,
P.le Aldo Moro 2, 00185 Rome, Italy}
\email{desole@mat.uniroma1.it}
\urladdr{www1.mat.uniroma1.it/$\sim$desole}

\author{Victor G. Kac}
\address{Department of Mathematics, MIT,
77 Massachusetts Ave., Cambridge, MA 02139, USA}
\email{kac@math.mit.edu}



\begin{abstract}
We review cohomology theories corresponding to the chiral and classical operads.
The first one is the cohomology theory of vertex algebras,
while the second one is the classical cohomology of Poisson vertex algebras (PVA),
and we construct a spectral sequence relating them.
Since in ``good'' cases the classical PVA cohomology coincides with the variational PVA cohomology
and there are well-developed methods to compute the latter,
this enables us to compute the cohomology of vertex algebras in many interesting cases.
Finally, we describe a unified approach to integrability through vanishing 
of the first cohomology,
which is applicable to both classical and quantum systems of Hamiltonian PDEs.
\end{abstract}

\maketitle

\tableofcontents

\pagestyle{plain}

\section{Introduction}\label{sec:1}

In the series of papers \cite{BDSHK18,BDSHK19,BDSK19,BDSKV19,BDSHKV20}
we developed, with our collaborators,
the foundations of cohomology theory of vertex algebras.

This theory is the last in the series of cohomology theories
beyond the Lie (super)algebra cohomology,
which are intimately related to each other.
All these theories are based on a $\mb Z$-graded Lie superalgebra
\begin{equation}\label{eq:intro1}
W_{\mc P}=\bigoplus_{k\geq-1}W_{\mc P}^k\,,
\end{equation}
associated to a linear symmetric operad $\mc P$,
governing the corresponding cohomology theory.
Recall that $W_{\mc P}^{k-1}=\mc P(k)^{S_k}$
and the Lie superalgebra bracket in $W_{\mc P}$
is defined in terms of the $\circ_i$-products of the operad $\mc P$,
see \cite{Tam02} or \cite{BDSHK18} for details.
An odd element $X\in W_{\mc P}^1$ satisfying $[X,X]=0$
defines a cohomology complex $(W_{\mc P},\ad X)$,
which is a differential graded Lie superalgebra.

The most known example of this construction is the Lie (super)algebra cohomology.
In this case one takes the well-known operad $\mc Hom(V)$
(also often denoted $\mc End(V)$)
for which $\mc Hom(V)(n)=\Hom(V^{\otimes n},V)$,
where $V$ is a fixed vector superspace,
and the action of $S_n$ is by permutation of the factors of $V^{\otimes n}$.
Then $W_{\mc Hom(V)}$ is the Lie superalgebra of polynomial vector fields on $V$
with the $\mb Z$-grading defined by letting $\deg V=-1$.
Furthermore, odd elements $X\in W_{\mc Hom(\Pi V)}^1$
(where $\Pi$ stands for parity reversal)
such that $[X,X]=0$,
correspond bijectively to Lie superalgebra structures on $V$
by letting
\begin{equation}\label{eq:intro2}
[a,b]=(-1)^{p(a)}X(a\otimes b)
\,\,,\qquad
a,b\in V
\,.
\end{equation}
The complex $(W_{\mc Hom(\Pi V)},\ad X)$
is the Chevalley--Eilenberg cohomology complex of the Lie superalgebra $V$
with Lie bracket \eqref{eq:intro2}, with coefficients in the adjoint module.
Moreover, given a $V$-module $M$,
we extend the Lie (super)algebra structure of $V$ to $V\oplus M$
by taking $M$ to be an abelian ideal,
and let $\widetilde X\in W_{\mc Hom(\Pi V\oplus\Pi M)}^1$ 
be the element corresponding to this Lie (super)algebra structure.
Then a natural reduction of the complex $(W_{\mc Hom(\Pi V\oplus\Pi M)},\ad\widetilde X)$
produces the Chevalley--Eilenberg cohomology complex
of $V$ with coefficients in $M$.
Note that, while the cohomology of $V$ with coefficients in the adjoint module
inherits the Lie superalgebra structure from $W_{\mc Hom(\Pi V)}$,
this is not the case for the reduction.

Returning to vertex algebras,
recall that they were defined in \cite{B86}
as algebras with bilinear products labeled by $n\in\mb Z$
satisfying the rather complicated Borcherds identity.
According to an equivalent, Poisson-like definition given in \cite{BK03}, a vertex algebra is a module $V$
over the algebra of polynomials $\mb F[\partial]$,
where $\partial$ is an even endomorphism,
endowed with the following two structures:
a structure of a Lie conformal (super)algebra (LCA),
defined by a $\lambda$-bracket
\begin{equation}\label{eq:intro3}
V\otimes V\,\to\, V[\lambda]
\,\,,\,\,\,\,
a\otimes b\mapsto [a_\lambda b]
\,,
\end{equation}
satisfying axioms L1--L3 from Definition \ref{def:lca} in Section \ref{sec:2.1},
and a bilinear product 
$$
V\otimes V\,\to\, V
\,\,,\,\,\,\,
a\otimes b\mapsto \,{:}ab{:}
\,,
$$
called the normally ordered product,
which defines a commutative and associative,
up to ``quantum corrections,'' differential (super)algebra structure on $V$.
These two operations are related by the Leibniz rule
up to a ``quantum correction''.
See \eqref{eq:qc}--\eqref{eq:wi} in Section \ref{sec:4.1} for the precise identities.

Since the LCA structure is an important part of a vertex algebra structure,
the first step towards vertex algebra cohomology
is the LCA cohomology theory.
The latter was constructed in \cite{BKV99}
and in the more general setting of Lie pseudoalgebras in \cite{BDAK01};
see also \cite{DSK09} for a correction of \cite{BKV99}.
In \cite{DSK13} the $\mb Z$-graded Lie superalgebra ``governing'' the LCA cohomology
was introduced,
and in \cite[Sec. 5.2]{BDSHK18} the corresponding operad $\mc P=\mc Chom(V)$
was explicitly constructed for an $\mb F[\partial]$-module $V$.
The construction goes as follows.

Introduce the vector superspaces
\begin{equation}\label{eq:1.3}
V_n=V[\lambda_1,\dots,\lambda_n]/\langle\partial+\lambda_1+\dots+\lambda_n\rangle
\,,
\end{equation}
where all $\lambda_i$ have even parity and $\langle\Phi\rangle$
stands for the image of the endomorphism $\Phi$.
Then 
\begin{equation}\label{eq:1.4}
\mc Chom(V)(n)
\subset\Hom(V^{\otimes n},V_n)
\end{equation}
consists of all maps $Y_{\lambda_1,\dots,\lambda_n}\colon V^{\otimes n}\to V_n$
satisfying the sesquilinearity property ($1\leq i\leq n$):
\begin{equation}\label{eq:1.5}
Y_{\lambda_1,\dots,\lambda_n}(v_1\otimes\dots\otimes\partial v_i \otimes\dots\otimes v_n)
=
-\lambda_i \,
Y_{\lambda_1,\dots,\lambda_n}(v_1\otimes\dots\otimes v_n)
\,.
\end{equation}
The action of $S_n$ on $\mc Chom(V)(n)$
is given by the simultaneous permutation of the factors of $V^{\otimes n}$
and the $\lambda_i$'s.
The construction of the products $\circ_i$
can be found in \cite{BDSHK18}.

Then 
odd elements $X\in W^1_{\mc Chom(\Pi V)}$ bijectively correspond
to \emph{skewsymmetric} $\lambda$-\emph{brackets} on $V$,
i.e., maps $[\cdot\,_\lambda\,\cdot]\colon  V^{\otimes2}\to V[\lambda]$
satisfying the sesquilinearity L1
and skewsymmetry L2 from Definition \ref{def:lca}.
Explicitly, this bijection is given by
\begin{equation}\label{eq:1.6}
[a_\lambda b]=(-1)^{p(a)}X_{\lambda,-\lambda-\partial}(a\otimes b)
\,.
\end{equation}
Finally, the condition $[X,X]=0$ is equivalent to the Jacobi identity L3.

Thus, taking for $X\in W^1_{\mc Chom(\Pi V)}$ 
the map corresponding to the LCA structure on $V$
defined by \eqref{eq:1.6},
we obtain the cohomology complex $(W_{\mc Chom(\Pi V)},\ad X)$,
with the structure of a differential graded Lie superalgebra.
The cohomology of this complex is the LCA cohomology complex with coefficients in the adjoint
module.
By a reduction, mentioned above,
one defines the LCA cohomology complex of $V$ with coefficients in an arbitrary $V$-module.

Yet another important to us example is the \emph{variational Poisson vertex (super)algebra 
(PVA) cohomology} \cite{DSK13}.
Recall that a PVA $\mc V$ is an $\mb F[\partial]$-module
equipped with a structure of a unital commutative associative differential superalgebra
with derivation $\partial$, and a structure of an LCA,
such that the Leibniz rule L4 from Definition \ref{def:pva} holds.
In other words, a PVA is an ``approximation'' of a vertex algebra for which all quantum 
corrections disappear.
The variational PVA cohomology complex is constructed for a commutative associative differential
superalgebra $\mc V$
by considering the subalgebra 
\begin{equation}\label{eq:1.11}
W_{\PV}(\Pi\mc V)
=
\bigoplus_{k=-1}^\infty W^k_{\PV}(\Pi\mc V)
\end{equation}
of the Lie superalgebra $W_{\mc Chom(\Pi\mc V)}$
consisting of all maps $Y$ satisfying, besides the sesquilinearity property \eqref{eq:1.5}
and the $S_{k}$-invariance,
the Leibniz rule \eqref{eq:leib} in Section \ref{sec:clpva}.
Then odd elements $X\in W^1_{\PV}(\Pi\mc V)$
correspond bijectively via \eqref{eq:1.6}
to skewsymmetric $\lambda$-brackets on $V$ satisfying the Leibniz rule L4.
The condition $[X,X]=0$ is again equivalent to the Jacobi identity L3;
hence such $X$ correspond bijectively to PVA structures on the differential algebra $\mc V$.
The resulting complex $(W_{\PV}(\Pi\mc V),\ad X)$
is called the variational cohomology complex of the PVA $\mc V$
with coefficients in the adjoint module.
As mentioned above, given a $\mc V$-module $M$
one defines the corresponding cohomology complex
with coefficients in $M$
by a simple reduction procedure.
The corresponding variational PVA cohomology is denoted by
\begin{equation}\label{eq:1.12}
H_{\PV}(\mc V,M)
=
\bigoplus_{n=0}^\infty 
H^n_{\PV}(\mc V,M)
\,.
\end{equation}
Here and for all other cohomology theories,
we shift the indices by $1$ as compared with \eqref{eq:1.11}
in order to keep the traditional notation.

The goal of the present paper is to develop methods of
computation of the vertex algebra cohomology introduced in \cite{BDSHK18}.
This cohomology is defined by considering the operad $\mc P^{\ch}(V)$,
which is a local version of the chiral operad of Beilinson and Drinfeld \cite{BD04}
associated to a $\mc D$-module on a smooth algebraic curve $X$,
by considering $X=\mb F$ and the $\mc D$-module translation equivariant.
We showed that in this case the operad $\mc P^{\ch}(V)$ 
admits a simple description,
which is an enhancement of the operad $\mc Chom(V)$ described above.

In order to describe this construction,
let $\mc O^{\star,T}_n=\mb F[z_i-z_j,(z_i-z_j)^{-1}]_{1\leq i<j\leq n}$.
For an $\mb F[\partial]$-module $V$,
the superspace $\mc P^{\ch}(V)(n)$
is defined as the set of all linear maps
\begin{equation}\label{eq:1.13}
Y\colon V^{\otimes n}\otimes\mc O^{\star,T}_n\to V_n
\,\,,\qquad
v_1\otimes\dots\otimes v_n\otimes f\mapsto
Y_{\lambda_1,\dots,\lambda_n}(v_1\otimes\dots\otimes v_n\otimes f)
\,,
\end{equation}
where $V_n$ is defined by \eqref{eq:1.3},
satisfying the following two sesquilinearity properties ($1\leq i\leq n$):
\begin{equation}\label{eq:1.14}
Y_{\lambda_1,\dots,\lambda_n}(v_1\otimes\dots\otimes (\partial+\lambda_i)v_i \otimes\dots\otimes v_n\otimes f)
=
Y_{\lambda_1,\dots,\lambda_n}\Bigl(v_1\otimes\dots\otimes v_n\otimes \frac{\partial f}{\partial z_i}\Bigr)
\,,
\end{equation}
and
\begin{equation}\label{eq:1.15}
Y_{\lambda_1,\dots,\lambda_n}(v_1\otimes\dots\otimes v_n\otimes (z_i-z_j)f)
=
\Bigl(
\frac{\partial}{\partial\lambda_j}
-
\frac{\partial}{\partial\lambda_i}
\Bigr)
Y_{\lambda_1,\dots,\lambda_n}(v_1\otimes\dots\otimes v_n\otimes f)
\,.
\end{equation}
(Note that \eqref{eq:1.14} turns into \eqref{eq:1.5} if $f=1$.)
In \cite{BDSHK18} we also defined the action of $S_n$ on $\mc P^{\ch}(V)(n)$
and the $\circ_i$-products, making $\mc P^{\ch}(V)$ an operad.

As a result,
we obtain the Lie superalgebra 
$$
W_{\ch}(V):=W_{\mc P^{\ch}(V)}=\bigoplus_{k=-1}^\infty W_{\ch}^k(V)
\,,
$$
see \eqref{eq:intro1}.
We show in \cite{BDSHK18}
that odd elements $X\in W_{\ch}^1(\Pi V)$
such that $[X,X]=0$
correspond bijectively to vertex algebra structures on the $\mb F[\partial]$-module $V$ 
such that $\partial$ is the translation operator.
As before, this leads to the vertex algebra cohomology
$$
H_{\ch}(V,M)=\bigoplus_{n=0}^\infty H_{\ch}^n(V,M)
\,,
$$ 
for any $V$-module $M$.

Now suppose that the $\mb F[\partial]$-module $V$
is equipped with an increasing $\mb Z_+$-filtration by $\mb F[\partial]$-submodules.
Taking the increasing filtration of $\mc O^{\star,T}_n$ by the number of divisors,
we obtain an increasing filtration of $V^{\otimes n}\otimes\mc O^{\star,T}_n$.
This filtration induces a decreasing filtration of the superspace $\mc P^{\ch}(V)(n)$.
The associated graded spaces $\gr\mc P^{\ch}(V)(n)$
form a graded operad.

On the other hand, in \cite{BDSHK18}
we introduced the closely related operad $\mc P^{\cl}(V)$,
which ``governs'' the Poisson vertex algebra structures on the $\mb F[\partial]$-module $V$.
The vector superspace $\mc P^{\cl}(V)(n)$
is the space of linear maps (cf. \eqref{eq:1.13})
\begin{equation}\label{eq:1.16}
Y\colon V^{\otimes n}\otimes\mc G(n)
\to V_n \,,
\qquad v\otimes\Gamma \mapsto Y^\Gamma(v)
\,,
\end{equation}
where $\mc G(n)$ is the space spanned by oriented graphs with $n$ vertices,
subject to certain conditions.
The corresponding $\mb Z$-graded Lie superalgebra 
$$
W_{\cl}(\Pi V)=\bigoplus_{k=-1}^\infty W_{\cl}^k(\Pi V)
$$
is such that odd elements $X\in W_{\cl}^1(\Pi V)$ with $[X,X]=0$
parametrize the PVA structures on the $\mb F[\partial]$-module $V$
by (cf.\ \eqref{eq:1.6}):
\begin{equation}\label{eq:1.17}
ab=(-1)^{p(a)}X^{\bullet\to\bullet}(a\otimes b) 
\,,\qquad
[a_\lambda b]=
(-1)^{p(a)}X^{\bullet\,\,\,\bullet}_{\lambda,-\lambda-\partial}(a\otimes b)
\,.
\end{equation}
This leads to the classical PVA cohomology
$$
H_{\cl}(V,M)
=
\bigoplus_{n=0}^\infty H^n_{\cl}(V,M)
\,.
$$

Assuming that $V$ is endowed with an increasing $\mb Z_+$-filtration by $\mb F[\partial]$-submodules,
we have a canonical linear map of graded operads
\begin{equation}\label{eq:1.18}
\gr\mc P^{\ch}(V)
\to
\mc P^{\cl}(\gr V)
\,.
\end{equation}
We proved in \cite{BDSHK18}
that the map \eqref{eq:1.18} is injective.
The main result of \cite{BDSHK19} is that this map is an isomorphism provided that the filtration of $V$
is induced by a grading by $\mb F[\partial]$-modules.
If, in addition, this filtration of $V$ is such that $\gr V$ inherits from the vertex algebra structure of $V$ a PVA structure, and $\gr M$ inherits a structure of a PVA module over $\gr V$ (see \cite{Li04, DSK05}), then as
a result, the vertex algebra cohomology is majorized by the classical PVA cohomology
(see Corollary \ref{cor:ss}):
\begin{equation}\label{eq:1.19}
\dim H_{\chb}^n(V,M)
\leq
\dim H_{\cl}^n(V,M)
\,\,,\,\,\,\,
n\geq0
\,.
\end{equation}

Unfortunately, 
we had to replace in \eqref{eq:1.19} the space $H_{\ch}^n(V,M)$
by $H_{\chb}^n(V,M)$.
It is because we were unable to prove that the decreasing filtration on $\mc P^{\ch}(V)$,
induced by the increasing filtration of $V$, is exhaustive.
However, see Section \ref{sec:wick} for the first step in this direction.
Therefore, we have to replace $\mc P^\ch(V)$ by $\mc P^{\chb}(V)$,
which is the union of all members of the filtatrion of $\mc P^{\ch}(V)$,
and introduce the ``bounded'' VA cohomology $H_{\chb}(V,M)$
of the complex $(W_{\mc P^{\chb}(V)},\ad X)$.

Fortunately,
it is easy to show that 
\begin{equation}\label{eq:intro17}
H^1_{\ch}(V,M)
=
H^1_{\chb}(V,M)
\,,
\end{equation}
if $V$ is a finitely strongly generated vertex algebra
(see Proposition \ref{thm:exhcoc}).

Finally, the obvious inclusion of Lie superalgebras 
$W_{\PV}(\Pi\mc V) \hookrightarrow W_{\cl}(\Pi\mc V)$
induces an injective map in cohomology, and we prove in \cite{BDSHKV20}
that this map is an isomorphism, provided that as a differential algebra, $\mc V$ is an algebra of differential polynomials. Hence, we obtain from \eqref{eq:1.19} the following inequality:
\begin{equation}\label{eq:1.20}
\dim H_{\chb}^n(V,M)
\leq
\dim H_{\PV}^n(\gr V,\gr M)
\,,
\end{equation}
provided that as a differential algebra, $\gr V$ is an algebra of differential polynomials.

The inequality \eqref{eq:1.20} is used to obtain upper bounds for the dimension of vertex algebra cohomology,
using the results of \cite{BDSK19} on computation of the variational PVA cohomology.
A more powerful tool is a spectral sequence from classical PVA cohomology to vertex algebra cohomology,
constructed in Section \ref{sec:ss}.
It allows us to obtain in many interesting cases an equality in \eqref{eq:1.19}, hence in \eqref{eq:1.20}.
Some of the resulting computations are given by Theorems \ref{thm:ss21},
 \ref{thm:triv3}--\ref{thm:bos}, 
and are stated in the following theorem.
\begin{theorem}\label{thm:1.1}
\begin{enumerate}[(a)]
\item
Let\/ $V$ be a commutative associative superalgebra and\/ $M$ be a\/ $V$-module.
We can view\/ $V$ as a vertex algebra with $\partial=0$,
zero\/ $\lambda$-bracket, and\/ ${:}ab{:}\,=ab$.
Then we have
$$
H_{\ch}(V,M)=H_{\Har}(V,M)
\,,
$$
where the subscript\/ $\Har$ stands for the Harrison cohomology \cite{Har62}.
\item
Let\/ $V$ be a vertex algebra freely generated by elements $W_0=L, W_1,\dots, W_r$,
where $L$ is a Virasoro element and the $W_i$'s have positive conformal weights.
Then\/ $\dim H_{\chb}^n(V)<\infty$ for all\/ $n\geq0$.
In particular, this holds for all universal $W$-algebras\/ $V=W^k(\mf g,f)$
where\/ $k\ne -h^\vee$.
\item
The bounded $n$-th cohomology of the universal Virasoro vertex algebra $\Vir^c$
with any central charge $c$ is $1$-dimensional for $n=0,2,3$,
and $0$ otherwise.
\item
The bounded $n$-th cohomology of the VA of free superfermions is $0$ if $n\geq1$
and it is $1$-dimensional for $n=0$.
\item
For the VA of free superbosons $B_{\mf h}$ on a superspace $\mf h$,
one has 
$$
H^n_{\chb}(B_{\mf h})
\simeq
\big(S^n(\Pi\mf h)\big)^*
\oplus
\big(S^{n+1}(\Pi\mf h)\big)^*
\,\,,\,\,\,\,
n\geq0
\,.
$$
\end{enumerate}
\end{theorem}
In the conclusion of the paper we explain how to use vanishing of the first PVA (respectively, VA) cohomology
in order to prove integrability of classical (respectively, quantum) Hamiltonian PDEs.

Throughout the paper, the base field $\mb F$ is a field of characteristic $0$
and, unless otherwise specified, all vector spaces, their tensor products and Homs 
are over $\mb F$, and the parity of a vector superspace is denoted by $p$.

\subsubsection*{Acknowledgments} 

This research was partially conducted during the authors' visits
to MIT and to the University of Rome La Sapienza.
The first author was supported in part by a Simons Foundation grant 584741.
The second author was partially supported by the national PRIN fund n. 2015ZWST2C$\_$001
and the University funds n. RM116154CB35DFD3 and RM11715C7FB74D63.
All three authors were supported in part by the Bert and Ann Kostant fund.

\section{Basic definitions}\label{sec:basic}

In this section, we review the definition of a vertex algebra and some related constructions. We start by a short discussion of Lie conformal (super)algebras, which are an important part of the vertex algebra structure.
We also review Poisson vertex algebras, which naturally appear as the associated graded of filtered vertex algebras.

\subsection{Lie conformal algebras}
\label{sec:2.1}

\begin{definition}\label{def:lca}
Let $R$ be a vector superspace with parity $p$,
endowed with an even endomorphism $\partial$.
A \emph{Lie conformal superalgebra} (LCA) structure on $R$
is a bilinear, parity preserving $\lambda$-\emph{bracket}
$R\otimes R\to R[\lambda]$, $a\otimes b\mapsto[a_\lambda b]$,
satisfying ($a,b,c\in R$):
\begin{enumerate}[L1\,]
\item
$[\partial a_\lambda b]=-\lambda[a_\lambda b]$, \;
$[a_\lambda\partial b]=(\lambda+\partial)[a_\lambda b]$ \;
(sesquilinearity);
\item
$[a_\lambda b]=-(-1)^{p(a)p(b)}[b_{-\lambda-\partial}a]$ \;
(skewsymmetry);
\item
$[a_\lambda[b_\mu c]]-(-1)^{p(a)p(b)}[b_\mu[a_\lambda c]]
=
[[a_\lambda b]_{\lambda+\mu}c]$ \;
(Jacobi identity).
\end{enumerate}
A \emph{module} over the LCA $R$
is a vector superspace $M$ with an even endomorphism $\partial$,
endowed with a bilinear, parity preserving $\lambda$-action 
$R\otimes M\to M[\lambda]$, $a\otimes m\mapsto a_\lambda m$,
satisfying ($a,b\in R$, $m\in M$):
\begin{enumerate}[M1]
\item
$(\partial a)_\lambda m=-\lambda a_\lambda m$, \;
$a_\lambda(\partial m)=(\lambda+\partial)(a_\lambda m)$;
\item
$a_\lambda(b_\mu m)-(-1)^{p(a)p(b)}b_\mu(a_\lambda m)
=
[a_\lambda b]_{\lambda+\mu}m$.
\end{enumerate}
\end{definition}

\begin{example}[Free superboson LCA]\label{ex:boson-lca}
Let $\mf h$ be a finite-dimensional superspace, with parity $p$, and a supersymmetric nondegenerate bilinear form $(\cdot|\cdot)$. By supersymmetry of the form we mean that $(a|b)=(-1)^{p(a)p(b)} (b|a)$ for $a,b\in\mf h$
and $(a|b)=0$ whenever $p(a)\ne p(b)$.
The \emph{free superboson} LCA corresponding to $\mf h$ is the $\mb F[\partial]$-module
$$
R^b_{\mf h}=\mb F[\partial]\mf h\oplus\mb FK
\,,\qquad\text{where}\quad
\partial K=0 \,, \;\; p(K)=\bar0
\,,
$$
endowed with the $\lambda$-bracket
\begin{equation}\label{eq:boson}
[a_\lambda b]=\lambda (a|b) K
\,\,\text{ for }\,\, a,b\in\mf h\,,\qquad
K \;\text{ central}
\end{equation}
(uniquely extended to $R^b_{\mf h}\otimes R^b_{\mf h}$ by the sesquilinearity axioms). 
In the case when $\mf h$ is purely even, i.e., $p(a)=\bar0$ for all $a\in\mf h$, 
the LCA $R^b_{\mf h}$ is called the \emph{free boson} LCA.
\end{example}
\begin{example}[Free superfermion LCA]\label{ex:fermion-lca}
Let $\mf h$ be a finite-dimensional superspace, with parity $p$, and a super-skewsymmetric nondegenerate bilinear form $(\cdot|\cdot)$. Now we have $(a|b)=-(-1)^{p(a)p(b)} (b|a)$ for $a,b\in\mf h$
and $(a|b)=0$ whenever $p(a)\ne p(b)$.
The \emph{free superfermion} LCA corresponding to $\mf h$ is the $\mb F[\partial]$-module
$$
R^f_{\mf h}=\mb F[\partial]\mf h\oplus\mb F K
\,,\qquad\text{where}\quad
\partial K=0 \,, \;\; p(K)=\bar 0
\,,
$$
endowed with the $\lambda$-bracket
\begin{equation}\label{eq:fermion}
[a_\lambda b]=(a|b) K
\,,\qquad
K \;\text{ central}
\end{equation}
(uniquely extended to $R^f_{\mf h}\otimes R^f_{\mf h}$ by the sesquilinearity axioms). 
In the case when $p(a)=\bar1$ for all $a\in\mf h$, the LCA $R^f_{\mf h}$ is called the \emph{free fermion} LCA.
\end{example}
\begin{example}[Affine LCA]\label{ex:affine-lca}
Let $\mf g$ be a Lie algebra with a nondegenerate invariant symmetric bilinear form $(\cdot\,|\,\cdot)$.
The corresponding \emph{affine} LCA is the purely even $\mb F[\partial]$-module
$$
\cur\mf g=\mb F[\partial]\mf g\oplus\mb FK
\,,\;\;
\text{ where }\;
\partial K=0
\,,
$$
endowed with the $\lambda$-bracket
given on the generators by
\begin{equation}\label{eq:current}
[a_\lambda b]=[a,b]+(a|b)\lambda K
\,,\qquad
a,b\in\mf g
\,,\quad
K \text{ central.}
\end{equation}
Note that, in the special case of an abelian Lie algebra $\mf g$,
we recover the definition of the free boson: $\cur\mf g=R^b_{\mf g}$.
\end{example}
\begin{example}[Virasoro LCA]\label{ex:virasoro-lca}
The \emph{Virasoro} LCA is the purely even $\mb F[\partial]$-module
$$
R^{\vir}=\mb F[\partial]L\oplus\mb FC
\,,\;\;
\text{ where }\;
\partial C=0
\,,
$$
endowed with the $\lambda$-bracket
\begin{equation}\label{eq:vir}
[L_\lambda L]=(\partial+2\lambda)L+\frac{1}{12}\lambda^3 C
\,,\qquad
C \text{ central.}
\end{equation}
\end{example}

The importance of the last two examples stems from the fact that the LCA's $\overline{\cur}\,\mf g = \cur\mf g / \mb F K$ for $\mf g$ simple and $\bar R^{\vir} = R^{\vir} / \mb F C$ exhaust all simple LCA's, which are finitely generated as $\mb F[\partial]$-modules \cite{DAK98}.

\subsection{Poisson vertex algebras}
\label{sec:3.1}

\begin{definition}\label{def:pva}
Let $\mc V$ be a commutative associative unital differential superalgebra with parity $p$,
with an even derivation $\partial$.
A \emph{Poisson vertex superalgebra} (PVA) structure on $\mc V$
is an LCA $\lambda$-bracket
$\mc V\otimes \mc V\to \mc V[\lambda]$, $a\otimes b\mapsto[a_\lambda b]$,
such that the following left Leibniz rule holds ($a,b,c\in R$):
\begin{enumerate}[L1]
\setcounter{enumi}{3}
\item
$[a_\lambda bc]=[a_\lambda b] c+(-1)^{p(b)p(c)}[a_\lambda c] b$.
\end{enumerate}

By the skewsymmetry L2, this axiom is equivalent to the right Leibniz rule
\begin{enumerate}[L1']
\setcounter{enumi}{3}
\item
$[ab_\lambda c]=
(e^{\partial\partial_\lambda}a)[b_\lambda c] 
+(-1)^{p(a)p(b)}(e^{\partial\partial_\lambda}b)[a_\lambda c]$.
\end{enumerate}

A \emph{module} $M$ over the PVA $\mc V$
is a vector superspace endowed 
with a structure of a module over the differential algebra $\mc V$,
denoted by $a\otimes m\mapsto a m$,
and with a structure of a module over the LCA $\mc V$,
denoted by $a\otimes m\mapsto a_\lambda m$,
satisfying 
\begin{enumerate}[M1]
\setcounter{enumi}{2}
\item
$a_\lambda(bm)=[a_\lambda b]m+(-1)^{p(a)p(b)}b(a_\lambda m)$;
\end{enumerate}
\begin{enumerate}[M1']
\setcounter{enumi}{2}
\item
$(ab)_\lambda m=
(e^{\partial\partial_\lambda}a)(b_\lambda m) 
+(-1)^{p(a)p(b)}(e^{\partial\partial_\lambda}b)(a_\lambda m)$.
\end{enumerate}
\end{definition}
A PVA $\mc V$ is called \emph{graded}
if there is a grading by $\mb F[\partial]$-submodules
$$
\mc V=\bigoplus_{n\in\mb Z_+}\mc V_n
\,,
$$
such that ($m,n\in\mb Z_+$)
\begin{equation}\label{eq:grading}
\mc V_m\mc V_n \subset\mc V_{m+n}
\,\,,\qquad
[{\mc V_m}_\lambda\mc V_n]
\subset
\mc V_{m+n-1}[\lambda]
\,.
\end{equation}
If $\mc V$ is a graded PVA,
a $\mc V$-module $M$ is \emph{graded} if there is a grading by $\mb F[\partial]$-submodules
$$
M=\bigoplus_{n\in\mb Z_+} M_n \,,
$$
such that ($m,n\in\mb Z_+$)
\begin{equation}\label{eq:graded-module}
\mc V_m M_n \subset M_{m+n}
\,,\qquad
{\mc V_m}_\lambda M_n
\subset
M_{m+n-1}[\lambda]
\,.
\end{equation}
Note that every PVA is a module over itself, called the adjoint module.

\begin{definition}\label{def:vir}
A PVA $\mc V$ is called \emph{conformal} if it has a \emph{Virasoro element}, namely, 
an even element $L\in\mc V$ such that the following properties hold:
$$
[L_\lambda L]=(\partial+2\lambda)L+\frac{c}{12}\lambda^3
\,,\,\,\text{ for some } c\in\mb F \,\,(\text{the \emph{central charge} of } L)
\,,
$$
\begin{equation*}
L_{(0)}
:=[L_\lambda\,\cdot\,]\big|_{\lambda=0}
=\partial \,,
\end{equation*}
and
$$L_{(1)}
:=\frac{d}{d\lambda}[L_\lambda\,\cdot\,]\big|_{\lambda=0}
\,\in\End\mc V
\,\text{ is diagonalizable.}
$$
One says that $a\in\mc V$ 
has \emph{conformal weight} $\Delta(a)\in\mb F$
if it is an eigenvector of $L_{(1)}$ of eigenvalue $\Delta(a)$.
%
 \end{definition}
 
Given an LCA $R$,
there is the canonical universal PVA $\mc V(R)$ over $R$, constructed as follows.
As a commutative associative superalgebra it is $\mc V(R)=S(R)$,
the symmetric superalgebra over $R$, viewed as a vector superspace.
The endomorphism $\partial\in\End R$ uniquely extends 
to an even derivation of the superalgebra $\mc V(R)$.
The $\lambda$-bracket on $R$ uniquely extends to a PVA $\lambda$-bracket on $\mc V(R)$
by the Leibniz rules L4 and L4'.
Note that the universal PVA $\mc V(R)$ over the LCA $R$
is automatically graded, by the usual symmetric superalgebra degree.

If $C\in R$ is such that $\partial C=0$, then $C$ is central in $R$, i.e., $[C_\la R]=0$. In fact, $C$ acts as zero on any $R$-module. Then for any $c\in\mb F$, we have a PVA ideal $\mc V(R)(C-c) \subset \mc V(R)$, and we can consider the quotient PVA
\begin{equation}\label{eq:vcr}
\mc V^c(R) := \mc V(R) / \mc V(R)(C-c)
\,.
\end{equation}
This PVA is not graded unless $c=0$. As a differential superalgebra, $\mc V^c(R)$ is isomorphic to the symmetric algebra $S(\bar R)$, where $\bar R:=R/\mb FC$.

\subsection{Vertex algebras}
\label{sec:4.1}

\begin{definition}\label{def5}
A \emph{vertex (super)algebra} (VA) is a vector superspace $V$
endowed with an even endomorphism $\partial$,
an even element $\vac$ 
and an {\itshape integral of} $\lambda$-{\itshape bracket},
namely a linear map 
\begin{equation}\label{eq:int-lambda}
V \otimes V \to V[\lambda]
\,\,,\,\,\,\,
u\otimes v\mapsto \int^\lambda d\sigma[u_\sigma v]
\,,
\end{equation}
such that the following axioms hold ($u,v,w\in V$):
\begin{enumerate}[V1]
\item
$\int^\lambda d\sigma[\vac_\sigma v]=\int^\lambda d\sigma[v_\sigma \vac]=v$,
\item
$\int^\lambda\!\! d\sigma\, [\partial u_\sigma v] = -\int^\lambda\!\! d\sigma\,\sigma[u_\sigma v]$,
$\int^\lambda\!\! d\sigma\,[u_\sigma \partial v] = \int^\lambda\!\!d\sigma\, (\partial+\sigma)[u_\sigma v]$,
\item 
$\int^\lambda\!\! d\sigma\,[v_\sigma u] =(-1)^{p(u)p(v)}\int^{-\lambda-\partial}\!\! d\sigma\,[u_\sigma v]$,
\item
$\int^\lambda\!\! d\sigma\!\! \int^\mu\!\! d\tau \Big(
[u_\sigma[v_\tau w]] - (-1)^{p(u)p(v)} [v_\tau[u_\sigma w]] - [[u_\sigma v]_{\sigma+\tau}w] 
\Big) = 0$.
\end{enumerate}
If we do not assume the existence of the unit element $\vac\in V$ and drop axiom V1,
we call $V$ a \emph{non-unital vertex algebra}.
\end{definition}
See \cite{DSK06,BDSHK18} for a discussion on the meaning of these axioms
and their equivalence to other definitions of vertex algebras.

The $\lambda$-\emph{bracket} of $u,v\in V$
is defined as the derivative by $\lambda$ of their integral of $\lambda$-bracket:
\begin{equation}\label{eq:lambda}
[u_\lambda v]=\frac{d}{d\lambda} \int^\lambda\!d\sigma\,[u_\sigma v]
\,,
\end{equation}
while their \emph{normally ordered product}
is defined as the constant term:
\begin{equation}\label{eq:nop}
{:}uv{:}\,=\,\int^0\!d\sigma\,[u_\sigma v]
\,.
\end{equation}
Note that by differentiating axioms V2-V4 of a VA we recover the axioms L1-L3 of an LCA.
Hence, the $\lambda$-bracket of a VA $V$ defines a structure of an LCA on $V$.

\begin{remark}\label{rem:nprod}
One defines \emph{$n$-th products} on a VA $V$ for any integer $n$ by the formulas:
$$
u_{(n)}v = \frac{d^n}{d\lambda^n} [u_\la v] \big|_{\la=0} \,, \quad
u_{(-n-1)}v = \frac1{n!} \, {:}(\partial^n u)v{:} \,, \qquad n\ge0 \,.
$$
Then the definition of a VA can be given equivalently in terms of the identities satisfied by these products (see \cite{B86,K96,BK03,DSK06}).
\end{remark}

The notions of a \emph{conformal vertex algebra} and \emph{conformal weight} are defined in exactly the same way as for PVA; see Definition \ref{def:vir}.

An equivalent definition of a vertex algebra, which we will also use, is as follows \cite{BK03}.
A non-unital VA is a quadruple $(V,\partial,[\,{}_\la\,],{:}\,{:})$, where 
$(V,\partial,[\,{}_\la\,])$ is an LCA,
$(V,\partial,{:}\,{:})$ is a (noncommutative and nonassociative) differential algebra, 
such that the following identities hold:
\begin{align}
\label{eq:qc}
{:}ab{:} &- (-1)^{p(a)p(b)} {:}ba{:} = \int_{-\partial}^0 d\si\, [a_\si b] 
\,, \\ \label{eq:qa}
{:}({:}ab{:}) c{:} &- {:}a({:}bc{:}){:}
= {:}\Bigl( \int_0^\partial d\si \, a \Bigr) [b_\si c] {:}
+ (-1)^{p(a)p(b)} {:}\Bigl( \int_0^\partial d\si \, b \Bigr) [a_\si c] {:}
\,, \\ \label{eq:wi}
[a_\la {:}bc{:}] &= {:}[a_\la b]c{:} + (-1)^{p(a)p(b)} {:}b[a_\la c]{:} + \int_0^\la d\si\, [a_\la [b_\si c]]
\,.
\end{align}
Equations \eqref{eq:qc}, \eqref{eq:qa} and \eqref{eq:wi} are known as the
quasicommutativity, quasiassociativity and non-commutative Wick formula, respectively.
To get the definition of a VA, one requires, in addition, the existence of an even vector $\vac$ such that
${:}a \vac{:} = a$ for every $a$.

Our convention will be that the normally ordered product of more than two elements is defined inductively from right to left; for example,
$$
{:} abc {:} = {:}a({:}bc{:}){:} \,, \qquad
{:} abcd {:} = {:}a({:}b({:}cd{:}){:}){:} \,, \;\text{ etc.}
$$
One says that a (non-unital) VA $V$ is \emph{strongly generated} by an $\mb F[\partial]$-submodule $R\subset V$ if $V$ is the span over $\mb F$ of all elements of the form
$$
{:}a_1 a_2 \cdots a_k{:}
\,, \qquad
k\ge1 \,, \; a_i \in R
$$
(together with $\vac$ when it exists).
A VA $V$ is called \emph{finitely generated},
if it is strongly generated by a finitely-generated $\mb F[\partial]$-submodule $R$ of $V$.
We say that $V$ is \emph{freely generated} by $R$, if $V$ has a \emph{PBW-type basis}, i.e.,
for some ordered $\mb F$-basis $\{a_i\}_{i\in I}$ of $R$, compatible with parity, 
the ordered monomials
\begin{equation}\label{eq:PBW}
\bigl\{
{:} a_{i_1} a_{i_2}\cdots a_{i_s} {:}
\,\big|\,
i_\ell\leq i_{\ell+1} \;\;\forall\ell,\;\text{ and } \; i_\ell< i_{\ell+1}
\;\text{ if }\; p(a_{i_\ell})=\bar1
\bigr\}
\end{equation}
form an $\mb F$-basis of $V$. Then any other ordered $\mb F$-basis of $R$, compatible with parity, gives a PBW-type basis of $V$.

\begin{definition}\label{def-mod}
A left \emph{module} $M$ over a vertex algebra $V$ 
is a $\mb Z /2\mb Z$-graded $\mb F[\partial]$-module
endowed with an {\itshape integral of} $\lambda$-{\itshape action},
\begin{equation}\label{eq:int-lambda2}
V \otimes M \to M[\lambda]
\,\,,\,\,\,\, 
v\otimes m\mapsto \int^\lambda\!d\sigma\,(v_\sigma m)\,,
\end{equation}
preserving the $\mb Z/2\mb Z$-grading,
such that the following axioms hold ($u,v\in V$, $m\in M$):
\begin{enumerate}[VM1]
\item 
$\int^\lambda\!\! d\sigma\,\vac_\sigma m=m$,
\item
$\int^\lambda\!\! d\sigma\, (\partial v_\sigma m) = -\int^\lambda\!\! d\sigma\,\sigma(v_\sigma m)$,
$\int^\lambda\!\! d\sigma\,(v_\sigma \partial m) = \int^\lambda\!\!d\sigma\, (\partial+\sigma)(v_\sigma m)$,
\item
$\int^\lambda\!\! d\sigma\!\! \int^\mu\!\! d\tau \Big(
u_\sigma(v_\tau m) - (-1)^{p(u)p(v)} v_\tau(u_\sigma m) - [u_\sigma v]_{\sigma+\tau}m 
\Big) = 0$.
\end{enumerate}
If $V$ is a non-unital vertex algebra, then axiom VM1 is dropped from the definition of a $V$-module.
In analogy with the notation used for vertex algebras,
we call the \emph{normally ordered action} ${:}vm{:}$ the constant term 
of the integral of $\lambda$-action \eqref{eq:int-lambda2}:
\begin{equation}\label{eq:int-lambda3}
\int^\lambda\!d\sigma\,(v_\sigma m)
=
\,{:}vm{:}\,+\int_0^\lambda\!d\sigma\,(v_\sigma m)
\,,
\end{equation}
and the $\lambda$-\emph{action} $v_\lambda m$ the derivative of \eqref{eq:int-lambda3}.

\end{definition}

\subsection{Filtrations}\label{sec:filva}

Recall that an increasing (resp. decreasing) $\mb Z$-filtration of a vector superspace $V$ by subspaces
\begin{equation}\label{eq:filtration}
\dots\subset\fil^{p-1}V\subset \fil^pV\subset \fil^{p+1}V\subset \cdots
\,\bigl(\text{resp.}
\dots\supset\fil^{p-1}V\supset \fil^pV\supset \fil^{p+1}V\supset \cdots
\bigr)
\end{equation}
is called \emph{exhaustive} if
\begin{equation}\label{eq:exfil}
\bigcup_{m\in\mb Z} \fil^m V = V
\,,
\end{equation}
and is called \emph{separated} if
\begin{equation}\label{eq:sepfil-def}
\bigcap_{m\in\mb Z} \fil^m V = \{0\}
\,.
\end{equation}
An increasing $\mb Z$-filtration is called a $\mb Z_+$-filtration if $\fil^pV=\{0\}$ for $p<0$.
\begin{remark}\label{rem:nonint}
All constructions and arguments of the present paper apply if we replace $\mb Z$-filtrations
by $\frac1N\mb Z$-filtrations, where $N$ is a positive integer.
For simplicity of the exposition, we will assume $N=1$.
\end{remark}

\begin{definition}[\cite{Li04}]\label{def:filtration}
Let $V$ be a (non-unital) VA. 
An increasing $\mb Z_+$-filtration \eqref{eq:filtration} of $V$ 
is called \emph{good} if it is exhaustive, all $\fil^m V$ are $\mb F[\partial]$-submodules of $V$, and
\begin{equation}\label{eq:filtr-cond}
{:}(\fil^m V)(\fil^n V){:}\,\subset \fil^{m+n} V
\,,\qquad
[(\fil^m V)_\lambda (\fil^n V)]\subset (\fil^{m+n-1} V)[\lambda] \,
\end{equation}
for all $m,n\in\mb Z_+$. 
By a \emph{filtered VA}, we mean a VA with a given good $\mb Z_+$-filtration.
\end{definition}
If $V$ is a filtered VA,
a $V$-module $M$ is called \emph{filtered} 
if there is an increasing exhaustive $\mb Z_+$-filtration by $\mb F[\partial]$-submodules
\begin{equation}\label{eq:filtration-M}
\{0\}=\fil^{-1}M\subset \fil^0M\subset \fil^1M\subset \fil^2M\subset\dots\subset M
\,,
\end{equation}
such that ($m,n\in\mb Z_+$)
\begin{equation}\label{eq:filtr-cond-M}
{:}(\fil^m V)(\fil^n M){:}\,\subset \fil^{m+n} M
\,,\qquad
(\fil^m V)_\lambda (\fil^n M)\subset (\fil^{m+n-1} M)[\lambda]
\,.
\end{equation}

\begin{proposition}[{\cite{Li04,DSK05}}]\label{prop:assgr}
Let\/ $V$ be a filtered VA. 
Then the associated graded 
$$
\gr V=\bigoplus_{n\in\mb Z_+} \gr^n V \,, \qquad
\gr^n V := \fil^n V / \fil^{n-1} V \,,
$$
has a natural structure of a graded PVA.
Namely, 
\begin{align*}
\bar u\bar v
&={:}uv{:} + \fil^{m+n-1} V \in \gr^{m+n} V
\,, \\
[\bar u_\lambda \bar v]
&= [u_\lambda v] + (\fil^{m+n-2} V)[\la] \in (\gr^{m+n-1} V)[\la]
\,,
\end{align*}
for\/ $u\in\fil^m V$, $v\in\fil^n V$ such that 
$$
\bar u = u+\fil^{m-1}V \in\gr^m V \,, \qquad
\bar v = v+\fil^{n-1}V \in\gr^n V \,.
$$
\end{proposition}
\begin{proof}
Follows immediately from formulas \eqref{eq:qc}--\eqref{eq:wi}.
\end{proof}

A similar result holds for modules:
if $V$ is a filtered VA and $M$ is a filtered $V$-module as in \eqref{eq:filtration-M},
then the associated graded 
$$
\gr M=\bigoplus_{n\in\mb Z_+} \fil^n M / \fil^{n-1} M
$$ 
has a natural structure of a graded module over the graded PVA $\gr V$.

We will impose an additional assumption on a good filtration, 
which will allow us later to apply the main result of \cite{BDSHK19}.

\begin{definition}\label{def:vgfil}
A good filtration of a VA $V$ is called \emph{very good} if 
$V\simeq\gr V$ as $\mb F[\partial]$-modules.
Likewise, a filtration \eqref{eq:filtration-M}-\eqref{eq:filtr-cond-M}
is called \emph{very good} if $M\simeq\gr M$ as $\mb F[\partial]$-module.
\end{definition}

We present here two examples of very good filtrations, which, although trivial, are still useful.

\begin{example}\label{ex:vgfil1}
For a VA $V$, let 
$$
\fil^{-1} V = \fil^0 V = \{0\} \subset \fil^1 V = \fil^2 V = \cdots = V \,.
$$
This is obviously a very good filtration and $\gr V = \gr^1 V = V$. The $\la$-bracket in $\gr V$, defined by Proposition \ref{prop:assgr}, coincides with the $\la$-bracket in $V$, while the commutative associative product in $\gr V$ is zero.
\end{example}

\begin{example}\label{ex:vgfil2}
Consider a VA $V$ with the filtration 
$$
\fil^{-1} V = \{0\} \subset \fil^0 V = \fil^1 V = \cdots = V \,.
$$
This filtration is good if and only if the $\la$-bracket in $V$ is zero, i.e., $V$ is a \emph{commutative} vertex algebra. 
Then the $\la$-bracket in $\gr V$ is also zero.
Moreover, the normally ordered product in $V$ is commutative and associative, and it coincides with the product in $\gr V = \gr^0 V = V$.
\end{example}

\begin{proposition}\label{prop:vgfil}
Let\/ $V$ be a (non-unital) vertex algebra generated by an\/ $\mb F[\partial]$-submodule\/ $R$.
Suppose that\/ $R$ is decomposed in a direct sum of\/ $\mb F[\partial]$-submodules
\begin{equation}\label{RDe1}
R = \bigoplus_{\De\in\mb Z_{>0}} R_\De \,.
\end{equation}
For\/ $s\in\mb Z$, let
\begin{equation}\label{RDe3}
\fil^s V = \Span_{\mb F} \bigl\{ {:} a_1 \cdots a_k {:} 
\,\big|\, a_i \in R_{\De_i} \,, \; \De_1+\dots+\De_k \le s \bigr\} \,,
\end{equation}
where the empty product is included and set equal to\/ $\vac$ when\/ $\vac$ exists.
Assume that
\begin{equation}\label{RDe2}
\bigl[ {R_{\De_1}}_\lambda R_{\De_2} \bigr] \subset (\fil^{\De_1+\De_2-1} V)[\lambda] 
\qquad\text{for all}\quad \De_1 \,, \, \De_2 \,.
\end{equation}
Then the filtration \eqref{RDe3} of\/ $V$ is good, and it is very good if\/ $V$ is freely generated by\/ $R$.
\end{proposition}
\begin{proof}
It is straightforward to show that $\{\fil^s V\}$ is a good filtration, using \eqref{eq:qc}--\eqref{eq:wi};
see \cite{Li04,DSK05}. The filtration is very good if $V$ is freely generated by $R$, because the ordered monomials ${:} a_1 \cdots a_k {:}$ with $\De_1+\dots+\De_k=s$ span over $\mb F$ a complementary $\mb F[\partial]$-submodule to $\fil^{s-1} V$ in $\fil^s V$.
\end{proof}

\subsection{Universal enveloping VA of an LCA}\label{sec:4.15}

In analogy to Lie superalgebras,
given an LCA $R$, one constructs the corresponding 
universal enveloping (unital) VA $V(R)$, containing $R$ as an LCA subalgebra.
The VA $V(R)$ is defined by the following universal property:
for every LCA homomorphism $\phi\colon R\to V$
from the LCA $R$ to a VA $V$,
there is a unique VA homomorphism $\widehat{\phi}\colon V(R)\to V$
extending the map $\phi$.

It is known that $V(R)$ is freely generated by $R$ (see \cite{K96,BK03,DSK05}).
If we let $R=R_1$ in \eqref{RDe1}, then \eqref{RDe2} holds and Proposition \ref{prop:vgfil}
defines a very good filtration of $V(R)$, which is called
the \emph{canonical filtration}.
The corresponding associated graded PVA, given by Proposition \ref{prop:assgr},
is isomorphic to the (graded) universal PVA over the LCA $R$:
\begin{equation}\label{eq:grenv}
\gr V(R)\simeq\mc V(R)
\,.
\end{equation}

If $C\in R$ is such that $\partial C=0$, then $C$ is central in $R$. Hence, for any $c\in\mb F$, we have a VA ideal ${:}V(R)(C-c\vac){:} \subset V(R)$, and we can consider the quotient VA
\begin{equation}\label{eq:vcr2}
V^c(R) := V(R) / {:}V(R)(C-c\vac){:}
\,.
\end{equation}
The canonical filtration of $V(R)$ descends to a filtration of $V^c(R)$, which will also be called canonical.

\begin{lemma}\label{lem:uvg}
Let\/ $R$ be an LCA, $C\in R$ be such that\/ $\partial C=0$, and\/ $c\in\mb F$. Then:
\begin{enumerate}[(a)]
\item
The canonical filtration of the VA\/ $V^c(R)$ is very good.

\item
The associated graded PVA of\/ $V^c(R)$ is
\begin{equation*}
\gr V^c(R) \simeq\mc V^0(R) \simeq\mc V(\bar R) 
\,, 
\end{equation*}
where\/ $\bar R = R/\mb FC$ is the quotient LCA.
\end{enumerate}
\end{lemma}
\begin{proof}
Straightforward from the above discussion.
\end{proof}

\section{The chiral and classical operads}\label{sec:chcl}

In this section, we recall the definitions of the chiral operad $\mc P^\ch(V)$ 
and classical operad $\mc P^\cl(V)$ from \cite{BDSHK18}. 
The only new material is in Section \ref{sec:trivial}.

\subsection{The spaces $\mc O_{n}^{\star T}$}\label{sec:ostart}

Here and further, we will consider rational functions in the variables $z_1,z_2,\dots$ and use the shorthand notation $z_{ij}=z_i-z_j$.
For a fixed positive integer $n$, we denote by
$$
\mc O_{n}^T
=
\mb F[z_{ij}]_{1\leq i<j\leq n}
$$
the algebra of translation invariant polynomials in $n$ variables.
Let 
$$
\mc O_{n}^{\star T}
=\mb F[z_{ij}^{\pm 1}]_{1\leq i<j\leq n}
$$
be the localization of $\mc O_{n}^T$ with respect to the diagonals $z_i=z_j$ for $i\neq j$.
We set $\mc O_{0}^T = \mc O_{0}^{\star T}=\mb F$, and
note that $\mc O_{1}^T = \mc O_{1}^{\star T}=\mb F$.
%

We introduce an increasing $\mb Z_+$-filtration of $\mc O_{n}^{\star T}$ given by the number of divisors:
\begin{equation}\label{fil1}
\begin{split}
\fil^{-1} \mc O_{n}^{\star T} &= \{0\} \subset \fil^0 \mc O_{n}^{\star T} = \mc O_{n}^T \subset
\fil^1 \mc O_{n}^{\star T} = \sum_{i<j} \mc O_{n}^T [z_{ij}^{-1}] \subset \\
\cdots&\subset \fil^r \mc O_{n}^{\star T} = \sum \mc O_{n}^T [z_{i_1,j_1}^{-1},\dots, z_{i_r,j_r}^{-1}] \subset\cdots
\subset  \fil^{n-1} \mc O_{n}^{\star T} = \mc O_{n}^{\star T}.
\end{split}
\end{equation}
In other words, the elements of $\fil^r \mc O_{n}^{\star T}$ are sums of rational functions with
at most $r$ poles each, not counting multiplicities.
The fact that $\fil^{n-1} \mc O_{n}^{\star T} = \mc O_{n}^{\star T}$ follows from \cite[Lemma 8.4]{BDSHK18}.
%

\subsection{The operad $\mc P^\ch(V)$}\label{sec:wch}

Let $V=V_{\bar 0}\oplus V_{\bar 1}$ be a vector superspace endowed
with an even endomorphism $\partial$. For every $i=1,\dots,n$, we will denote by $\partial_i$ the action of $\partial$ on the $i$-th factor of the tensor power $V^{\otimes n}$:
\begin{equation}\label{ddi}
\partial_i v = v_1 \otimes\cdots\otimes \partial v_i  \otimes\cdots\otimes v_n \quad\text{for}\quad
v = v_1 \otimes\cdots\otimes v_n \in V^{\otimes n}.
\end{equation}
Recall the superspace $V_n$ given by \eqref{eq:1.3}.
The corresponding to $V$ operad $\mc P^\ch(V)=\{\mc P^\ch(n)\,|\,n\in\mb Z_{\geq0}\}$
is defined as follows.
The space of $n$-ary \emph{chiral operations} $\mc P^\ch(n)$ 
is the set of all linear maps \cite[(6.11)]{BDSHK18}
\begin{equation}\label{20160629:eq2-c}
\begin{split}
X\colon
V^{\otimes n}\otimes\mc O_{n}^{\star T}
&\to 
V_n
\,,\\
\vphantom{\Big(}
v_1 \otimes\dots\otimes v_n\otimes f(z_1,\dots,z_n)
&\mapsto
X_{\lambda_1,\dots,\lambda_n} (v_1\otimes\dots\otimes v_n \otimes f)
\\
&=
X_{\lambda_1,\dots,\lambda_n}^{z_1,\dots,z_n} (v_1\otimes\dots\otimes v_n \otimes f(z_1,\dots,z_n))
\,,
\end{split}
\end{equation}
satisfying the following two \emph{sesquilinearity} conditions:
\begin{align}\label{20160629:eq4a}
X_{\lambda_1,\dots,\lambda_n} ( v \otimes \partial_{z_i} f )
&= X_{\lambda_1,\dots,\lambda_n} ( (\partial_i+\la_i) v \otimes f ) \,, \\
\label{20160629:eq4b}
X_{\lambda_1,\dots,\lambda_n} (v \otimes z_{ij}f)
&= (\partial_{\lambda_j}-\partial_{\lambda_i}) X_{\lambda_1,\dots,\lambda_n} (v \otimes f) \,.
\end{align}
For example, we have:
\begin{align}\label{pch0}
\mc P^\ch(0) &= \Hom_{\mb F} (\mb F, V/\langle\dd\rangle) \cong V/\dd V, \\
\label{pch1}
\mc P^\ch(1) &= \Hom_{\mb F[\dd]} (V, V[\la_0]/\langle\dd+\la_0\rangle) \cong \End_{\mb F[\dd]} (V).
\end{align}
The $\mb Z/2\mb Z$-grading of the superspace $\mc P^\ch(n)$ is induced 
by that of the vector superspace $V$, where $\mc O_{n}^{\star T}$ and all
variables $\lambda_i$ are considered even.

The symmetric group $S_{n}$ acts on the right on $\mc P^\ch(n)$ by permuting 
simultaneously the inputs 
$v_1,\dots,v_n$ of $X$, the variables $\la_1,\dots,\la_n$, 
and the corresponding variables $z_1,\dots,z_n$ in $f$. 
Explicitly, for $X\in \mc P^\ch(n)$ and $\sigma \in S_{n}$, we have
\begin{equation}\label{eq:symch}
\begin{split}
(X^\sigma &)^{z_1,\dots,z_n}_{\lambda_1,\dots,\lambda_n}(v_1\otimes\dots\otimes v_n 
\otimes f(z_1,\dots,z_n))
\\
&=
\epsilon_v(\sigma)
X^{z_1,\dots,z_n}_{\lambda_{i_1},\dots,\lambda_{i_n}}(v_{i_1} \otimes\dots\otimes 
v_{i_n} \otimes f(z_{i_1},\dots,z_{i_n})),
\end{split}
\end{equation}
where $i_s=\sigma^{-1}(s)$ and the sign $\epsilon_v(\sigma)$ is given by 
\begin{equation}\label{eq:operad14}
\epsilon_v(\sigma)
=
\prod_{i<j\,|\,\sigma(i)>\sigma(j)}(-1)^{p(v_i)p(v_j)}
\,.
\end{equation}
One can also define compositions of chiral operations, 
turning $\mc P^\ch(V)$ into an operad (see \cite[(6.25)]{BDSHK18}). 

\subsection{Filtration of $\mc P^\ch(V)$}\label{sec:pchfil}

Now suppose that $V$ is equipped with an increasing $\mb Z_+$-filtration
of $\mb F[\partial]$-submodules
\begin{equation}\label{eq:last3}
\fil^{-1}V=\{0\}
\,\subset\,
\fil^0V
\,\subset\,
\fil^1V
\,\subset\,
\fil^2V
\,\subset\,
\cdots
\,\subset\,
V
\,.
\end{equation}
Since $\mc O_{n}^{\star T}$ is also filtered by \eqref{fil1}, we obtain
an increasing $\mb Z_+$-filtration on the tensor products
$$
\fil^s\big(V^{\otimes n}\otimes \mc O_{n}^{\star T}\big)
=
\sum_{s_1+\dots+s_{n}+p\leq s}
\fil^{s_1}V
\otimes
\cdots
\otimes
\fil^{s_n}V
\otimes
\fil^{p}\mc O_{n}^{\star T}
\,\,\text{ if } s\geq0
\,,
$$
and $\fil^s(V^{\otimes n}\otimes \mc O_{n}^{\star T})=\{0\}$ if $s<0$.

This induces a decreasing $\mb Z$-filtration of $\mc P^\ch(n)$, where $\fil^r \mc P^\ch(n)$ for $r\in\mb Z$ is defined 
as the set of all elements $X$ such that
\begin{equation}\label{fil4-ref}
X\bigl( 
\fil^s(V^{\otimes n} \otimes \mc O_{n}^{\star T})
\bigr)
\subset
(\fil^{s-r}V)[\lambda_1,\dots,\lambda_n]/\langle\partial+\lambda_1+\dots+\lambda_n\rangle
\,\,,\,\,\,\,
s\in\mb Z
\,.
\end{equation}
The composition maps in $\mc P^\ch(V)$ and the actions of the symmetric groups
are compatible with the filtration \eqref{fil4-ref} (see \cite[Proposition 8.9]{BDSHK18}); hence, $\mc P^\ch(V)$ is a \emph{filtered operad} as in \cite[Section 3.1]{BDSHK18}.
Then, as usual, the associated graded spaces are defined by
\begin{equation}\label{grr}
\gr^r \mc P^\ch(n)
=
\fil^r \mc P^\ch(n) / \fil^{r+1} \mc P^\ch(n)
\,, \qquad r\in\mb Z \,,
\end{equation}
and we obtain that $\gr \mc P^\ch(V)$ is a \emph{graded operad} (see \cite[Section 3.1]{BDSHK18}).

\begin{lemma}\label{lem:filsep}
If the filtration \eqref{eq:last3} of\/ $V$ is exhaustive $($see \eqref{eq:exfil}$)$, then the filtration of\/ $\mc P^\ch(n)$ is \emph{separated}, see \eqref{eq:sepfil-def}.
\end{lemma}
\begin{proof}
Note that $\fil^{s+n-1}(V^{\otimes n} \otimes \mc O_{n}^{\star T}) \supset \fil^s(V^{\otimes n}) \otimes \mc O_{n}^{\star T}$, because $\fil^{n-1} \mc O_{n}^{\star T} = \mc O_{n}^{\star T}$.
Since $\fil^{-1}V=\{0\}$, we see from \eqref{fil4-ref} that $X(\fil^s(V^{\otimes n}) \otimes \mc O_{n}^{\star T}) = 0$
for all $X\in \fil^{s+n} \mc P^\ch(n)$. If the filtration of $V$ is exhaustive, then the induced filtration of $V^{\otimes n}$ is exhaustive. Hence, $X=0$ for every $X\in\bigcap_{r\in\mb Z} \fil^r \mc P^\ch(n)$.
\end{proof}


\subsection{$n$-graphs}\label{sec:6a.1}

For a positive integer $n$, we define an $n$-\emph{graph}
as a graph $\Gamma$ with $n$ vertices labeled by $1,\dots,n$
and an arbitrary collection $E(\Gamma)$ of oriented edges.
We denote 
by $\mc G(n)$ the collection of all $n$-graphs
without tadpoles,
and by $\mc G_0(n)$ the collection of all \emph{acyclic} $n$-graphs,
i.e., $n$-graphs that have 
no cycles (including tadpoles and multiple edges).
For example, $\mc G_0(1)$ consists of the graph with a single vertex labeled $1$ and no edges,
and $\mc G_0(2)$ consists of three graphs:
\begin{equation}\label{eq:2-graphs}
\begin{array}{l}
\begin{tikzpicture}
\draw (0.5,1) circle [radius=0.07];
\node at (0.5,0.7) {1};
\draw (1.5,1) circle [radius=0.07];
\node at (1.5,0.7) {2};
\node at (2,0.85) {,};
\draw (3.7,1) circle [radius=0.07];
\node at (3.7,0.7) {1};
\draw (4.7,1) circle [radius=0.07];
\node at (4.7,0.7) {2};
\draw[->] (3.8,1) -- (4.6,1);
\node at (5,0.85) {,};
\draw (7,1) circle [radius=0.07];
\node at (7,0.7) {1};
\draw (8,1) circle [radius=0.07];
\node at (8,0.7) {2};
\draw[<-] (7.1,1) --(7.9,1);
\node at (8.5,0.85) {.};
\end{tikzpicture}
\\
E(\Gamma)=\emptyset
\,\,\,\,\,,\,\,\,\,\qquad
E(\Gamma)\!=\!\{1\!\to\!2\}
\,\,,\,\,\qquad
E(\Gamma)\!=\!\{2\!\to\!1\}
\end{array}
\end{equation}
By convention, we also let $\mc G_0(0)=\mc G(0)=\{\emptyset\}$ be the set consisting of a single element (the empty graph with $0$ vertices).

A graph $L$ will be called a \emph{line} if its set of edges is of the form $\{i_1\to i_2,\,i_2\to i_3,\dots,\,i_{n-1}\to i_n\}$ where $\{i_1,\dots,i_n\}$ is a permutation of $\{1,\dots,n\}$:
\begin{equation}\label{eq:line}
\begin{tikzpicture}
\node at (-0.2,1) {$L=$};
\draw (0.5,1) circle [radius=0.07];
\node at (0.5,0.6) {$i_1$};
\draw[->] (0.6,1) -- (0.9,1);
\draw (1,1) circle [radius=0.07];
\node at (1,0.6) {$i_2$};
\draw[->] (1.1,1) -- (1.4,1);
\node at (1.7,1) {$\cdots$};
\draw[->] (1.9,1) -- (2.2,1);
\draw (2.3,1) circle [radius=0.07];
\node at (2.3,0.6) {$i_n$};
\node at  (2.9,0.8) {.};
\end{tikzpicture}
\end{equation}
An \emph{oriented cycle} $C$ in a graph $\Gamma$ is, by definition, 
a collection of edges of $\Gamma$ forming a closed sequence (possibly with self intersections):
\begin{equation}\label{20170823:eq4a}
C=\{i_1\to i_2,\,i_2\to i_3,\dots,\,i_{s-1}\to i_s,\,i_s\to i_1\}\subset E(\Gamma)
\,.
\end{equation}

There is a natural (left) action of the symmetric group $S_n$
on the set $\mc G(n)$ of $n$-graphs, which preserves the subset $\mc G_0(n)$ of acyclic graphs.
Given $\Gamma\in\mc G(n)$ and $\sigma\in S_n$,
we let $\sigma(\Gamma)$ be the same graph as $\Gamma$,
but with the vertex that was labeled $1$ relabeled as $\sigma(1)$,
the vertex $2$ relabeled as $\sigma(2)$,
and so on up to the vertex $n$ now relabeled as $\sigma(n)$.
For example, if $L_0$ is the line with edges $\{1\to2,2\to3,\dots,n-1\to n\}$ and $\sigma\in S_n$,
then $\sigma(L_0)=L$ is the line \eqref{eq:line} where $i_k=\sigma(k)$.

Let $\mc L(n) \subset \mc G(n)$ be the set of graphs that are disjoint unions of lines.
Consider the vector space $\mb F\mc G(n)$ with the set $\mc G(n)$ as a basis over $\mb F$.
The subspace $\mc R(n)\subset\mb F\mc G(n)$ of \emph{cycle relations} is spanned by the following elements:
\begin{enumerate}[(i)]
\item
all graphs $\Gamma\in\mc G(n)$ containing a cycle;
\item
all linear combinations of the form
$\sum_{e\in C}\Gamma\backslash e$,
for all $\Gamma\in\mc G(n)$
and all oriented cycles $C\subset E(\Gamma)$,
where $\Gamma\backslash e\in\mc G(n)$ is the graph obtained from $\Gamma$ 
by removing the edge $e$ and keeping the same set of vertices.
\end{enumerate}
Note that if we reverse an arrow in a graph $\Gamma\in\mc G(n)$,
we obtain, modulo cycle relations, the element $-\Gamma\in\mb F\mc G(n)$.
\begin{lemma}[\cite{BDSHK19}]\label{prop:basis}
The set\/ $\mc L(n)$ spans the space\/ $\mb F\mc G(n)$ modulo the cycle relations.
\end{lemma}
In \cite{BDSHK19}, we also give an explicit basis of the space $\mb F\mc G(n)/\mc R(n)$, but it will not be needed here.

\subsection{The operad $\mc P^\cl(V)$}\label{sec:6.2}

Now we will recall the definition 
of the classical operad $\mc P^\cl(V)=\{\mc P^\cl(n)\,|\,n\in\mb Z_{\geq0}\}$ 
from \cite[Section 10]{BDSHK18}.

As before, let $V=V_{\bar 0}\oplus V_{\bar 1}$ be a vector superspace endowed
with an even endomorphism $\partial$.
Set $\mc P^\cl(0) = V/\partial V$, and for $n\ge1$,
define $\mc P^\cl(n)$ as the vector superspace (with the pointwise addition and scalar multiplication)
of all maps 
\begin{align}\label{20170614:eq4}
Y\colon
\mc G(n)\times V^{\otimes n}
&\to
V_n
\,, \\ \label{20170614:eq5}
(\Gamma , v) &\mapsto
Y^{\Gamma}_{\lambda_1,\dots,\lambda_n}(v) \,,
\end{align}
which depend linearly on 
$v 
\in V^{\otimes n}$,
and satisfy
the {cycle relations}
and {sesquilinearity conditions} described below.
The $\mb Z/2\mb Z$-grading of the superspace $\mc P^\cl(n)$ is induced 
by that of the vector superspace $V$,
by letting $\mc G(n)$ and the variables $\lambda_i$ be even.

The \emph{cycle relations} in $\mc P^\cl(n)$ state that if an $n$-graph $\Gamma\in\mc G(n)$ contains an oriented cycle
$C\subset E(\Gamma)$, then:
\begin{equation}\label{eq:cycle}
Y^{\Gamma} =0 
\,,\qquad
\sum_{e\in C} Y^{\Gamma\backslash e} =0 
\,.
\end{equation}
In particular, applying the second cycle relation \eqref{eq:cycle} for an oriented cycle of length $2$,
we see that changing the orientation of a single edge of $\Gamma\in\mc G(n)$
amounts to a change of sign of $Y^{\Gamma}$.

To write the {sesquilinearity conditions}, let us first introduce some notation.
For a graph $G$ with a set of vertices labeled by a subset $I\subset\{1,\dots,n\}$,
we let
\begin{equation}\label{com3}
\la_G = \sum_{i\in I} \la_i \,, 
\qquad \dd_G = \sum_{i\in I} \dd_i \,, 
\end{equation}
where as before $\dd_i$ denotes the action of $\dd$ on the $i$-th factor in $V^{\otimes n}$
(see \eqref{ddi}).
Then for every connected component $G$ of $\Gamma\in\mc G(n)$ with a set of vertices $I$,
we have two \emph{sesquilinearity conditions}:
\begin{align}\label{eq:sesq1}
(\partial_{\lambda_j} - \partial_{\lambda_i})
Y^{\Gamma}_{\lambda_1,\dots,\lambda_n}(v) &= 0
\quad\text{for all}\quad
i,j\in I \,,
\\ \label{eq:sesq2}
Y^{\Gamma}_{\lambda_1,\dots,\lambda_n}
\bigl( (\partial_G+\la_G)v \bigr)
&=0 \,, \qquad v\in V^{\otimes n} \,.
\end{align}
The first condition \eqref{eq:sesq1} means that the polynomial 
$Y^{\Gamma}_{\lambda_1,\dots,\lambda_n}(v)$
is a function of the variables $\lambda_{\Gamma_\alpha}$, where 
the $\Gamma_\alpha$'s are the connected components of $\Gamma$, 
and not of the variables $\lambda_1,\dots,\lambda_n$ separately.

We have a natural right action of the symmetric group $S_n$ on $\mc P^\cl(n)$
by (parity preserving) linear maps:
\begin{equation}\label{20170615:eq1}
(Y^\sigma)^\Gamma_{\lambda_1,\dots,\lambda_n}(v_1\otimes\dots\otimes v_n)
= \epsilon_v(\sigma)
Y^{\sigma(\Gamma)}_{\lambda_{i_1},\dots,\lambda_{i_n}}(v_{i_1} \otimes\dots\otimes v_{i_n} )
\,,
\end{equation}
where $i_s=\sigma^{-1}(s)$, the sign $\epsilon_v(\sigma)$ is given by \eqref{eq:operad14},
and $\sigma(\Gamma)$ is defined in Section \ref{sec:6a.1}.
In \cite[(10.11)]{BDSHK18}, we also defined compositions of maps in $\mc P^\cl(V)$, turning it into an operad. 

\begin{remark}\label{rem:lines}
Due to \eqref{eq:cycle} and Lemma \ref{prop:basis}, any classical operation $Y\in \mc P^\cl(n)$ is uniquely determined by $Y^\Ga$ for graphs $\Ga\in\mc L(n)$ that are disjoint unions of lines.
\end{remark}


Suppose now that $V=\bigoplus_{t\in\mb Z} \gr^t V$ is graded by $\mb F[\partial]$-sub\-modules,
and consider the induced grading of the tensor powers $V^{\otimes n}$:
$$
\gr^t V^{\otimes n}
=
\sum_{t_1+\dots+t_{n}=t}
\gr^{t_1}V
\otimes
\cdots
\otimes
\gr^{t_n}V
\,.
$$
Then $\mc P^\cl(V)$
has a grading defined as follows:
$Y\in\gr^r \mc P^\cl(n)$ if
\begin{equation}\label{pclgrading}
Y^\Gamma_{\lambda_1,\dots,\lambda_n}(\gr^t V^{\otimes n})
\,\subset\,
(\gr^{s+t-r}V)[\lambda_1,\dots,\lambda_n]/
\langle\partial+\lambda_1+\dots+\lambda_n\rangle
\end{equation}
for every graph $\Gamma\in\mc G(n)$ with $s$ edges
(see \cite[Remark 10.2]{BDSHK18}).
%

\subsection{The isomorphism from $\gr \mc P^\ch(V)$ to $\mc P^\cl(\gr V)$}\label{sec:grpchpcl}

For a graph $\Gamma\in\mc G(n)$ with a set of edges $E(\Ga)$, we introduce the function
\begin{equation}\label{gr1}
p_\Ga = p_\Ga(z_1,\dots,z_n) = \prod_{(i\to j) \in E(\Ga)} z_{ij}^{-1} \,, 
\qquad z_{ij}=z_i-z_j \,.
\end{equation}
Note that $p_\Ga \in \fil^s \mc O_{n}^{\star T}$ if $\Ga$ has $s$ edges. 
\begin{lemma}[{\cite[Lemma 8.3]{BDSHK18}}]\label{lem:pgamma}
The space\/ $\fil^s \mc O_{n}^{\star T}$ is generated as an\/ $\mc O_{n}^T$-module by all partial derivatives of the products\/ $p_\Ga$, where\/ $\Ga$ runs over the set of acyclic graphs with\/ $n$ vertices and at most\/ $s$ edges.
\end{lemma}
\begin{corollary}\label{cor:pgamma}
A chiral operation\/ $Y\in \mc P^\ch(n)$ is uniquely determined by its values\/ $Y(v \otimes p_\Ga)$, where\/ $v\in V^{\otimes n}$ and\/ $\Ga$ runs over the set of connected lines with\/ $n$ vertices.
\end{corollary}
\begin{proof}
This follows from Remark \ref{rem:lines}, Lemma \ref{lem:pgamma}, and the sesquilinearity conditions \eqref{20160629:eq4a}, \eqref{20160629:eq4b}.
\end{proof}

Let $V$ be filtered by $\mb F[\partial]$-submodules as in \eqref{eq:last3}. 
Then we have the filtered operad $\mc P^\ch(V)$ associated to $V$ 
and the graded operad $\mc P^\cl(\gr V)$ associated to the graded superspace $\gr V$. 
These two operads are related as follows \cite[Section 8]{BDSHK18}.

Let $X\in\fil^r \mc P^\ch(V)(n)$ and $\Gamma\in\mc G(n)$ be a graph with $s$ edges. 
Then for every $v\in\fil^t V^{\otimes n}$, we have 
$v\otimes p_\Ga \in \fil^{s+t} ( V^{\otimes n} \otimes \mc O_{n}^{\star T} )$ 
and, by \eqref{fil4-ref},
\begin{equation}\label{gr2}
X_{\la_1,\dots,\la_n} (v\otimes p_\Ga) 
\in (\fil^{s+t-r} V)[\lambda_1,\dots,\lambda_n]
/\langle\partial+\lambda_1+\dots+\lambda_n\rangle \,.
\end{equation}
We define $Y\in\gr^r \mc P^\cl(\gr V)(n)$ by:
\begin{equation}\label{gr3}
\begin{split}
Y^\Ga_{\la_1,\dots,\la_n} \bigl( v&+\fil^{t-1} V^{\otimes n} \bigr)
= X_{\la_1,\dots,\la_n} (v\otimes p_\Ga) \\
&+ (\fil^{s+t-r-1} V)[\lambda_1,\dots,\lambda_n]/\langle\partial+\lambda_1+\dots+\lambda_n\rangle \\
&\in (\gr^{s+t-r} V)[\lambda_1,\dots,\lambda_n]/\langle\partial+\lambda_1+\dots+\lambda_n\rangle \,.
\end{split}
\end{equation}
Clearly, the right-hand side depends only 
on the image $\bar v = v+\fil^{t-1} V^{\otimes n} \in \gr^t V^{\otimes n}$ 
and not on the choice of representative $v\in\fil^t V^{\otimes n}$. 
We write \eqref{gr3} simply as
\begin{equation}\label{gr4}
Y^\Ga_{\la_1,\dots,\la_n} (\bar v)
= \overline{ X_{\la_1,\dots,\la_n} (v\otimes p_\Ga) } \,.
\end{equation}
The fact that $Y\in\gr^r \mc P^\cl(\gr V)(n)$ was proved in \cite[Corollary 8.8]{BDSHK18}. 

If $X\in\fil^{r+1} \mc P^\ch(V)(n)$, then the right-hand side of \eqref{gr3} (or \eqref{gr4}) vanishes.
Thus, \eqref{gr3} defines a map
\begin{equation}\label{eq:map}
\gr^r \mc P^\ch(V)(n)\,\to\,\gr^r \mc P^\cl(\gr V)(n)
\,,\qquad 
\bar X=X+\fil^{r+1} \mapsto Y
\,.
\end{equation}
\begin{theorem}[{\cite{BDSHK18,BDSHK19}}] \label{thm:chcl}
The map \eqref{eq:map} is an injective homomorphism of graded operads. 
If\/ $V\simeq\gr V$ as\/ $\mb F[\partial]$-modules, then \eqref{eq:map} is an isomorphism.
\end{theorem}
%

\subsection{The case of trivial filtration and the operad $\mc Chom(V)$}\label{sec:chom}

In this subsection, $V$ will be an $\mb F[\partial]$-module with the trivial filtration and grading:
\begin{equation}\label{eq:trfg1}
\fil^{-1} V = \{ 0 \} \subset \fil^0 V = V \,, \qquad \gr V = \gr^0 V = V \,.
\end{equation}
By \eqref{fil4-ref}, a chiral operation $Y\in \mc P^\ch(n)$ lies in $\fil^r \mc P^\ch(n)$ if and only if
\cite[(8.4)]{BDSHK18}:
\begin{equation}\label{eq:trfg2}
Y\bigl(V^{\otimes n} \otimes \mc \fil^{r-1} \mc O_n^{\star T}\bigr) = 0 \,.
\end{equation}
In this case, the filtration of each $\mc P^\ch(n)$ is finite:
\begin{equation}\label{eq:trfg3}
\mc P^\ch(n) = \fil^0 \mc P^\ch(n) \supset \fil^1 \mc P^\ch(n) \supset\cdots\supset \fil^{n-1} \mc P^\ch(n) \supset \fil^{n} \mc P^\ch(n) = \{0\} \,.
\end{equation}

Similarly, by \eqref{pclgrading}, a classical operation $Y\in \mc P^\cl(n)$ is in $\gr^r \mc P^\cl(n)$ if and only if
$Y^\Ga = 0$ for every graph $\Ga$ with $n$ vertices and number of edges $\ne r$
(see \cite[(10.9)]{BDSHK18}). By Remark \ref{rem:lines}, we can restrict to graphs $\Ga$ that are disjoint unions of lines; such graphs have $\le n-1$ edges. Hence, the grading of $\mc P^\cl(n)$ has the form
\begin{equation}\label{eq:grpcl}
\mc P^\cl(n) = \bigoplus_{r=0}^{n-1} \gr^r \mc P^\cl(n) \,.
\end{equation}
By Theorem \ref{thm:chcl}, we have an isomorphism 
of graded operads $\gr \mc P^\ch(V) \simeq \mc P^\cl(V)$.

Notice that $\gr^0 \mc P^\cl(V)$ is a suboperad of $\mc P^\cl(V)$, which is also denoted $\mc Chom(V)$ and is described explicitly in \cite[Section 5]{BDSHK18}. We have a surjective morphism of operads
(with kernel $\fil^1 \mc P^\ch(V)$)
given by evaluation at the function $1$:
\begin{equation}\label{eq:ch-chom}
\mc P^\ch(V) \to 
\mc Chom(V) \,, \quad
Y \mapsto \bar Y \,, \quad 
\bar Y(v) = Y(v\otimes1) 
\,\,\,\,\,\,
\text{ for } Y\in \mc P^\ch(n) \,,\,\, v\in V^{\otimes n}
\,.
\end{equation}
Let us review the definition of the operad $\mc Chom(V) = \{\mc Chom(n) \,|\, n\in\mb Z_{\ge0} \}$.
Elements of $\mc Chom(n)$ are called $n$-ary \emph{conformal operations}. These are linear maps
\begin{equation}\label{eq:chom}
\begin{split}
Y\colon
V^{\otimes n}
&\to
V_n
\,,\\
\vphantom{\Big(}
v_1 \otimes\dots\otimes v_n &
\mapsto
Y_{\lambda_1,\dots,\lambda_n} (v_1\otimes\dots\otimes v_n)
\,,
\end{split}
\end{equation}
satisfying the \emph{sesquilinearity} conditions:
\begin{equation}\label{eq:chomses}
Y_{\lambda_1,\dots,\lambda_n} ( (\partial_i+\la_i) v) = 0 \,, 
\qquad 1\le i \le n\,, \;\; v\in V^{\otimes n} \,,
\end{equation}
where, as before, $V_n$ is defined by \eqref{eq:1.3} and
$\dd_i$ denotes the action of $\dd$ on the $i$-th factor in $V^{\otimes n}$.
The symmetric group $S_{n}$ acts on the right on $\mc Chom(n)$ by permuting 
simultaneously the inputs 
$v_1,\dots,v_n$ and the variables $\la_1,\dots,\la_n$. 
Explicitly, for $Y\in \mc Chom(n)$ and $\sigma \in S_{n}$, we have
\begin{equation}\label{eq:symch-chom}
(Y^\sigma )_{\lambda_1,\dots,\lambda_n}(v_1\otimes\dots\otimes v_n)
=
\epsilon_v(\sigma) \,
Y_{\lambda_{i_1},\dots,\lambda_{i_n}}(v_{i_1} \otimes\dots\otimes v_{i_n}) \,,
\end{equation}
where $i_s=\sigma^{-1}(s)$ and the sign $\epsilon_v(\sigma)$ is given by 
\eqref{eq:operad14}.
The compositions in the operad $\mc Chom(V)$ are given by \cite[(5.6)]{BDSHK18}.
Note that $\mc Chom(0) = V/\partial V$.
\subsection{The case $\partial=0$}\label{sec:trivial}

Finally, let us consider the case when the action of $\mb F[\partial]$ on $V$ is trivial.
We take the trivial increasing filtration of $V$ given by \eqref{eq:trfg1}.

\begin{theorem}\label{thm:triv}
Let\/ $V$ be a vector superspace equipped with the trivial\/ $\mb F[\partial]$-action and filtration.
Then, for every\/ $n\ge1$, the filtration of\/ $\mc P^\ch(n)$ has the form
$$
\mc P^\ch(n) = \fil^{n-1} \mc P^\ch(n) \supset \fil^{n} \mc P^\ch(n) = \{0\} \,.
$$
Hence,
$$
\mc P^\ch(n) = \gr^{n-1} \mc P^\ch(n) = \gr \mc P^\ch(n) \simeq \mc P^\cl(n) = \gr^{n-1} \mc P^\cl(n) \,.
$$
\end{theorem}
\begin{proof}
To prove the first claim, we need to check that $X(v \otimes f) = 0$ for every
$X\in \mc P^\ch(n)$, $v\in V^{\otimes n}$ and $f\in\fil^{n-2} \mc O_n^{\star T}$.
By the sesquilinearity \eqref{20160629:eq4a}, \eqref{20160629:eq4b}
and Lemma \ref{lem:pgamma}, we can assume that $f=p_\Ga$
where $\Ga$ is an acyclic graph with $n$ vertices and $\le n-2$ edges.
Let $G$ be a connected component of $\Ga$. Then the set $I$ of vertices of $G$ is a proper subset of $\{1,\dots,n\}$. Since, by translation invariance, 
$\sum_{i\in I} \partial_{z_i} p_\Ga = 0$, the sesquilinearity \eqref{20160629:eq4a} implies
$$
\la_G X_{\lambda_1,\dots,\lambda_n}(v \otimes f) = 0 \in V[\la_1,\dots,\la_n] / \langle \la_1+\cdots+\la_n \rangle \,.
$$
Hence, $X(v \otimes f) = 0$.

A similar argument, using \eqref{eq:sesq2}, proves that $\mc P^\cl(n) = \gr^{n-1} \mc P^\cl(n)$.
The second claim of the theorem then follows immediately from the first and from Theorem \ref{thm:chcl}.
\end{proof}

In the case when $V=\mb F$ is a $1$-dimensional even vector space with a trivial action of $\partial$, the operad 
$\mc P^\ch(V)$ is isomorphic to the operad $\mc Lie$ (see \cite[3.1.5]{BD04} and \cite[Theorem 5.2]{BDSHK19}).
In order to state the general result, let us recall the operad $\mc Hom(V)$, defined by 
\begin{equation}\label{eq:hom}
\mc Hom(V)(n) = \Hom(V^{\otimes n},V) \,, \qquad n\ge 0\,.
\end{equation}
Also recall that the tensor product of two operads $\mc P_1$ and $\mc P_2$ is given by
$$
(\mc P_1 \otimes \mc P_2)(n) = \mc P_1(n) \otimes \mc P_2(n) \,, \qquad n\ge0 \,,
$$
with component-wise compositions and actions of the symmetric groups.

\begin{theorem}\label{thm:triv2}
Let\/ $V$ be a vector superspace equipped with the trivial action of\/ $\mb F[\partial]$.
Then
\begin{equation*}
\mc P^\ch(V) \simeq \mc P^\cl(V) \simeq \mc Hom(V) \otimes \mc Lie \,.
\end{equation*}
\end{theorem}
\begin{proof}
We saw in the proof of Theorem \ref{thm:triv} that for $Y\in \mc P^\cl(V)(n)$ we have $Y^\Ga=0$ unless the graph $\Ga$ is connected. For a connected graph $\Ga$, the sesquilinearity \eqref{eq:sesq1} implies that 
$$
Y^\Ga(v) \in V[\la_1+\dots+\la_n] / \langle\la_1+\dots+\la_n\rangle \simeq V \,, \qquad v\in V^{\otimes n} \,.
$$
Hence, we can view $Y$ as a collection of linear maps $Y^\Ga\colon V^{\otimes n} \to V$ satisfying the cycle relations \eqref{eq:cycle}. In the case $V=\mb F$, this is a collection of scalars $y^\Ga \in \mb F$.
Thus, we have a morphism of operads
\begin{equation}\label{eq:bojko1}
\mc Hom(V) \otimes \mc P^\cl(\mb F) \to \mc P^\cl(V) \,,
\end{equation}
which sends $\phi\otimes y$ to $Y$ given by $Y^\Ga(v) = y^\Ga \phi(v)$ for $v\in V^{\otimes n}$.
Let us check that the map \eqref{eq:bojko1} is an isomorphism.
By definition we have 
\begin{align*}
\mc Hom(V)(n)&=\Hom(V^{\otimes n},V)
\,, \\
\mc P^\cl(\mb F)(n)
&=
\Hom\bigl((\mb F\mc G(n)/\mc R(n))\otimes \mb F^{\otimes n},\mb F\bigr)
\,, \\
\mc P^\cl(V)(n)
&=
\Hom\bigl((\mb F\mc G(n)/\mc R(n))\otimes V^{\otimes n},V\bigr)
\,.
\end{align*}
The fact that \eqref{eq:bojko1} is an isomorphism
follows from the linear algebra isomorphisms
$$
\Hom(A_1,B_1)\otimes\Hom(A_2,B_2)
\simeq
\Hom(A_1\otimes A_2,B_1\otimes B_2)
\,,\qquad
\mb F\otimes A\simeq A
\,.
$$
\end{proof}

\section{LCA, VA and PVA cohomology}\label{sec:vapva}

In this section, we review the definitions of the cohomology of Lie conformal algebras, vertex algebras, and Poisson vertex algebras, following \cite{BDSHK18}.

\subsection{Lie superalgebra associated to an operad}\label{sec:lieoper}

We start by reviewing a general construction, which goes back to \cite{Tam02} (see also \cite{BDSHK18}).
For a linear superoperad $\mc P$, we define
$$
W_{\mc P} = \bigoplus_{k=-1}^\infty W^k_{\mc P} \,, \qquad
W^k_{\mc P} := \mc P(k+1)^{S_{k+1}} \,,
$$ 
where $\mc P(k+1)^{S_{k+1}}$ denotes the subspace of $\mc P(k+1)$ consisting of elements invariant under the action of the symmetric group. One defines a Lie superalgebra bracket on $W_{\mc P}$ only in terms of the compositions and symmetric group actions in the operad $\mc P$ (see \cite[Theorem 3.4]{BDSHK18}).
We thus get a functor $\mc P \mapsto W_{\mc P}$ from the category of linear superoperads to the category of $\mb Z$-graded Lie superalgebras.

\begin{remark}\label{Wfilgr}
\begin{enumerate}[(a)]
\item 
If $\mc P$ is filtered operad with a filtration $\{\fil^r \mc P\}_{r\in\mb Z}$, 
then $W_{\mc P}$ is a filtered Lie superalgebra
with a filtration $\fil^r W^k_{\mc P} := (\fil^r \mc P(k+1))^{S_{k+1}}$, so that
$$
[\fil^r W^k_{\mc P}, \fil^s W^\ell_{\mc P}] \subset \fil^{r+s} W^{k+\ell}_{\mc P}
\,, \qquad r,s\in\mb Z \,, \; k,\ell \ge -1 \,.
$$
\item
If $\mc P$ is graded operad with a grading $\{\gr^r \mc P\}_{r\in\mb Z}$, 
then $W_{\mc P}$ is a graded Lie superalgebra
with a grading $\gr^r W^k_{\mc P} := (\gr^r \mc P(k+1))^{S_{k+1}}$, so that
$$
[\gr^r W^k_{\mc P}, \gr^s W^\ell_{\mc P}] \subset \gr^{r+s} W^{k+\ell}_{\mc P}
\,, \qquad r,s\in\mb Z \,, \; k,\ell \ge -1 \,.
$$

\end{enumerate}
\end{remark}

If $X\in W^1_{\mc P}$ is an odd element such that $[X,X]=0$, then $d=\ad_X$ is a differential on $W_{\mc P}$, i.e., $d^2=0$ and $d$ is an odd endomorphism of degree $1$. We obtain a cochain complex
$$
C_{\mc P} = \bigoplus_{n=0}^\infty C^n_{\mc P} \,, \qquad
C^n_{\mc P} := W^{n-1}_{\mc P} = \mc P(n)^{S_n} \,, \qquad
d\colon C^n_{\mc P} \to C^{n+1}_{\mc P} \,.
$$
There are several examples of operads $\mc P$ that give rise to interesting algebraic structures and  corresponding cohomology theories (see \cite{BDSHK18}).

First, consider the operad $\mc Hom(V)$ given by \eqref{eq:hom}, for a fixed vector superspace $V$.
Then odd elements $X\in W^1_{\mc Hom(V)}$ such that $[X,X]=0$ correspond bijectively to Lie superalgebra structures on $\Pi V$, the superspace $V$ with opposite parity $\bar p=1-p$, where $p$ denotes the parity of $V$. 
Consequently, an odd element $X\in W^1_{\mc Hom(\Pi V)}$ with $[X,X]=0$ is the same as a Lie superalgebra bracket on $V$, given explicitly by
$$
[a,b] = (-1)^{p(a)} X(a \otimes b) \,, \qquad a,b \in V \,.
$$
Indeed, one checks that the symmetry of $X$ with respect to $\bar p$ corresponds to the skewsymmetry of $[\cdot,\cdot]$ with respect to $p$, and the Jacobi identity for $[\cdot,\cdot]$ corresponds to the identity $[X,X]=0$ (see \cite{NR67,DSK13}).
Taking the complex $C_{\mc Hom(\Pi V)}$, we get the Chevalley--Eilenberg
cohomology complex of the Lie superalgebra $V$ with coefficients in the adjoint representation. 

The cohomology with coefficients in a module can be obtained by a reduction procedure as follows.
If $M$ is a module over the Lie superalgebra $V$, then $V \oplus M$ is a Lie superalgebra such that $V$ is a subalgebra, $[M,M]=0$ and $[a,m] = am$ for $a\in V$, $m\in M$. For each $n\ge0$, consider the subspaces
$$
U^n := \Hom(\Pi (V \oplus M)^{\otimes n}, \Pi M) 
\subset C^n_{\mc Hom(\Pi V \oplus \Pi M)}
$$
and
$$
K^n := \{ Y \in U^n \,|\, Y((\Pi V)^{\otimes n}) = 0 \} 
\subset U^n \,.
$$
Using the restriction map, we get a short exact sequence
$$
0 \to K^n \to U^n \to \Hom((\Pi V)^{\otimes n}, \Pi M) \to 0 \,.
$$
It is easy to check that $dU^n\subset U^{n+1}$ and $dK^n\subset K^{n+1}$.
Then the cohomology complex of $V$ with coefficients in $M$ is defined as the subquotient 
of the complex $C_{\mc Hom(\Pi V \oplus\Pi M)}$, which in degree $n$ is $U^n/K^n$.

\subsection{LCA cohomology}\label{sec:lca-coh}

Now let
$V$ be a vector superspace (with parity $p$), which is equipped with an $\mb F[\partial]$-module structure, where $\partial$ is an even endomorphism of $V$.
Consider the operad $\mc Chom(\Pi V)$ introduced in \cite{BDSHK18}
and briefly discussed in Section \ref{sec:chom}.
Consider the corresponding Lie superalgebra
$$
W_{\mc Chom(\Pi V)} = \bigoplus_{k=-1}^\infty W_{\mc Chom(\Pi V)}^k \,, \qquad
W_{\mc Chom(\Pi V)}^{k-1} = (\mc Chom(\Pi V))(k)^{S_{k}} \,.
$$
Then structures of a Lie conformal algebra on $V$ are in bijection with odd elements $X\in W_{\mc Chom}^1(\Pi V)$ 
such that $[X,X]=0$, so that
\begin{equation}\label{eq:Xbra}
[a_\la b] = (-1)^{p(a)}  X_{\la,-\la-\partial} (a \otimes b) \,, \qquad a,b \in V
\end{equation}
(see \cite[Section 4.3]{DSK13}).
As a result, we get the cohomology complex of the LCA $V$, defined by $X$, 
with coefficients in the adjoint module:
$$
C_{\mc Chom}(V) = \bigoplus_{n=0}^\infty C^n_{\mc Chom}(\Pi V) \,, \qquad
C^n_{\mc Chom}(V) := W_{\mc Chom(\Pi V)}^{n-1} = (\mc Chom(\Pi V))(n)^{S_n} \,,
$$
with the differential $d=\ad_X$ (cf. \cite{BKV99,DSK13}).
Given a module $M$ over the LCA $V$, the cohomology complex $C_{\mc Chom}(V,M)$ of $V$ 
with coefficients in $M$ can be obtained by a reduction procedure as in Section \ref{sec:lieoper} above. 
An explicit expression for the differential $d$ can be found in \cite[Eq.\ (4.19)]{DSK13}.

\subsection{Vertex algebra cohomology}
\label{sec:4.2}


Let again
$V$ be a vector superspace (with parity $p$), equipped with an $\mb F[\partial]$-module structure.
Consider the operad $\mc P^\ch(\Pi V)$ and the corresponding Lie superalgebra
$$
W_{\ch}(\Pi V) = \bigoplus_{k=-1}^\infty W_{\ch}^k(\Pi V) \,, \qquad
W_{\ch}^k(\Pi V) = (\mc P^\ch(\Pi V))(k+1)^{S_{k+1}} \,.
$$
Then structures of a non-unital vertex algebra on $V$ are in bijection with odd elements $X\in W_{\ch}^1(\Pi V)$ 
such that $[X,X]=0$ (see \cite[Theorem 6.12]{BDSHK18}). Explicitly, for $a,b\in V$,
\begin{equation}\label{eq:Xint}
(-1)^{p(a)} X^{z,w}_{\la,-\la-\partial} \Bigl(a \otimes b \otimes \frac1{w-z} \Bigr) 
= \int^\la d\si [a_\si b] = {:}ab{:} + \int_0^\la d\si [a_\si b] \,.
\end{equation}
We obtain the cohomology complex of a non-unital VA $V$ with coefficients in $V$:
\begin{equation}\label{eq:cva}
C_{\ch}(V) = \bigoplus_{n=0}^\infty C^n_{\ch}(V) \,, \qquad
C^n_{\ch}(V) := W_{\ch}^{n-1}(\Pi V) = \mc P^\ch(\Pi V)(n)^{S_n} \,,
\end{equation}
with the differential $d=\ad_X$.
Given a module $M$ over $V$, the cohomology complex $C_{\ch}(V,M)$ of $V$ with coefficients in $M$ is given by a reduction procedure as in Section \ref{sec:lieoper}. Namely, we consider the non-unital VA $V \oplus M$ such that $V$ is a subalgebra and
$$
\int^\la d\si [a_\si m] = \int^\la d\si (a_\si m) \,, \qquad
\int^\la d\si [m_\si m'] = 0 \,, \qquad a\in V \,, \; m,m' \in M \,.
$$
Then $C^n_{\ch}(V,M) = U^n/K^n$, where $U^n$ and $K^n$ are defined as in Section \ref{sec:lieoper}. 

Now we will give a more explicit description of this complex.
Elements of $C^n_{\ch}(V,M)$ are linear maps
\begin{equation}\label{eq:Ymap}
Y\colon
V^{\otimes n}\otimes\mc O^{\star T}_{n}\to
M[\lambda_1,\dots,\lambda_n]/\langle\partial+\lambda_1+\dots+\lambda_n\rangle
\,,
\end{equation}
satisfying the \emph{sesquilinearity} conditions
($v_1,\dots,v_n\in V$, $f\in\mc O^{\star T}_{n}$, $i=1,\dots,n$):
\begin{equation}\label{20160629:eq4}
\begin{array}{l}
\displaystyle{
\vphantom{\Big(}
Y_{\lambda_1,\dots,\lambda_n}(v_1\otimes\dots\otimes (\partial+\la_i) v_i \otimes\dots\otimes v_n\otimes f)
=Y_{\lambda_1,\dots,\lambda_n}\Bigl(v_1 \otimes\dots\otimes v_n \otimes\frac{\partial f}{\partial z_i}\Bigr)
\,,} \\
\displaystyle{
\vphantom{\Big(}
Y_{\lambda_1,\dots,\lambda_n}(v_1\otimes\dots\otimes v_n\otimes z_{ij}f)
=\Bigl(\frac{\partial}{\partial\lambda_j}-\frac{\partial}{\partial\lambda_i}\Bigr)
Y_{\lambda_1,\dots,\lambda_n}(v_1\otimes\dots\otimes v_n\otimes f)
\,,}
\end{array}
\end{equation}
and the \emph{symmetry} conditions ($1\leq i< n$):
\begin{equation}\label{20160629:eq5}
\begin{split}
Y&_{\lambda_1,\dots,\lambda_i,\lambda_{i+1},\dots,\lambda_n}
(v_1\otimes\dots\otimes v_i\otimes v_{i+1} \otimes\dots\otimes v_n\otimes f(z_1,\dots, z_i, z_{i+1},\dots,z_n))
\\
&=
(-1)^{\bar p(v_i) \bar p(v_{i+1})}
Y_{\lambda_1,\dots,\lambda_{i+1},\lambda_i,\dots,\lambda_n}
(v_1\otimes\cdots\otimes v_{i+1}\otimes v_i \otimes\cdots
\\
& \qquad\qquad
\cdots\otimes v_n\otimes f(z_1,\dots, z_{i+1}, z_i,\dots,z_n))
\,.
\end{split}
\end{equation}
In \eqref{20160629:eq5}, as before, $\bar p =1-p$ is the opposite parity to the parity $p$ of $V$.
(Note that when $V$ is purely even, $\bar p = \bar 1$ and \eqref{20160629:eq5} becomes a
\emph{skewsymmetry} condition on $Y$.)


We can describe explicitly the spaces $C_{\ch}^n(V,M)$ for $n=0,1,2$.
We have
$$
C_{\ch}^{0}(V,M)=M/\partial M
\,,\qquad
C_{\ch}^1(V,M)=\Hom_{\mb F[\partial]}(V,M)
\,.
$$
For $n=2$ we can identify  
$M[\lambda_1,\lambda_2]/\langle\partial+\lambda_1+\lambda_2\rangle\simeq M[\lambda_1]$.
Moreover, a map $Y$ as in \eqref{eq:Ymap} is uniquely determined by its value 
on the function $z_{21}^{-1}$.
Indeed, its values on the negative powers of $z_{21}$ are determined 
by the first sesquilinearity condition \eqref{20160629:eq4},
while its values on the non-negative powers of $z_{21}$ are determined 
by the second sesquilinearity condition \eqref{20160629:eq4}.
Hence, as in \eqref{eq:Xint}, $C_{\ch}^2(V,M)$ can then be identified with the superspace of 
integrals of $\lambda$-brackets
$$
V\otimes V\to M[\lambda]
\,,\qquad
u\otimes v\mapsto \int^\lambda d\sigma[u_\si v]^Y
\,,
$$
satisfying axiom V2 of sesquilinearity under integration,
and axiom V3 of symmetry under integration with respect to the opposite parity $\bar p=1-p$.

Next, we write explicitly the differential $d$ of the cohomology complex $C_{\ch}(V,M)$
(see \cite[Eq.\ (7.6)]{BDSHK18}).
In order to do so, we need to introduce some notation.
For every $n\in\mb Z$, we define the map
$V\otimes V\to V[\lambda]$, that sends $u\otimes v$ to
\begin{equation}\label{eq:notation-lambda}
\nabla_\lambda^{n}[u_\lambda v]
=
\left\{
\begin{array}{ll}
\displaystyle{
\frac{d^{n+1}}{d\lambda^{n+1}}\int^\lambda d\sigma[u_\sigma v] 
}
&\text{ for } n\geq0\,,\\
\displaystyle{
\frac1{m!}\int^\lambda d\sigma[(\partial+\lambda)^mu_\sigma v]
}
&\text{ for } n=-m-1\leq-1
\end{array}
\right.
\,.
\end{equation}
In particular, for $n=-1$ we recover the integral of $\lambda$-bracket \eqref{eq:int-lambda}
defining the vertex algebra $V$.
The reason for this notation is that, as it can be easily checked,
\begin{equation}\label{eq:vnumber}
\frac{d}{d\lambda}\big(\nabla_\lambda^n[u_\lambda v]\big)
=
\nabla_\lambda^{n+1}[u_\lambda v]
\;\;\text{ for all }\, n\in\mb Z\,.
\end{equation}
Likewise, for $u\in V$, $v\in M$ and $n\in\mb Z$, we let
\begin{equation}\label{eq:notation-lambda2}
\nabla_\lambda^{n}(u_\lambda v)
=
\left\{
\begin{array}{ll}
\displaystyle{
\frac{d^{n+1}}{d\lambda^{n+1}}\int^\lambda d\sigma(u_\sigma v) 
}
&\text{ for } n\geq0\,,\\
\displaystyle{
\frac1{m!}\int^\lambda d\sigma((\partial+\lambda)^mu_\sigma v)
}
&\text{ for } n=-m-1\leq-1
\end{array}
\right.
\,.
\end{equation}
In particular, for $n=-1$ we recover the integral of $\lambda$-action \eqref{eq:int-lambda2}
defining the $V$-module $M$.
We also extend the notations \eqref{eq:notation-lambda}--\eqref{eq:notation-lambda2} 
by linearity, thus making sense 
of $P(\nabla_\lambda)[u_\lambda v]$ and $P(\nabla_\lambda)(v_\lambda m)$,
where $P(z)\in\mb F[z,z^{-1}]$ is an arbitrary Laurent polynomial in $z$.

Next, 
let $h(z_0,\dots,z_k)\in\mc O^{\star T}_{k+1}$.
For every $i=0,\dots,k$, we decompose $h$ as
\begin{equation}\label{eq:dec1}
h(z_0,\dots,z_k)
=
f_i(z_0,\stackrel{i}{\check{\dots}},z_k)g_i(z_0,\dots,z_k)
\,,
\end{equation}
where $f_i\in\mc O^{\star T}_{k}$, 
$g_i\in\mc O^{\star T}_{k+1}$,
and $g_i$ may have poles only at $z_j=z_i$ for $j\neq i$.
Here and further, the notation $\stackrel{i}{\check{\dots}}$ means that the $i$-th term is omitted.
Moreover, for every $0\leq i<j\leq k$, 
we also decompose $h$ as
\begin{equation}\label{eq:dec2}
h(z_0,\dots,z_k)
=
f_{ij}(z_{ij})g_{ij}(z_0,\dots,z_k)
\,,
\end{equation}
where 
$f_{ij}\in\mb F[z_{ij},z_{ij}^{-1}]=\mc O^{\star T}_2$,
$g_{ij}\in\mc O^{\star T}_{k+1}$,
and $g_{ij}$ has no poles at $z_i=z_j$.

Using the notations \eqref{eq:notation-lambda}--\eqref{eq:notation-lambda2}
and the decompositions \eqref{eq:dec1}--\eqref{eq:dec2},
we can write the cohomology differential of $Y\in C_{\ch}^k(V,M)$,
as follows:
\begin{equation}\label{eq:dY}
\begin{split}
& (dY)^{z_0,\dots,z_k}_{\lambda_0,\dots,\lambda_k}(v_0\otimes\dots\otimes v_k \otimes h(z_0,\dots,z_k))
\\
& = 
\sum_{i=0}^k (-1)^{\gamma_i}
g_i(-\nabla_{\lambda_0},\dots,-\nabla_{\lambda_k})\,
{v_i}_{\lambda_i}
Y^{z_0,\stackrel{i}{\check{\dots}},z_k}_{\lambda_0,\stackrel{i}{\check{\dots}},\lambda_k}
(v_0\otimes\stackrel{i}{\check{\dots}}\otimes v_k\otimes 
f_i(z_0,\stackrel{i}{\check{\dots}},z_k))
\\
& +
\sum_{0\leq i<j\leq k} (-1)^{\gamma_{ij}}
Y^{w,z_0,\stackrel{i}{\check{\dots}}\stackrel{j}{\check{\dots}},z_k}_{\lambda_i+\lambda_j,
\lambda_0,\stackrel{i}{\check{\dots}}\stackrel{j}{\check{\dots}},\lambda_k}
\Big(
e^{-\partial_{z_i}\partial_{\lambda_i}}
\big(\\
& \qquad\qquad
f_{ij}(-\nabla_{\lambda_i})[{v_i}_{\lambda_i}{v_j}]
\otimes
v_0\otimes\stackrel{i}{\check{\dots}}\,\stackrel{j}{\check{\dots}}\otimes v_k\otimes
g_{ij}(z_0,\dots,z_k)
\big)\big|_{z_i=z_j=w}
\Big)\,,
\end{split}
\end{equation}
where $\gamma_i,\gamma_{i,j}\in\mb Z/2\mb Z$ are given by 
\begin{equation}\label{eq:gammai}
\begin{split}
\gamma_i &=
\bar p(v_i) (\bar p(Y)+\bar p(v_0)+\dots+\bar p(v_{i-1})+1)+1 \,,
\\
\gamma_{ij} &=
\bar p(Y)+\bar p(v_i)(\bar p(v_0)+\dots+\bar p(v_{i-1})+1)
+ \bar p(v_j)(\bar p(v_0)+\stackrel{i}{\check{\dots}}+\bar p(v_{j-1}))
\,.
\end{split}
\end{equation}

A few words of explanation are needed to clarify the meaning of formula \eqref{eq:dY}.
By construction, the rational function $g_i(z_0,\dots,z_k)$ 
has no poles at $z_j=z_\ell$ for $j$ and $\ell\neq i$.
If it has a pole at $z_i=z_j$ for some $j\neq i$,
we expand the operator 
$$
g_i(-\nabla_{\lambda_0},\dots,-\nabla_{\lambda_k})
$$
by geometric expansion in the domain $|z_i|>|z_j|$.
We then get non-negative powers $\nabla_{\lambda_j}^n=\frac{d^n}{d\lambda_j^n}$ which,
when applied to the polynomial
$$
Y^{z_0,\stackrel{i}{\check{\dots}},z_k}_{\lambda_0,\stackrel{i}{\check{\dots}},\lambda_k}
(v_0\otimes\stackrel{i}{\check{\dots}}\otimes v_k\otimes 
f_i(z_0,\stackrel{i}{\check{\dots}},z_k)) \,,
$$
vanish for $n$ large enough.
As a result, we are left with a Laurent polynomial of $\nabla_{\lambda_i}$
which can be ``applied'' to 
$$
{v_i}_{\lambda_i}Y(\cdots)
$$
according to the notation \eqref{eq:notation-lambda2}.

As for the second summand in \eqref{eq:dY},
when we expand the exponential $e^{-\partial_{z_i}\partial_{\lambda_i}}$
and apply it to the polynomial
$$
f_{ij}(-\nabla_{\lambda_i})[{v_i}_{\lambda_i}{v_j}]
$$
(defined by notation \eqref{eq:notation-lambda}),
only finitely many (non-negative) powers of $-\partial_{z_i}\partial_{\lambda_i}$ survive.
We can then apply the resulting operator to the rational function
$$
g_{ij}(z_0,\dots,z_k)
\,.
$$
Evaluating, as instructed, at $z_i=z_j=w$,
we get a function in $\mc O^{\star T}_k$
in the variables ${w,z_0,\stackrel{i}{\check{\dots}}\stackrel{j}{\check{\dots}},z_k}$
(on which we evaluate the map $Y$).

It is straightforward to check that the differential $d$ defined by formula \eqref{eq:dY} coincides 
with the differential given by \cite[Eq.\ (7.6)]{BDSHK18}.

\begin{definition}\label{def:va-coho}
Given a module $M$ over the VA $V$,
the cohomology of the complex $(C_{\ch}(V,M),d)$
is called the \emph{VA cohomology} of $V$ with coefficients in $M$:
\begin{equation}\label{eq:h-va}
H_{\ch}(V,M)
=
\bigoplus_{n=0}^\infty
H_{\ch}^n(V,M)
\,.
\end{equation}
\end{definition}

The following lemma will be useful for computing the cohomology of a vertex algebra.

\begin{lemma}\label{lem:rvm}
Let\/ $V$ be a VA, $M$ be a\/ $V$-module, and\/ $R$ be an LCA, which is a subalgebra of the LCA\/ $(V,[\,{}_\la\,])$.
Then the restriction map 
\begin{equation}\label{eq:rvm1}
Y\in C^n_{\ch}(V,M) \mapsto
Y|_{R^{\otimes n} \otimes 1} \in C^n_{\mc Chom}(R,M)
\end{equation}
is a morphism of complexes, i.e., it commutes with the differentials.
\end{lemma}
\begin{proof}
The special case $R=V$ follows from \eqref{eq:ch-chom}.
In general, the claim can be deduced directly from the definitions.
Indeed, formula \eqref{eq:dY} for the differential in $C^n_{\ch}(V,M)$,
evaluated at the function $h(z_0,\dots,z_k)=1$,
reduces to formula \cite[Eq.\ (4.19)]{DSK13} for the differential in $C^n_{\mc Chom}(R,M)$.
\end{proof}

Now we review the description of the low degree cohomology $H_{\ch}^n(V,M)$ ($n=0,1,2$) in terms of Casimirs, derivations, and extensions \cite{BDSHK18}.
We denote by $\int\colon M \to M/\partial M$ the canonical projection, and say that
$\int m\in M/\partial M$ is a \emph{Casimir element} if $V_{-\partial}m=0$.
Let $\Cas(V,M)\subset M/\partial M$ be the space of Casimir elements.
Note that
$$
\Cas(V,V) = \bigl\{ \tint a \in V/\partial V \,\big|\, a_{(0)} V = 0 \bigr\} \,.
$$

A \emph{derivation} from $V$ to $M$
is an $\mb F[\partial]$-module homomorphism $D\colon V\to M$ such that 
\begin{equation}\label{eq:derv}
\begin{split}
D\Big(\int^\lambda\!d\sigma[{u}_\sigma{v}]\Big)
& =
(-1)^{p(D) p(u)}
\int^\lambda\!d\sigma
({u}_\sigma{D(v)}) \\
& +
(-1)^{(p(D)+p(u))p(v)}
\int^{-\lambda-\partial}\!d\sigma
(v_\sigma{D(u)})
\,.
\end{split}
\end{equation}
Note that $D\in C^1_{\ch}(V,M)$ is closed if and only if $D$ is a derivation $V\to M$,
and it is exact if and only if this derivation is \emph{inner}, 
i.e. it has the form 
\begin{equation}\label{eq:inner}
D_{\tint m}(a)=(-1)^{1+p(m)p(a)}a_{-\partial}m
\,\quad\text{ for some }\quad \tint m\in M/\partial M
\,.
\end{equation}
In the special case when $V=M$,
we have the usual definition of a derivation and an inner derivation of the vertex algebra $V$
(an inner derivation is of the form $v_{(0)}$ for some $v\in V$).
Denote by $\Der(V,M)$ the space of derivations from $V$ to $M$,
and by $\Inder(V,M)$ the subspace of inner derivations.

\begin{proposition}[{\cite{BDSHK18}}]\label{prop:lowcoho}
Let\/ $V$ be a (non-unital) VA and let\/ $M$ be a\/ $V$-module. Then:
\begin{enumerate}[(a)]
\item
$H_\ch^0(V,M)=\Cas(V,M)$.
\item\medskip
$H_\ch^1(V,M)=\Der(V,M)/\Inder(V,M)$.
\item\medskip
$H_\ch^2(V,M)$ is the space of isomorphism classes of\/ $\mb F[\partial]$-split extensions
of the vertex algebra\/ $V$ by the\/ $V$-module\/ $M$, where\/ $M$ is
viewed as a non-unital vertex algebra with zero integral of\/ $\lambda$-bracket.
\end{enumerate}
\end{proposition}

\subsection{Classical and variational PVA cohomology}\label{sec:clpva}

Let $\mc V$ be a vector superspace (with parity $p$) equipped with an $\mb F[\partial]$-module structure.
Consider the operad $\mc P^\cl(\Pi \mc V)$ and the corresponding Lie superalgebra
$$
W_{\cl}(\Pi \mc V) = \bigoplus_{k=-1}^\infty W_{\cl}^k(\Pi \mc V) \,, \qquad
W_{\cl}^k(\Pi \mc V) := W_{\mc P^\cl(\Pi \mc V)}^k = \mc P^\cl(\Pi \mc V)(k+1)^{S_{k+1}} \,.
$$
The structures of Poisson vertex algebra on $\mc V$ are in bijection with odd elements $X\in W_{\cl}^1(\Pi \mc V)$ such that $[X,X]=0$ (see \cite[Theorem 10.7]{BDSHK18}). Explicitly, for $a,b\in V$, we have
\begin{equation}\label{eq:Xpva}
ab=(-1)^{p(a)}
X^{\bullet\!-\!\!\!\!\to\!\bullet}(a\otimes b)
\,,\quad
[a_\lambda b]=(-1)^{p(a)}X^{\bullet\,\,\bullet}_{\lambda,-\lambda-\partial}(a\otimes b)
\,.
\end{equation}
As before, we get a complex
\begin{equation}\label{eq:Ccl}
C_{\cl}(\mc V) = \bigoplus_{n=0}^\infty C_{\cl}^n(\mc V) \,, \qquad
C_{\cl}^n(\mc V) := W_{\cl}^{n-1}(\Pi \mc V) \,,
\end{equation}
with the differential $d=\ad_X$. The cohomology of this complex will be called the \emph{classical PVA cohomology} of the PVA $\mc V$ with coefficients in $\mc V$, and will be denoted by
$$
H_{\cl}(\mc V) = \bigoplus_{n=0}^\infty H_{\cl}^n(\mc V) \,.
$$

For the rest of this subsection, we consider $\mc V$ with its trivial grading: $\mc V = \gr^0\mc V$.
Then, by \eqref{pclgrading}, $Y\in \mc P^\cl(n)$ is in $\gr^p \mc P^\cl(n)$ if and only if
$Y^\Ga = 0$ for every graph $\Ga$ with $n$ vertices and number of edges $\ne p$
(see \cite[(10.9)]{BDSHK18}). Since the symmetric group $S_n$ preserves this grading,
the Lie superalgebra $W_{\cl}(\Pi\mc V)$ and the complex $C_{\cl}(\mc V)$ are graded by:
$$
W_{\cl}^k(\Pi \mc V) = \bigoplus_{p=0}^{k} \gr^p W_{\cl}^k(\Pi \mc V) \,, \qquad
C_{\cl}^n(\mc V) = \bigoplus_{p=0}^{n-1} \gr^p C_{\cl}^n(\mc V)
$$
(see \eqref{eq:grpcl}).
We can write our element $X\in W_{\cl}^1(\Pi \mc V)$ as $X=X_0+X_1$ with 
$X_r \in \gr^r W_{\cl}^1(\Pi \mc V)$ ($r=0,1$), so that
$$
(X_0)^{\bullet\!-\!\!\!\!\to\!\bullet} = 0 \,, \quad
(X_0)^{\bullet\,\,\bullet} = X^{\bullet\,\,\bullet} \,, \quad
(X_1)^{\bullet\!-\!\!\!\!\to\!\bullet} = X^{\bullet\!-\!\!\!\!\to\!\bullet} \,, \quad
(X_1)^{\bullet\,\,\bullet} = 0 \,.
$$
Comparing terms of different degrees in the equation $[X,X]=0$, we obtain
$$
[X_0,X_0]=[X_0,X_1]=[X_1,X_1]=0 \,.
$$
Hence, $d=\ad_X$ is a sum of two anticommuting differentials 
$d_r=\ad_{X_r}$ ($r=0,1$), such that $d_r\colon\gr^p C_{\cl}^n(\mc V) \to \gr^{p+r} C_{\cl}^{n+1}(\mc V)$.

Note that $\gr^0 \mc P^\cl(\Pi\mc V) = \mc Chom(\Pi\mc V)$ and the Lie superalgebra $W_{\mc Chom}(\Pi \mc V)$ from Section \ref{sec:lca-coh} coincides with the subalgebra $\gr^0 W_{\cl}(\Pi \mc V)$ of $W_{\cl}(\Pi \mc V)$.
The odd element $X_0\in W_{\mc Chom}^1(\Pi \mc V)$ corresponds to the $\la$-bracket in $\mc V$ considered as an LCA (see \eqref{eq:Xbra}). Then $d_0$ gives the differential in the LCA cohomology complex
$C_{\mc Chom}(\mc V) = \gr^0 C_{\cl}(\mc V)$.

Since $d_1$ is a derivation of the Lie superalgebra $W_{\cl}(\Pi \mc V)$, its kernel is a subalgebra. Hence,
the subspace
$$
W_{\PV}(\Pi \mc V) := \gr^0 W_{\cl}(\Pi \mc V) \cap \ker d_1
$$
is a subalgebra of $W_{\mc Chom}(\Pi \mc V)$. As $X_0\in W_{\PV}^1(\Pi \mc V)$, 
the differential $d_0$ restricts to $W_{\PV}(\Pi \mc V)$ and gives a complex
$$
C_{\PV}(\mc V) = \bigoplus_{n=0}^\infty C_{\PV}^n(\mc V) \,, \qquad
C_{\PV}^n(\mc V) := W_{\PV}^{n-1}(\Pi \mc V) \,,
$$
called the \emph{variational PVA complex} of $\mc V$ (see \cite{DSK13}).
The cohomology of this complex is called the \emph{variational PVA cohomology} of the PVA $\mc V$ with coefficients in $\mc V$, and is denoted by
$$
H_{\PV}(\mc V) = \bigoplus_{n=0}^\infty H_{\PV}^n(\mc V) \,.
$$
Explicitly, by \cite[Lemma 11.3(c)]{BDSHK18},
$C_{\PV}^n(\mc V)$ can be described as the subspace of $C_{\mc Chom}^n(\mc V)$ consisting 
of cochains $Y$ 
that satisfy the following \emph{Leibniz rules}:
\begin{align}
\notag
& 
Y_{\lambda_1,\dots,\lambda_n}
(a_1\otimes\cdots\otimes b_ic_i \otimes\cdots\otimes a_n) \\
\label{eq:leib}
& =
(-1)^{p(b_i)(\bar p(Y)+\bar p(a_1)+\dots+\bar p(a_{i-1}))}
(e^{\partial\partial_{\lambda_i}}b_i)
Y_{\lambda_1,\dots,\lambda_n}
(a_1\otimes\cdots\otimes c_i \otimes\cdots\otimes a_n) \\
\notag
& +
(-1)^{p(c_i)(p(b_i)+\bar p(Y)+\bar p(a_1)+\dots+\bar p(a_{i-1}))}
(e^{\partial\partial_{\lambda_i}}c_i)
Y_{\lambda_1,\dots,\lambda_n}
(a_1\otimes\cdots\otimes b_i \otimes\cdots\otimes a_n)
\,,
\end{align}
for all $i=1,\dots,n$ and $a_j,b_i,c_i \in\mc V$.

\begin{remark}\label{rem:pva-vac}
As a consequence of the Leibniz rules \eqref{eq:leib} and the sesquilinearity \eqref{eq:chomses},
any cochain $Y \in C_{\PV}^n(\mc V)$ vanishes whenever one of its arguments is the unit $1\in\mc V$
(see \cite[Lemma 3.7(b)]{BDSK19}).
\end{remark}

The relationship between the classical and variational PVA cohomology is clarified in the next theorem.

\begin{theorem}\label{thm:cl-pv}
Let\/ $\mc V$ be a Poisson vertex algebra.
\begin{enumerate}[(a)]
\item
The embedding of complexes\/ $(C_{\PV}(\mc V),d_0) \hookrightarrow (C_{\cl}(\mc V),d)$
induces an injective homomorphism of cohomology\/
\begin{equation}\label{eq:pv-cl}
H_{\PV}(\mc V) \hookrightarrow H_{\cl}(\mc V) \,.
\end{equation}
\item
Suppose that\/ $\mc V$, as a differential superalgebra, is isomorphic to an algebra of differential polynomials in finitely many even or odd variables. Then the map \eqref{eq:pv-cl} is an isomorphism.
\end{enumerate}
\end{theorem}
\begin{proof}
Part (a) is \cite[Theorem 11.4]{BDSHK18}, while part (b) is in \cite{BDSHKV20}.
\end{proof}

In the case of a universal PVA over an LCA, the variational PVA cohomology is related to the LCA cohomology as follows (see \cite[Theorem 3.13]{BDSK19}).

\begin{proposition}\label{prop:pv-lc}
Let\/ $R$ be an LCA and\/ $\mc V=\mc V(R)$ be its universal PVA.
\begin{enumerate}[(a)]
\item
If\/ $M$ is a\/ $\mc V$-module, then\/ $M$ is also an\/ $R$-module by restriction, and the restriction of cochains
\begin{equation}\label{eq:rvm2}
Y\in C^n_{\PV}(\mc V,M) \mapsto
Y|_{R^{\otimes n}} \in C^n_{\mc Chom}(R,M)
\end{equation}
gives an isomorphism of complexes\/ $C_{\PV}(\mc V,M) \simeq C_{\mc Chom}(R,M)$.

\item
Let\/ $C\in R$ be such that\/ $\partial C=0$, let\/ $\bar R=R/\mb F C$ be the quotient LCA, and\/ 
$\mc V^c :=\mc V^c(R) = \mc V(R) / \mc V(R) (C-c)$ be the corresponding quotient PVA. Then every\/ $\mc V^c$-module\/ $M$ is an\/ $\bar R$-module by restriction, we have natural embeddings of complexes
\begin{equation}\label{eq:rvm3}
C_{\mc Chom}(\bar R,M) \subset C_{\mc Chom}(R,M) \,, \qquad
C_{\PV}(\mc V^c,M) \subset C_{\PV}(\mc V,M) \,,
\end{equation}
and the isomorphism \eqref{eq:rvm2} restricts to an isomorphism of complexes
\begin{equation}\label{eq:rvm4}
C_{\mc Chom}(\bar R,M) \simeq C_{\PV}(\mc V^c,M) \,.
\end{equation}
\end{enumerate}
\end{proposition}
\begin{proof}
This is, in a different notation, the same as \cite[Theorem 3.13]{BDSK19}.
\end{proof}

\section{Wick-type formula for VA cocycles and bounded VA cohomology}\label{sec:wick}

\subsection{Wick-type formula}\label{subsec:wick}

The main result of this section
is a Wick-type formula for cocycles in $C_{\ch}(V,M)$.
This formula should help in proving that a good filtration on $V$ induces an exhaustive filtration
in cohomology.

\begin{proposition}\label{prop:wick-formula}
Let\/ $V$ be a vertex algebra, $M$ be a $V$-module, and\/
$Y\in C^n_{\ch}(V,M)$ be a closed element of the VA cohomology complex: $dY=0$. 
Then, the following Wick-type formula holds, for every $\ell=1,\dots,n$ 
\begin{equation}\label{eq:wick11}
\begin{split}
& 
Y_{\lambda_0,\dots \lambda_{\ell-1}+\lambda_\ell \dots,\lambda_n}
\big(
v_0\otimes\dots\otimes {:}v_{\ell-1}v_\ell{:} \otimes\dots \otimes v_n\otimes
1
\big) \\
& =
(-1)^{p_{\ell-1}\bar p_{-1,{\ell-1}}}
{:}
{(e^{\partial\partial_{\lambda_{\ell-1}}}v_{\ell-1})}
Y_{\lambda_0,\dots\lambda_{\ell-1}+\lambda_\ell\dots,\lambda_n}
(v_0\otimes\stackrel{\ell-1}{\check{\dots}}\otimes v_n\otimes 1) 
{:} \\
& + (-1)^{p_\ell(1+\bar p_{-1,\ell})}
{:}
{(e^{\partial\partial_{\lambda_\ell}} v_\ell)}
Y_{\lambda_0,\dots\lambda_{\ell-1}+\lambda_\ell\dots,\lambda_n}
(v_0\otimes\stackrel{\ell}{\check{\dots}}\otimes v_n\otimes 1) 
{:}
\\
& +
(-1)^{p_{\ell-1}\bar p_{-1,{\ell-1}}}
\int_0^{\lambda_{\ell-1}} \!\!\! d\sigma \,
{v_{\ell-1}}_{\sigma}
Y_{\lambda_0,\dots\lambda_{\ell-1}+\lambda_\ell-\sigma\dots,\lambda_n}
(v_0\otimes\stackrel{\ell-1}{\check{\dots}}\otimes v_n\otimes 1) \\
& + 
(-1)^{p_\ell(1+\bar p_{-1,\ell})}
\int_0^{\lambda_\ell}
\!d\sigma\,
{v_\ell}_{\sigma}
Y_{\lambda_0,\dots\lambda_{\ell-1}+\lambda_\ell-\sigma\dots,\lambda_n}
(v_0\otimes\stackrel{\ell}{\check{\dots}}\otimes v_n\otimes 1) \\
& - \sum_{\substack{ i=0 \\ i\neq\ell-1,\ell }}^{n}
(-1)^{
p_i\bar p_{-1,i}
+
\bar p_{i,\ell-1}
}
{v_i}_{\lambda_i}
Y^{z_0,\stackrel{i}{\check{\dots}},z_n}_{\lambda_0,\stackrel{i}{\check{\dots}},\lambda_n}
\Big(v_0\otimes\stackrel{i}{\check{\dots}}\otimes v_n\otimes 
\frac1{z_{\ell-1}-z_\ell}\Big) \\
& - 
\int_0^{\lambda_{\ell-1}}\!\!\!\!\!\!\!\!d\sigma\,
Y_{\lambda_0,\dots \lambda_{\ell-1}\!+\!\lambda_\ell \dots,\lambda_n}
\big(
v_0\otimes \dots\otimes [{{v_{\ell-1}}_{\sigma}{v_\ell}}] \otimes\dots \otimes v_n\otimes
1
\big) \\
& - \!\!\! \sum_{\substack{ i=0 \\ i\neq\ell-1,\ell }}^{n} \!\!\!
(-1)^{ p_{\ell-1}(1+\bar p_{i,{\ell-1}}) }
Y^{z_0,\stackrel{\ell-1}{\check{\dots}},z_n}_{
\lambda_0,\dots \lambda_i+\lambda_{\ell-1} \stackrel{\ell-1}{\check{\dots}},\lambda_n}
\Big(
v_0\otimes \cdots [{v_i}_{\lambda_i}{v_{\ell-1}}] \stackrel{\ell-1}{\check{\dots}} \otimes v_n\otimes
\frac1{z_i-z_\ell}
\Big) \\
& + \sum_{\substack{ i=0 \\ i\neq\ell-1,\ell }}^{n} 
(-1)^{p_\ell\bar p_{i,\ell}}
Y^{z_0,\stackrel{\ell}{\check{\dots}},z_n}_{
\lambda_0,\dots \lambda_i+\lambda_\ell \stackrel{\ell}{\check{\dots}},\lambda_n}
\Big(
v_0\otimes \cdots [{v_i}_{\lambda_i}{v_\ell}] \stackrel{\ell}{\check{\dots}} \otimes v_n\otimes
\frac1{z_{\ell-1}-z_i}
\Big) \\
& + \!\!\!\! \sum_{\stackrel{ 0\leq i<j \leq n }{ i,j\not\in\{\ell-1,\ell\} }} \!\!\!\!
(-1)^{
p_j\bar p_{i,j}+\bar p_{j-1,\ell-1}
}
Y^{z_0,\stackrel{j}{\check{\dots}},z_n}_{
\lambda_0,\dots \lambda_i+\lambda_j \stackrel{j}{\check{\dots}},\lambda_n}
\Big(
v_0\otimes\cdots [{v_i}_{\lambda_i}{v_j}]\stackrel{j}{\check{\dots}} \otimes v_n\otimes
\frac1{z_{\ell-1}-z_\ell}
\Big) 
\,,
\end{split}
\end{equation}
where we let $p_i=p(v_i)$, $\bar p_i=\bar p(v_i)$, and
\begin{equation}\label{eq:pibar}
\bar p_{i,j}
=
\left\{\begin{array}{ll}
\vphantom{\Big(}
\bar p(Y)+\bar p(v_0)+\dots+\bar p(v_j) & \text{ if }\; i=-1 \text{ and }\, j\geq0
\,, \\
\vphantom{\Big(}
\bar p(v_{i+1})+\dots+\bar p(v_j) & \text{ if }\; 0\leq i\leq j 
\,, \\
\vphantom{\Big(}
\bar p(v_{j+1})+\dots+\bar p(v_i) & \text{ if }\; 0\leq j\leq i
\,.
\end{array}\right.
\end{equation}
\end{proposition}
\begin{proof}
The function $h=(z_{\ell-1}-z_\ell)^{-1}$
decomposes as in \eqref{eq:dec1} with 
\begin{align*}
& f_i=h\,\,,\,\,\,\, g_i=1
\,\,\text{ for }\,\, i\neq\ell-1,\ell\,, \\
& f_i=1 \,\,,\,\,\,\, g_i=h
\,\,\text{ for }\,\, i\in\{\ell-1,\ell\}\,,
\end{align*}
and as in \eqref{eq:dec2} with 
\begin{align*}
& f_{ij}=1 \,\,,\,\,\,\, g_{ij}=h \,\,\text{ for } (i,j)\neq(\ell-1,\ell) \,, \\
& f_{\ell-1,\ell}=h\,\,,\,\,\,\, g_{\ell-1,\ell}=1 
\,.
\end{align*}
Hence, if we evaluate equation \eqref{eq:dY} at $h$, we get
\begin{equation}\label{eq:wick1}
\begin{split}
& (dY)^{z_0,\dots,z_k}_{\lambda_0,\dots,\lambda_k}\Big(v_0\otimes\dots\otimes v_k \otimes\frac1{z_{\ell-1}-z_\ell}\Big) \\
& = (-1)^{\gamma_{\ell-1}}
(\nabla_{\lambda_\ell}-\nabla_{\lambda_{\ell-1}})^{-1}
{v_{\ell-1}}_{\lambda_{\ell-1}}
Y_{\lambda_0,\stackrel{\ell-1}{\check{\dots}},\lambda_n}
(v_0\otimes\stackrel{\ell-1}{\check{\dots}}\otimes v_n\otimes 1) \\
& + (-1)^{\gamma_\ell}
(\nabla_{\lambda_\ell}-\nabla_{\lambda_{\ell-1}})^{-1}
{v_\ell}_{\lambda_\ell}
Y_{\lambda_0,\stackrel{\ell}{\check{\dots}},\lambda_n}
(v_0\otimes\stackrel{\ell}{\check{\dots}}\otimes v_n\otimes 1) \\
& + \sum_{i\neq\ell-1,\ell} (-1)^{\gamma_i}
{v_i}_{\lambda_i}
Y^{z_0,\stackrel{i}{\check{\dots}},z_n}_{\lambda_0,\stackrel{i}{\check{\dots}},\lambda_n}
\Big(v_0\otimes\stackrel{i}{\check{\dots}}\otimes v_n\otimes 
\frac1{z_{\ell-1}-z_\ell}\Big) \\
& +
(-1)^{\gamma_{\ell-1,\ell}}
Y_{\lambda_{\ell-1}+\lambda_\ell,
\lambda_0,\stackrel{\ell-1,\ell}{\check{\dots}},\lambda_n}
\Big(
(-\nabla_{\lambda_{\ell-1}})^{-1}[{{v_{\ell-1}}_{\lambda_{\ell-1}}{v_\ell}}]
\otimes
v_0\otimes\stackrel{\ell-1,\ell}{\check{\dots}} \otimes v_n\otimes
1
\Big) \\
& +
\sum_{\substack{ 0\leq i<j\leq n \\ (i,j)\neq(\ell-1,\ell) }} (-1)^{\gamma_{ij}}
Y^{w,z_0,\stackrel{i}{\check{\dots}}\stackrel{j}{\check{\dots}},z_n}_{\lambda_i+\lambda_j,
\lambda_0,\stackrel{i}{\check{\dots}}\stackrel{j}{\check{\dots}},\lambda_n}
\Big(
e^{-\partial_{z_i}\partial_{\lambda_i}}
\Big (\\
& \qquad\qquad
[{v_i}_{\lambda_i}{v_j}]
\otimes
v_0\otimes\stackrel{i}{\check{\dots}}\stackrel{j}{\check{\dots}} \otimes v_n\otimes
\frac1{z_{\ell-1}-z_\ell}
\Big )\Big |_{z_i=z_j=w}
\Big) \,,
\end{split}
\end{equation}
where $\gamma_i,\gamma_{i,j}\in\mb Z/2\mb Z$ are as in \eqref{eq:gammai}.
According to the explanation of equation \eqref{eq:dY},
in the first term in the right-hand side of \eqref{eq:wick1} we need 
to expand $(\nabla_{\lambda_\ell}-\nabla_{\lambda_{\ell-1}})^{-1}$
by geometric series in non-negative powers of $\nabla_{\lambda_\ell}$
and use the notation \eqref{eq:notation-lambda2}.
Recalling \eqref{eq:int-lambda3}, we get as a result
\begin{equation}\label{eq:wick2}
\begin{split}
& -
(-1)^{\gamma_{\ell-1}}
\sum_{n=0}^\infty
\frac1{n!}
\int^{\lambda_{\ell-1}} \!\!\! d\sigma \,
{((\partial+\lambda_{\ell-1})^nv_{\ell-1})}_{\sigma}
\Big(\frac{\partial^n}{\partial{\lambda_\ell}^{n}}
Y_{\lambda_0,\stackrel{\ell-1}{\check{\dots}},\lambda_n}
(v_0\otimes\stackrel{\ell-1}{\check{\dots}}\otimes v_n\otimes 1) 
\Big) \\
& = -
(-1)^{\gamma_{\ell-1}}
{:}
{(e^{\partial\partial_{\lambda_{\ell-1}}}v_{\ell-1})}
Y_{\lambda_0,\dots\lambda_{\ell-1}+\lambda_\ell\dots,\lambda_n}
(v_0\otimes\stackrel{\ell-1}{\check{\dots}}\otimes v_n\otimes 1) 
{:} \\
& -
(-1)^{\gamma_{\ell-1}}
\int_0^{\lambda_{\ell-1}} \!\!\! d\sigma \,
{v_{\ell-1}}_{\sigma}
Y_{\lambda_0,\dots\lambda_{\ell-1}+\lambda_\ell-\sigma\dots,\lambda_n}
(v_0\otimes\stackrel{\ell-1}{\check{\dots}}\otimes v_n\otimes 1) 
\,.
\end{split}
\end{equation}
Similarly, the second term in the right-hand side of \eqref{eq:wick1} is
\begin{equation}\label{eq:wick3}
\begin{split}
& (-1)^{\gamma_\ell}
{:}
{(e^{\partial\partial_{\lambda_\ell}} v_\ell)}
Y_{\lambda_0,\dots\lambda_{\ell-1}+\lambda_\ell\dots,\lambda_n}
(v_0\otimes\stackrel{\ell}{\check{\dots}}\otimes v_n\otimes 1) 
{:}
\\
& + 
(-1)^{\gamma_\ell}
\int_0^{\lambda_\ell}
\!d\sigma\,
{v_\ell}_{\sigma}
Y_{\lambda_0,\dots\lambda_{\ell-1}+\lambda_\ell-\sigma\dots,\lambda_n}
(v_0\otimes\stackrel{\ell}{\check{\dots}}\otimes v_n\otimes 1) 
\,.
\end{split}
\end{equation}
The fourth term 
is, by the notation \eqref{eq:notation-lambda2}
and the symmetry conditions \eqref{20160629:eq5},
\begin{equation}\label{eq:wick5}
\begin{split}
& - (-1)^{\bar p(Y)+\bar p(v_0)+\dots+\bar p(v_{\ell-1})}
Y_{\lambda_0,\dots \lambda_{\ell-1}+\lambda_\ell \dots,\lambda_n}
\big(
v_0\otimes\cdots {:}v_{\ell-1}v_\ell{:} \cdots \otimes v_n\otimes
1
\big) \\
& - (-1)^{\bar p(Y)+\bar p(v_0)+\dots+\bar p(v_{\ell-1})}
\!\!\!\!
\int_0^{\lambda_{\ell-1}}\!\!\!\!\!\!\!\!d\sigma\,
Y_{\lambda_0,\dots \lambda_{\ell-1}\!+\!\lambda_\ell \dots,\lambda_n}
\big(
v_0\otimes \cdots {[{{v_{\ell-1}}_{\sigma}{v_\ell}}]} \cdots \otimes v_n\otimes
1
\big) 
\,.
\end{split}
\end{equation}
Let us consider now the last, fifth term of the right-hand side of \eqref{eq:wick1}.
We sum over all pairs of indices $i,j$ such that $1\leq i<j\leq n$ and $(i,j)\neq(\ell-1,\ell)$.
The terms with $i<j=\ell-1$ give
\begin{equation}\label{eq:wick6}
\begin{split}
\sum_{i=0}^{\ell-2} &
(-1)^{\bar p(Y)+\bar p(v_0)+\dots+\bar p(v_i)+\bar p(v_{\ell-1})(\bar p(v_{i+1})+\dots+\bar p(v_{\ell-2}))}
\\
& \times Y^{z_0,\stackrel{\ell-1}{\check{\dots}},z_n}_{
\lambda_0,\dots \lambda_i+\lambda_{\ell-1} \stackrel{\ell-1}{\check{\dots}},\lambda_n}
\Big (
v_0\otimes \cdots [{v_i}_{\lambda_i}{v_{\ell-1}}] \stackrel{\ell-1}{\check{\dots}} \otimes v_n\otimes
\frac1{z_i-z_\ell}
\Big ) \,.
\end{split}
\end{equation}
The terms with $i<\ell-1,\,j=\ell$ give
\begin{equation}\label{eq:wick7}
\begin{split}
\sum_{i=0}^{\ell-2} &
(-1)^{
\bar p(Y)
+
\bar p(v_0)+\dots+\bar p(v_i)
+
\bar p(v_\ell)(\bar p(v_{i+1})+\dots+\bar p(v_{\ell-1}))
} \\
& \times
Y^{z_0,\stackrel{\ell}{\check{\dots}},z_n}_{
\lambda_0,\dots \lambda_i+\lambda_\ell \stackrel{\ell}{\check{\dots}},\lambda_n}
\Big (
v_0\otimes \cdots [{v_i}_{\lambda_i}{v_\ell}] \stackrel{\ell}{\check{\dots}} \otimes v_n\otimes
\frac1{z_{\ell-1}-z_i}
\Big ) \,.
\end{split}
\end{equation}
The terms with $i=\ell-1,\,j>\ell$ give, by the first sesquilinearity condition \eqref{20160629:eq4} 
and the skewsymmetry of the $\lambda$-bracket,
\begin{equation}\label{eq:wick8}
\begin{split}
\sum_{j=\ell+1}^n &
(-1)^{
\bar p(Y)
+
\bar p(v_0)+\stackrel{\ell-1}{\check{\dots}}+\bar p(v_j)
+
\bar p(v_{\ell-1})(\bar p(v_\ell)+\dots+\bar p(v_j))
} \\
& \times
Y^{z_0,\stackrel{\ell-1}{\check{\dots}},z_n}_{
\lambda_0,\stackrel{\ell-1}{\check{\dots}} \lambda_{\ell-1}+\lambda_j \dots,\lambda_n}
\Big(
v_0\otimes \stackrel{\ell-1}{\check{\dots}} [{v_j}_{\lambda_j}{v_{\ell-1}}] \dots \otimes v_n\otimes
\frac1{z_{j}-z_\ell}
\Big) \,.
\end{split}
\end{equation}
Similarly, the terms with $i=\ell<j$ give
\begin{equation}\label{eq:wick9}
\begin{split}
\sum_{j=\ell+1}^n &
(-1)^{
\bar p(Y)
+
\bar p(v_0)+ \stackrel{\ell}{\check{\dots}} + \bar p(v_j)
+
\bar p(v_\ell)(\bar p(v_{\ell+1})+\dots+\bar p(v_j))
} \\
& \times
Y^{z_0,\stackrel{\ell}{\check{\dots}},z_n}_{
\lambda_0,\stackrel{\ell}{\check{\dots}} \lambda_\ell+\lambda_j \dots,\lambda_n}
\Big(
v_0\otimes\stackrel{\ell}{\check{\dots}} [{v_j}_{\lambda_j}{v_\ell}] \dots \otimes v_n\otimes
\frac1{z_{\ell-1}-z_j}
\Big) \,.
\end{split}
\end{equation}
Finally, the terms with $\{i,j\}\cap\{\ell-1,\ell\}=\emptyset$ give
\begin{equation}\label{eq:wick10}
\begin{split}
& \sum_{\stackrel{ 0\leq i<j \leq n }{ i,j\not\in\{\ell-1,\ell\} }}
(-1)^{\gamma_{ij}}
Y^{w,z_0,\stackrel{i}{\check{\dots}}\stackrel{j}{\check{\dots}},z_n}_{\lambda_i+\lambda_j,
\lambda_0,\stackrel{i}{\check{\dots}}\stackrel{j}{\check{\dots}},\lambda_n}
\Big(
[{v_i}_{\lambda_i}{v_j}]
\otimes
v_0\otimes\stackrel{i}{\check{\dots}}\!\!\stackrel{j}{\check{\dots}} \otimes v_n\otimes
\frac1{z_{\ell-1}-z_\ell}
\Big) \,,
\end{split}
\end{equation}
Combining all equations \eqref{eq:wick1}--\eqref{eq:wick10}, 
and using the assumption that $dY=0$, we obtain \eqref{eq:wick11}.
\end{proof}

We can compute more explicitly formula \eqref{eq:wick11} for $n=1$ and $2$. 
First, if $Y$ is a closed element in $C^1_{\ch}(V,M)=\Hom_{\mb F[\partial]}(V,M)$, we get:
\begin{align*}
Y({:}ab{:}) 
&=
(-1)^{p(a)\bar p(Y)}
{:} a Y(b) {:} 
+ 
(-1)^{p(a)p(b) + p(b)\bar p(Y)}
{:} b Y(a) {:}
\\
& +
(-1)^{p(a)\bar p(Y)}\!\!\!
\int_0^{\mu} \!\!\! d\sigma \,
{a}_{\sigma}
Y(b) 
- 
\int_0^{\mu}\!\!d\sigma\,
Y([{a}_{\sigma}{b}])
\\
&+ 
(-1)^{ p(a)p(b) + p(b)\bar p(Y) }\!\!\!
\int_0^{-\mu-\partial}
\!\!\!\!\!\!d\sigma\,
{b}_{\sigma}
Y(a)
\,.
\end{align*}
The above equation holds for every $\mu$, hence it gives
\begin{align*}
Y({:}ab{:}) 
&=
(-1)^{p(a)\bar p(Y)}
{:} a Y(b) {:} 
+ 
(-1)^{p(a)p(b) + p(b)\bar p(Y)}
{:} b Y(a) {:}
\\
& -
(-1)^{ p(a)p(b) + p(b)\bar p(Y) }
\int_{-\partial}^0
\!d\sigma\,
{b}_{\sigma}
Y(a)
\end{align*}
and
\begin{equation*}
Y([{a}_{\lambda}{b}])
=
(-1)^{p(a)\bar p(Y)}
{a}_{\lambda}Y(b) 
-
(-1)^{ p(a)p(b) + p(b)\bar p(Y) }
{b}_{-\lambda-\partial}
Y(a)
\,.
\end{equation*}
If we take the adjoint representation $M=V$ and use the skewsymmetry of the $\lambda$-bracket
and the quasicommutativity of the normally ordered product, the above equations become
\begin{equation*}
\begin{split}
Y({:}ab{:}) 
&=
{:}Y(a)b{:}
+
(-1)^{p(a)\bar p(Y)}
{:} a Y(b) {:} \,,
\\
Y([{a}_{\lambda}{b}])
&=
(-1)^{p(a)\bar p(Y)}
[{a}_{\lambda}Y(b)]
+
[Y(a)_\lambda b]
\,,
\end{split}
\end{equation*}
which means that $Y$ is a derivation (of parity $\bar p(Y)$) of the vertex algebra $V$.

Next, we discuss the special case of formula \eqref{eq:wick11} for $n=2$.
For simplicity, we consider the case of the adjoint module $M=V$.
Recall that an element $Y\in C^2_{\ch}(V)$
is a map 
\begin{equation*}
Y_\lambda=Y_{\lambda}^{z,w} \colon V^{\otimes2}\otimes\mb F[(z-w)^{\pm1}]
\to V[\lambda,\mu]/\langle\partial+\lambda+\mu\rangle\simeq V[\lambda]
\,.
\end{equation*}
Let us denote, according to \eqref{eq:Xint},
the normally ordered product and the $\lambda$-bracket
corresponding to $Y$ by
\begin{equation}\label{eq:Ynot}
\begin{split}
{:}ab{:}^Y
\, &=
(-1)^{p(a)} 
Y^{z,w}_{0} \Big( a \otimes b \otimes \frac1{w-z} \Big) 
\,, \\
[a_\lambda b]^Y
&=
(-1)^{p(a)} 
Y_{\lambda} ( a \otimes b \otimes 1 ) 
\,,
\end{split}
\end{equation}
so that
\begin{equation}\label{eq:Ynot2}
(-1)^{p(a)} Y^{z,w}_{\la} \Big(a \otimes b \otimes \frac1{w-z} \Big) 
=\,
{:}ab{:}^Y + \int_0^\la d\si [a_\si b]^Y \,.
\end{equation}
Then we have the following corollary.
\begin{corollary}\label{cor:wick12}
Let\/ $V$ be a vertex algebra and\/
$Y\in C^2_{\ch}(V)$ be a closed element of the VA cohomology complex: $dY=0$. 
Then with the notation \eqref{eq:Ynot}, we have:
\smallskip
\begin{align}
\notag
[a_\lambda {:}bc{:}]^Y
&+ 
(-1)^{
p(a)(\bar p(Y)+1)
}
[ a_{\lambda} {:}bc{:}^Y ] 
=
{:} [a_{\lambda}b] c {:}^Y
+
{:} [a_\lambda b]^Y c {:} \\
\label{eq:wick1a} 
& +
(-1)^{p(a)p(b)}
{:} b [a_{\lambda}c] {:}^Y
+
(-1)^{
p(b)(p(a)+\bar p(Y)+1)
}
{:} b [a_\lambda c]^Y {:} \\
\notag
& + 
\int_{0}^{\lambda}
d\tau
\big[
{[a_\lambda b]^Y}_{\tau} c
\big]
+
\int_0^\lambda d\sigma 
[ [a_{\lambda}b]_\sigma c]^Y
\,, 
\\[10pt]
\notag
[a_\lambda [{b_{\mu}c}] ]^Y
&+
(-1)^{
p(a)(\bar p(Y)+1)
}
[ a_{\lambda}  [b_\mu c]^Y ]
\\ \label{eq:wick1b}
& =
(-1)^{ p(a)p(b) }
[ b_\mu [a_{\lambda}c] ]^Y
+
(-1)^{
p(b)(p(a)+\bar p(Y)+1)
}
[ b_{\mu} [a_\lambda c]^Y ] 
\\ \notag
& +
[ {[a_\lambda b]^Y}_{\lambda+\mu} c ]
+
[ [a_{\lambda}b]_{\lambda+\mu} c]^Y
\,,
\\[10pt] \notag
[ {:}ab{:} _\lambda c ]^Y
&-
[ {{:}ab{:}^Y}_\lambda c ]
\\ \label{eq:wick2a} 
& =
{:} (e^{\partial\partial_\lambda}a) [b_{\lambda}c] {:}^Y
+
(-1)^{ p(a)(\bar p(Y)+1) }
{:} {(e^{\partial\partial_{\lambda}}a)} [b_\lambda c]^Y {:} 
\\ \notag
& +
(-1)^{ p(a)p(b) }
{:} (e^{\partial\partial_\lambda}b) [a_{\lambda}c] {:}^Y
+ 
(-1)^{ p(b)(p(a)+\bar p(Y)+1) }
{:} (e^{\partial\partial_{\lambda}} b) [a_\lambda c]^Y {:}
\\ \notag
& +
(\!-1)^{ p(a)p(b) }
\int_0^\lambda
[ b_{\sigma} [a_{\lambda-\sigma}c] ]^Y
+
(\!-1)^{
p(b)(p(a)+\bar p(Y)+1)
}
\int_0^{\lambda}
\!d\sigma\,
[ b_{\sigma} [a_{\lambda-\sigma} c]^Y ]
\,, 
\\[10pt] \notag
[ [{a_{\mu}b}]_\lambda c ]^Y
&+
[ {[a_\mu b]^Y}_\lambda c ]
\\ \label{eq:wick2b}
& =
[ a_\mu [b_{\lambda-\mu}c] ]^Y
+
(-1)^{
p(a)(\bar p(Y)+1)
}
[ a_{\mu} [b_{\lambda-\mu} c]^Y ] 
\\ \notag
& -
(-1)^{ 
p(a)p(b)
}
[ b_{\lambda-\mu} [a_{\mu}c] ]^Y
-
(-1)^{
p(b)(p(a)+\bar p(Y)+1)
}
[ b_{\lambda-\mu} [a_\mu c]^Y ]
\,.
\end{align}
\end{corollary}

The above equations explain the name ``Wick-type formulas'' for \eqref{eq:wick11}.
Indeed, for $Y=X$, \eqref{eq:wick1a} and \eqref{eq:wick2a}
become the left and right Wick formula for the vertex algebra $V$,
while \eqref{eq:wick1b} and \eqref{eq:wick2a}
become two equivalent forms of the Jacobi identity for $V$.

\begin{proof}[Proof of Corollary \ref{cor:wick12}]
Equation \eqref{eq:wick11} gives for $\ell=2$
\begin{align*}
& 
Y_{\lambda}
\big(
a\otimes{:}bc{:}\otimes
1
\big) \\
& =
(-1)^{p(b)(p(a)+\bar p(Y)+1)}
{:}
b Y_{\lambda}(a\otimes c\otimes 1) 
{:} 
+ (-1)^{p(c)(p(a)+p(b)+\bar p(Y)+1)}
{:}
c Y_{\lambda}(a\otimes b\otimes 1) 
{:} \\
& +
(-1)^{p(b)(p(a)+\bar p(Y)+1)}
\int_0^{\mu} \!\!\! d\sigma \,
\big[
b_{\sigma} Y_{\lambda}(a\otimes c\otimes 1) 
\big] \\
& + 
(-1)^{p(c)(p(a)+p(b)+\bar p(Y)+1)}
\int_0^{-\lambda-\mu-\partial}
\!d\sigma\,
\big[
c_{\sigma} Y_{\lambda}(a\otimes b\otimes 1) 
\big] \\
& +
(-1)^{
p(a)\bar p(Y)
+
p(b)
}
\Big [
a_{\lambda} Y^{z,w}_{\mu}\Big(b\otimes c\otimes \frac1{z-w}\Big) 
\Big ]
- 
\int_0^{\mu}\!d\sigma\,
Y_{\lambda}
\big(
a\otimes [{b_{\sigma}c}] \otimes 1
\big) \\
& - 
(-1)^{p(b)}
Y^{z,w}_{\lambda+\mu}
\Big (
[a_{\lambda}b] \otimes c\otimes \frac1{z-w}
\Big ) 
-
(-1)^{p(c)+p(b)p(c)}
Y^{z,w}_{-\mu-\partial}
\Big (
[a_{\lambda}c] \otimes b\otimes \frac1{z-w}
\Big ) 
\,,
\end{align*}
while for $\ell=1$ it gives
\begin{align*}
& 
Y_{\lambda}
\big(
{:}ab{:} \otimes c\otimes 1
\big) \\
& =
(-1)^{p(a)\bar p(Y)}
{:}
{(e^{\partial\partial_{\lambda}}a)}
Y_{\lambda}
(b\otimes c\otimes 1) 
{:} 
+ (-1)^{p(a)p(b)+p(b)\bar p(Y)}
{:}
{(e^{\partial\partial_{\lambda}} b)}
Y_{\lambda}
(a\otimes c\otimes 1) 
{:}
\\
& +
(-1)^{p(a)\bar p(Y)}
\int_0^{\mu} \! d\sigma \,
\big[
a_{\sigma} Y_{\lambda-\sigma}(b\otimes c\otimes 1) 
\big] 
- 
\int_0^{\mu}\!d\sigma\,
Y_{\lambda}
\big(
[{a_{\sigma}b}] \otimes c\otimes
1
\big) 
\\
& + 
(-1)^{p(a)p(b)+p(b)\bar p(Y)}
\int_0^{\lambda-\mu}
\!d\sigma\,
\big[
b_{\sigma} Y_{\lambda-\sigma}(a\otimes c\otimes 1) 
\big]
\\
& - 
(-1)^{
p(c)(p(a)+p(b)+\bar p(Y))
+
p(b)+p(c)
}
\Big [
c_{-\lambda-\partial}
Y^{z,w}_{\mu}
\Big(a\otimes b\otimes \frac1{z-w}\Big) 
\Big ]
\\
& -
(-1)^{ p(a)(1+p(b)) }
Y^{z,w}_{\lambda-\mu}
\Big (
(e^{\partial\partial_\lambda}b) \otimes [a_{\lambda}c] \otimes \frac1{z-w}
\Big ) \\
& -
(-1)^{p(b)}
Y^{z,w}_{\mu}
\Big (
(e^{\partial\partial_\lambda}a) \otimes [b_{\lambda}c] \otimes \frac1{z-w}
\Big ) 
\,.
\end{align*}
The above equations hold for every $\mu$.
Hence, 
by the sesquilinearity and symmetry conditions on $Y$ 
and the vertex algebra axioms, we obtain:
\begin{align}
\notag
Y_{\lambda}
( a &\otimes{:}bc{:}\otimes 1 ) 
\\ \notag
& =
(-1)^{p(b)(p(a)+\bar p(Y)+1)}
{:}
b Y_{\lambda}(a\otimes c\otimes 1) 
{:} 
+ 
{:}
Y_{\lambda}(a\otimes b\otimes 1) c
{:} \\ \label{eq:wick-n=2-l=2a}
& + 
(-1)^{p(c)(p(a)+p(b)+\bar p(Y)+1)}
\int_{-\partial}^{-\lambda-\partial}
\!d\sigma\,
[ c_{\sigma} Y_{\lambda}(a\otimes b\otimes 1) ] 
\\ \notag
& -
(-1)^{
p(a)\bar p(Y)
+
p(b)
}
\Big [ a_{\lambda} Y^{z,w}_0\Big (b\otimes c\otimes \frac1{w-z}\Big ) \Big ] 
\\ \notag
& +
(-1)^{p(b)}
Y^{z,w}_{\lambda}
\Big (
[a_{\lambda}b] \otimes c\otimes \frac1{w-z}
\Big ) 
\\ \notag
&-
(-1)^{(p(a)+1)(p(b)+1)}
Y^{z,w}_{0}
\Big (
b \otimes [a_{\lambda}c] \otimes \frac1{w-z}
\Big ) 
\,,
\end{align}

\begin{align}
\notag
Y_{\lambda}
( a &\otimes [{b_{\mu}c}] \otimes 1 ) 
+
(-1)^{
p(a)\bar p(Y)
+
p(b)
}
[ a_{\lambda}  Y_{\mu}(b\otimes c\otimes 1) ]
\\ \notag
& =
(-1)^{p(b)(p(a)+\bar p(Y)+1)}
[ b_{\mu} Y_{\lambda}(a\otimes c\otimes 1) ]
\\ \label{eq:wick-n=2-l=2b}
&+
(-1)^{
p(a)p(b)+p(a)+p(b)
}
Y_{\mu}
( b \otimes [a_{\lambda}c] \otimes 1 )
\\ \notag
& +
[ {Y_{\lambda}(a\otimes b\otimes 1)}_{\lambda+\mu} c ]
+
(-1)^{p(b)}
Y_{\lambda+\mu}
( [a_{\lambda}b] \otimes c\otimes 1 ) 
\,,
\end{align}

\begin{align}
\notag
& 
Y_{\lambda}
\big(
{:}ab{:} \otimes c\otimes 1
\big) 
\\ \notag
& =
(-1)^{p(a)\bar p(Y)}
{:}
{(e^{\partial\partial_{\lambda}}a)}
Y_{\lambda}
(b\otimes c\otimes 1) 
{:} 
+ (-1)^{p(b)(p(a)+\bar p(Y))}
{:}
{(e^{\partial\partial_{\lambda}} b)}
Y_{\lambda}
(a\otimes c\otimes 1) 
{:}
\\ \label{eq:wick-n=2-l=1a}
& + 
(-1)^{p(b)(p(a)+\bar p(Y))}
\int_0^{\lambda}
\!d\sigma\,
\big[
b_{\sigma} Y_{\lambda-\sigma}(a\otimes c\otimes 1) 
\big]
\\ \notag
& -
(-1)^{ p(b) }
\Big [
Y^{z,w}_{0}
\Big(a\otimes b\otimes \frac1{w-z}\Big) _\lambda c
\Big ]
\\ \notag
& +
(-1)^{ p(a)(1+p(b)) }
Y^{z,w}_{\lambda}
\Big (
(e^{\partial\partial_\lambda}b) \otimes [a_{\lambda}c] \otimes \frac1{w-z}
\Big ) 
\\ \notag
& +
(-1)^{p(b)}
Y^{z,w}_{0}
\Big (
(e^{\partial\partial_\lambda}a) \otimes [b_{\lambda}c] \otimes \frac1{w-z}
\Big ) 
\,,
\end{align}
and
\begin{align}
\notag
& Y_{\lambda}( [{a_{\mu}b}] \otimes c\otimes 1 ) 
+
(-1)^{ p(b) }
[ {Y_{\mu}(a\otimes b\otimes 1)}_\lambda c ]
\\ \label{eq:wick-n=2-l=1b}
& =
(-1)^{p(b)}
Y_{\mu}( a \otimes [b_{\lambda-\mu}c] \otimes 1 ) 
+
(-1)^{p(a)\bar p(Y)}
[ a_{\mu} Y_{\lambda-\mu}(b\otimes c\otimes 1) ] 
\\ \notag
& -
(-1)^{ p(a)(1+p(b)) }
Y_{\lambda-\mu}( b \otimes [a_{\mu}c] \otimes 1 )
-
(-1)^{p(b)(p(a)+\bar p(Y))}
[ b_{\lambda-\mu} Y_{\mu}(a\otimes c\otimes 1) ]
\,.
\end{align}
Rewriting \eqref{eq:wick-n=2-l=2a}--\eqref{eq:wick-n=2-l=1b} using \eqref{eq:Ynot}, we get \eqref{eq:wick1a}--\eqref{eq:wick2b}.
\end{proof}

\subsection{Bounded VA cohomology}\label{subsec:bounded}

Let $V$ be a filtered VA and $M$ be a filtered $V$-module.
Recall that we have the induced decreasing $\mb Z$-filtration $\fil^p C^n_{\ch}(V,M)$
of the superspace $C^n_{\ch}(V,M)$, defined in Section \ref{sec:pchfil}.

\begin{proposition}\label{thm:exhcoc}
\begin{enumerate}[(a)]
\item
The filtration on\/ $C^0_{\ch}(V,M)=M/\partial M$ is exhaustive.
\item\medskip
If\/ $V$ is finitely strongly generated and\/ $Y \in C^1_{\ch}(V,M)$ is such that\/ $dY=0$, 
then\/ $Y \in \fil^p C^1_{\ch}(V,M)$ for some\/ $p\in\mb Z$.
Consequently, the decreasing\/ $\mb Z$-filtration on the space of closed elements of\/ $C^1_{\ch}(V,M)$
is exhaustive.
\end{enumerate}
\end{proposition}
\begin{proof}
Part (a) is obvious: if $m\in \fil^rM$,
then the corresponding map $Y\colon\mb C\to M/\partial M$ given by $Y(1)=\bar m$
lies in $\fil^{-r}C^0_{\ch}(V,M)$.
For part (b), recall that an element of $C^1_{\ch}(V,M)$ is an $\mb F[\partial]$-module homomorphism
$Y\colon V\to M$. Suppose that such $Y$ is closed, i.e., $dY=0$.
This, in particular, means that $Y$ is a derivation of the normally ordered products:
\begin{equation}\label{eq:wickn1}
Y
\big({:}uv{:}\big) 
=
(-1)^{p(u)\bar p(Y)}
{:}{u}Y(v) {:} \\
+ 
{:} Y(u){v} {:}
\end{equation}
(which is Wick formula \eqref{eq:wick11} in the special case $n=1$).
Hence, such $Y$ is uniquely determined by its values on the finite set of (strong) 
generators $v_1,\dots,v_s$.
Furthermore, since the filtration on $V$ is exhaustive,
each generator $v_i$ lies in some member of the filtration of $V$,
and since the filtration on $M$ is exhaustive,
the image $Y(v_i)$ lies in some member of the filtration of $M$.
The claim follows.
\end{proof}

We define the \emph{space of bounded} $n$-\emph{cochains} as
\begin{equation}\label{eq:bounded}
C^n_{\chb}(V,M)
=
\bigcup_{p\in\mb Z}\fil^p C^n_{\ch}(V,M)
\,\,,\,\,\,\,
n\in\mb Z_+
\,.
\end{equation}
Note that the differential $d$ on $C^n_{\ch}(V,M)$ shifts the filtration by $1$:
\begin{equation}\label{eq:522}
d(\fil^rC^n_{\ch}(V,M))\subset\fil^{r+1}C^{n+1}_{\ch}(V,M)
\,.
\end{equation}
To see this for the adjoint module,
just observe that 
the element $X$ defining the VA structure lies in $\fil^1C^2_{\ch}(V)$,
hence, by Remark \ref{Wfilgr}(a), 
$d=\ad X$ shifts the filtration by $1$.
For an arbitrary $V$-module $M$,
we obtain the same result by considering the VA $V\oplus M$
and taking the subquotient defining $C^n_{\ch}(V,M)$.
Alternatively,
\eqref{eq:522} can be checked directly from the explicit formula \eqref{eq:dY}
of the differential $d$.
As a consequence of \eqref{eq:522}, 
\begin{equation}\label{eq:bounded2}
C_{\chb}(V,M)=\bigoplus_{n\in\mb Z_+} C^n_{\chb}(V,M)
\end{equation}
is invariant with respect to the action of the differential $d$.
\begin{definition}\label{def:bounded}
Let $V$ be a filtered VA and let $M$ be a filtered $V$-module.
The \emph{bounded} cohomology of $V$ with coefficients in $M$,
denoted by $H_{\chb}(V,M)$, is defined as the cohomology of the complex $(C_{\chb}(V,M),d)$.
\end{definition}

By Proposition \ref{thm:exhcoc}, we have
\begin{equation}\label{eq:H0b}
H^0_{\ch}(V,M)=H^0_{\chb}(V,M)
\,,
\end{equation}
and, provided that $V$ is finitely strongly generated,
\begin{equation}\label{eq:H1b}
H^1_{\ch}(V,M)=H^1_{\chb}(V,M)
\,.
\end{equation}

\begin{remark}
Note that the decreasing $\mb Z$-filtration of $C_{\ch}(V,M)$ induces 
a decreasing $\mb Z$-filtration $\fil^pH^n_{\ch}(V,M)$ in $n$-th cohomology.
Clearly, we have a canonical surjective map 
$\bigcup_{p\in\mb Z}\fil^pH^n_{\ch}(V,M)\twoheadrightarrow H^n_{\chb}(V,M)$;
this does not need to be an isomorphism, since $(dC^{n-1}_{\ch})\cap \fil^pC^n_{\ch}$
might be larger than $(dC^{n-1}_{\chb})\cap \fil^pC^n_{\ch}$.
\end{remark}

\begin{remark}
By Lemma \ref{lem:filsep}, the decreasing $\mb Z$-filtration of $C^n_{\ch}(V,M)$ is separated.
By Proposition \ref{thm:exhcoc}, it is also exhaustive for $n=0$ and $1$
under natural assumptions on $V$ and $M$.
We were unable to prove that it is also exhaustive for $n>1$
under the same assumptions.
This is the reason for considering the bounded VA cohomology $H_{\chb}(V,M)$
in place of the VA cohomology $H_{\ch}(V,M)$.
\end{remark}
%

\section{A spectral sequence for VA cohomology}\label{sec:ss}

Throughout this section, $V$ will be a vertex algebra with a very good filtration \eqref{eq:filtration}
(cf.\ Definition \ref{def:vgfil}).
For simplicity, we consider vertex algebra cohomology with coefficients in the adjoint module $M=V$;
the same results for an arbitrary $V$-module $M$ can be derived by the standard reduction procedure.

\subsection{Spectral sequence from the classical PVA cohomology to the VA cohomology}
\label{sec:ss1}

As in Sections \ref{sec:pchfil} and \ref{subsec:bounded}, an exhaustive increasing $\mb Z_+$-filtration of $V$ 
induces a separated decreasing $\mb Z$-filtration on the cohomology complex $C_{\ch}(V)$,
$\{\fil^p C_{\ch}(V)\}_{p\in\mb Z}$, and the differential $d=\ad_X$ satisfies (cf. \eqref{eq:522})
\begin{equation}\label{eq:dfpcva}
d\colon \fil^p C^n_{\ch}(V) \to \fil^{p+1} C^{n+1}_{\ch}(V) 
\,, \qquad p\in\mb Z \,, \; n\ge 0\,.
\end{equation}
This implies, in particular, that $d \fil^p C^n_{\ch}(V) \subset \fil^p C^{n+1}_{\ch}(V)$.
Moreover, we have a separated, exhaustive, decreasing $\mb Z$-filtration 
on the bounded cohomology subcomplex $C_{\chb}(V)$,
with the differential $d$ given by restriction.

Consider the spectral sequence $\{(E_r,d_r)\}_{r\ge0}$
associated to the filtered complex $(C_{\ch}(V),d)$
or to the bounded subcomplex $(C_{\chb}(V),d)$.
Its definition is reviewed in Appendix \ref{seg:ssfil}.
By Remark \ref{rem:dfil+1}, we have $d_0=0$ and
$$
E_1^{p,q} = E_0^{p,q} = \gr^p C^{p+q}_{\ch} \,, \qquad p,q \in\mb Z \,.
$$
The differential
\begin{equation}\label{eq:d1}
d_1 \colon E_1^{p,q} \to E_1^{p+1,q}
\end{equation}
is induced by the restriction of $d$ to $\fil^p C^{p+q}_{\ch}(V)$.

Recall the definitions \eqref{eq:cva} and \eqref{eq:Ccl} of the complexes $C_{\ch}(V)$ and $C_{\cl}(\mc V)$ in terms of the operads $\mc P^\ch(\Pi V)$ and $\mc P^\cl(\Pi\mc V)$, respectively. 
Since, by assumption, the filtration of $V$ is very good, by Theorem \ref{thm:chcl},
the operads $\gr \mc P^\ch(\Pi V)$ and $\mc P^\cl(\Pi\mc V)$ are isomorphic, where $\mc V = \gr V$.
Therefore, the Lie superalgebras $\gr W_{\ch}(\Pi V)$ and $W_{\cl}(\Pi \mc V)$ are isomorphic.
Note that $\mc V$ inherits a PVA structure from the VA structure of $V$ (see Proposition \ref{prop:assgr}). Let us denote by $\bar X \in W^1_{\cl}(\Pi \mc V)$ the odd element with $[\bar X,\bar X]=0$ that corresponds to the PVA structure via \eqref{eq:Xpva} (with $X$ replaced by $\bar X$). Then $\bar X$ is the image of $X+\fil^2 W_{\ch}^1(\Pi V) \in\gr^1 W_{\ch}(\Pi V)$ under the isomorphism $\gr W_{\ch}(\Pi V)\simeq W_{\cl}(\Pi \mc V)$ (see \cite[Theorem 10.13]{BDSHK18}). The differential in the classical complex $C_{\cl}(\mc V)$ is given by $\bar d=\ad_{\bar X}$ (see Section \ref{sec:clpva}).
Recall that each $C^n_{\cl}(\mc V) = \mc P^\cl(\Pi\mc V)(n)^{S_n}$ inherits a grading $\gr^p C^n_{\cl}(\mc V)$ from the grading \eqref{pclgrading} of $\mc P^\cl(\Pi\mc V)$.

\begin{lemma}\label{lem:d1}
The bigraded complexes\/ $(E_1,d_1)$ and\/ $(C_{\cl}(\mc V),\bar d)$ are isomorphic, i.e.,
\begin{equation}\label{eq:d2}
E_1^{p,q} = \gr^p C^{p+q}_{\ch}(V) \simeq \gr^p C^{p+q}_{\cl}(\mc V)
\,, \qquad p,q \in\mb Z \,,
\end{equation}
and the isomorphisms \eqref{eq:d2} are compatible with the actions of\/ $d_1$ and\/ $\bar d$.
\end{lemma}
\begin{proof}
We already know \eqref{eq:d2} from the above discussion. 
For $Y \in \fil^p C^{p+q}_{\ch}(V)$, consider the element
$$
Y+\fil^{p+1} C^{p+q}_{\ch}(V) \in \gr^p C^{p+q}_{\ch}(V) = \gr^p W^{p+q-1}_{\ch}(\Pi V) 
= E_1^{p,q}
$$
and its image
$$
\bar Y \in \gr^p C^{p+q}_{\cl}(\mc V) = \gr^p W^{p+q-1}_{\cl}(\Pi \mc V)
$$
under the isomorphism \eqref{eq:d2}.
By definition,
\begin{equation}\label{eq:d3}
d_1(Y+\fil^{p+1} C^{p+q}_{\ch}(V)) 
= [X,Y] + \fil^{p+2} C^{p+q+1}_{\ch}(V) \in \gr^{p+1} W^{p+q}_{\ch}(\Pi V) = E_1^{p+1,q} \,,
\end{equation}
and
\begin{equation}\label{eq:d4}
\bar d(\bar Y) = [\bar X,\bar Y] \in \gr^{p+1} W^{p+q}_{\cl}(\Pi \mc V) \,. 
\end{equation}
Then \eqref{eq:d3} is mapped to \eqref{eq:d4}, since $\gr W_{\ch}(\Pi V) \simeq W_{\cl}(\Pi \mc V)$ is an isomorphism of Lie superalgebras, by Theorem \ref{thm:chcl}.
\end{proof}

Recall that the cohomology of every filtered complex has an induced filtration, given explicitly by \eqref{eq:fphn}. 
In our case, $\fil^p H^n_{\ch}(V)$ is the image of $\fil^p C^n_{\ch}(V) \cap \ker d$ 
under the canonical projection 
$C^n_{\ch}(V) \cap \ker d \twoheadrightarrow H^n_{\ch}(V)$, namely
$$
\fil^p H^n_{\ch}(V)
=
(\fil^pC^n_{\ch}(V) \cap \ker d)/(dC^{n-1}_{\ch}(V)\cap\fil^pC^n_{\ch}(V))
\,.
$$
Similarly, the filtration $\fil^p H^n_{\chb}(V)$ on bounded $n$-th cohomology 
is the image of the same space $\fil^p C^n_{\ch}(V) \cap \ker d$ 
under the canonical projection 
$C^n_{\chb}(V) \cap \ker d \twoheadrightarrow H^n_{\chb}(V)$, namely
$$
\fil^p H^n_{\chb}(V)
=
(\fil^pC^n_{\ch}(V) \cap \ker d)/(dC^{n-1}_{\chb}(V)\cap\fil^pC^n_{\ch}(V))
\,.
$$
Since $C^{n-1}_{\chb}(V)\subset C^{n-1}_{\ch}(V)$,
we have, in particular, a canonical surjective map 
$\fil^p H^n_{\chb}(V)\twoheadrightarrow\fil^p H^n_{\ch}(V)$.

From Appendix \ref{seg:ssfil} and Lemma \ref{lem:d1}, we obtain the following theorem.

\begin{theorem}\label{thm:ss1}
Let\/ $V$ be a vertex algebra with a very good filtration, and\/ $\mc V=\gr V$ be its associated graded Poisson vertex algebra.
Then there exists a spectral sequence\/ $\{(E_r,d_r)\}_{r\ge1}$, 
whose first term is the classical PVA cohomology complex\/ $(C_{\cl}(\mc V),\bar d)$
and whose limit is\/ $E_\infty^{p,q} \simeq \gr^p H^{p+q}_{\chb}(V)$.
Furthermore, $\gr^p H^{p+q}_{\ch}(V)$ is a quotient space of\/ $E_\infty^{p,q}$.
\end{theorem}
\begin{proof}
The isomorphism $E_\infty^{p,q} \simeq \gr^p H^{p+q}_{\chb}(V)$ follows from \eqref{eq:egrh1},
while the surjection $E_\infty^{p,q} \twoheadrightarrow \gr^p H^{p+q}_{\ch}(V)$ follows from \eqref{eq:egrh}.
\end{proof}

We can therefore apply Lemma \ref{lem:fdh}(b) 
to the complex $(C_{\chb}(V),d)$
to obtain upper bounds on the bounded cohomology of $V$.

\begin{corollary}\label{cor:ss}
Let\/ $V$ be a vertex algebra with a very good filtration, and\/ $\mc V=\gr V$ be its associated graded Poisson vertex algebra. Then
$$
\dim H^n_{\chb}(V) \le \dim H^n_{\cl}(\mc V) 
\,, \qquad n\ge0 \,.
$$
\end{corollary}

Under additional assumptions, we derive that the bounded cohomology of $V$ is finite-dimensional
and bounded by the variational PVA cohomology of $\mc V$.

\begin{theorem}\label{thm:ss2}
Let\/ $V$ be a freely finitely generated vertex algebra with a very good filtration.
Assume that the associated graded Poisson vertex algebra\/ $\mc V=\gr V$ 
is conformal and generated by elements of positive conformal weight.
Then
\begin{equation}\label{eq:bound}
\dim H^n_{\chb}(V) \le \dim H^n_{\PV}(\mc V) < \infty
\,, \qquad n\ge0 \,.
\end{equation}
\end{theorem}
\begin{proof}
Since $V$ is freely finitely generated and its increasing $\mb Z_+$-filtration is very good,
$\mc V=\gr V$ is isomorphic, as a differential algebra,  
to an algebra of differential polynomials in finitely many (even or odd) variables. 
By Theorem \ref{thm:cl-pv}, $H_{\PV}(\mc V) \simeq H_{\cl}(\mc V)$. Then the claim follows from Corollary \ref{cor:ss} and \cite[Theorem 3.29]{BDSK19}.
\end{proof}

\begin{theorem}\label{thm:ss21}
Let\/ $V$ be a vertex algebra freely generated by its\/ $\mb F[\partial]$-submodule
$
R= \bigoplus_{j=0}^s \mb F[\partial] W_j
$,
where\/ $L:=W_0$ is a Virasoro element in\/ $V$ and each\/ $W_j$ 
has positive conformal weight\/ $\De_j$ with respect to\/ $L$
$($where\/ $\De_j\in\mb Q$ or\/ $\mb R)$.
Then
$$
\dim H^n_{\chb}(V) 
< \infty
\,, \qquad n\ge0 \,.
$$
\end{theorem}
\begin{proof}
Recall that the conformal weights have the property
$$
\De(a_{(n)} b) = \De(a)+\De(b)-n-1 \,, \qquad a,b \in V \,, \;\; n \in\mb Z
$$
(see Remark \ref{rem:nprod} and \cite{K96}). Moreover, by Definition \ref{def:vir},
$$
[L_\la a] = \bigl(\partial+\la\De(a)\bigr) a + \sum_{n\ge2} \frac{\la^n}{n!} L_{(n)} a \,,
$$
for every $a\in V$ of conformal weight $\De(a)$. Then it is easy to check that, if we let
$$
R_\De = \delta_{\De,1} \mb F[\partial]L \,\oplus \bigoplus_{1\le j\le s \,:\, \De_j=\De} \mb F[\partial] W_j
\,,
$$
the condition \eqref{RDe2} holds. Hence, by Proposition \ref{prop:vgfil}, the filtration \eqref{RDe3} of $V$ is very good. In the associated graded PVA $\mc V=\gr V$, we have:
$$
[\bar L_\la \bar L] = (\partial+2\la) \bar L \,, \qquad
[\bar L_\la \bar W_j] = (\partial+\De_j\la) \bar W_j \,, \qquad 1\le j \le s \,,
$$
where $\bar L=L+\fil^2 V \in \gr^1 V$ and 
$\bar W_j = W_j + \fil^{\De_j+1} V \in \gr^{\De_j} V$.
Since $\mc V$ is generated by $\bar L=\bar W_0,\bar W_1\dots,\bar W_s$ as a differential algebra, it follows that $\mc V$ is a conformal PVA. Hence, we can apply Theorem \ref{thm:ss2}.
\end{proof}

Recall that, for any simple finite-dimensional Lie algebra (or more generally, a basic Lie superalgebra)
$\mf g$, any nonzero nilpotent element $f\in\mf g$, and any non-critical level $k\in\mb F$ (i.e., $k\ne-h^\vee$, where $h^\vee$ is the dual Coxeter number of $\mf g$), we have the \emph{universal $W$-algebra} $W^k(\mf g,f)$, which satisfies the conditions of Theorem \ref{thm:ss21} (see \cite{KW04,DSK06}).

\begin{corollary}\label{cor:ss21}
For the universal\/ $W$-algebra\/ $V=W^k(\mf g,f)$ with\/ $k\ne-h^\vee$, we have
$$
\dim H^n_{\chb}(V) < \infty
\,, \qquad n\ge0 \,.
$$
\end{corollary}

In the next two subsections, we will consider special cases, in which the spectral sequence from Theorem \ref{thm:ss1} collapses. 

\subsection{The case of commutative VA}\label{sec:commva}

In this subsection, $V$ will be a commutative vertex algebra, i.e., such that $[V_\la V]=0$.
Then, by \eqref{eq:qc}, \eqref{eq:qa}, $V$ is a commutative associative algebra with respect to the normally ordered product with a derivation $\partial$.
We consider $V$ with the trivial filtration \eqref{eq:trfg1}, as in Example \ref{ex:vgfil2} and Section \ref{sec:chom}. Then $\mc V = \gr V = V$ with the same commutative associative product as $V$.

The filtration of $\mc P^\ch(\Pi V)(n)$ is given by \eqref{eq:trfg3}. Hence, the filtration of $C^n_{\ch}(V)$ has the form
$$
C^n_{\ch}(V) = \fil^0 C^n_{\ch}(V) \supset \fil^1 C^n_{\ch}(V) \supset\cdots\supset \fil^{n-1} C^n_{\ch}(V) \supset \fil^{n} C^n_{\ch}(V) = \{0\} \,.
$$
Since this filtration is finite for every $n$, it follows that the spectral sequence from Theorem \ref{thm:ss1} converges: for every $p,q\in\mb Z$, there exists $s\ge1$ such that $E^{p,q}_r = E^{p,q}_\infty = \gr^p H^{p+q}_{\ch}(V)$ for all $r\ge s$.

Here is a similar example, in which we have convergence.

\begin{example}\label{ex:vgfil3}
Consider an arbitrary VA $V$ with the filtration
$$
\fil^{-1} V = \fil^0 V = \{0\} \subset \fil^1 V = \fil^2 V = \cdots = V \,,
$$
as in Example \ref{ex:vgfil1}.
Then $\mc V = \gr V = V$ with the same $\la$-bracket as $V$, but with the zero product.
The filtration of $\mc P^\ch(\Pi V)(n)$ has the form (cf.\ \eqref{eq:trfg3}):
$$
\mc P^\ch(n) = \fil^{n-1} \mc P^\ch(n) \supset \fil^n \mc P^\ch(n) \supset\cdots\supset \fil^{2n-2} \mc P^\ch(n) \supset \fil^{2n-1} \mc P^\ch(n) = \{0\} \,.
$$
Hence, the filtration of $C^n_{\ch}(V)$ is finite for every $n$, and the spectral sequence converges to $E^{p,q}_\infty = \gr^p H^{p+q}_{\ch}(V)$.
\end{example}

Let us go back to the case when the VA $V$ is commutative. If we assume, in addition, that $V$ is an algebra of differential polynomials, then the spectral sequence collapses and we can completely determine $H_{\ch}(V)$.

\begin{theorem}\label{thm:ss3}
Let\/ $V$ be a superalgebra of differential polynomials in finitely many even or odd variables,
$$
V=\mb F\bigl[u_i^{(k)} \,\big|\, 1\le i \le N, \; k\in\mb Z_+ \bigr] 
\,, \qquad u_i^{(k)} = \partial^k u_i \,,
$$
considered as a vertex algebra with the zero\/ $\la$-bracket. Then
$$
H^n_{\ch}(V) \simeq H^n_{\PV}(V) \simeq C^n_{\PV}(V) \,, \qquad n\ge0 \,.
$$
Explicitly, this is the space of all collections of polynomials
$$
P^{i_1,\dots,i_n}_{\la_1,\dots,\la_n} 
\in V[\la_1,\dots,\la_n] / \langle\partial+\la_1+\dots+\la_n\rangle 
\,, \qquad 1\le i_s \le N \,,
$$
satisfying the symmetry conditions
$$
P_{\lambda_1,\dots,\lambda_s,\lambda_{s+1},\dots,\lambda_n}^{i_1,\dots,i_s,i_{s+1},\dots,i_n}
= (-1)^{\bar p(u_{i_s}) \bar p(u_{i_{s+1}})}
P_{\lambda_1,\dots,\lambda_{s+1},\lambda_s,\dots,\lambda_n}^{i_1,\dots,i_{s+1},i_s,\dots,i_n}
\,, \qquad 1\le s\le n-1 \,.
$$
Such a collection determines a unique (up to coboundary) cocycle\/ $Y\in C^n_{\ch}(V)$ by
$$
Y_{\lambda_1,\dots,\lambda_n}(u_{i_1} \otimes\dots\otimes u_{i_n} \otimes 1)
= P^{i_1,\dots,i_n}_{\la_1,\dots,\la_n} \,.
$$
\end{theorem}
\begin{proof}
Notice that $\mc V=\gr V=V$ is both a VA and a PVA with the same product and with the zero $\la$-bracket. We have
$$
E_2^{p,q} \simeq \gr^p H^{p+q}_{\cl}(V) \simeq \gr^p H^{p+q}_{\PV}(V) 
\simeq \gr^p C^{p+q}_{\PV}(V) \,.
$$
The second isomorphism holds by Theorem \ref{thm:cl-pv}, and the third because the $\la$-bracket in $V$ is zero, which implies that the differential in $C_{\PV}(V)$ is zero.
Recall from Section \ref{sec:clpva} that the grading of $C_{\cl}(V)$ is given by the number of edges of the graph, and $C_{\PV}(V) = \gr^0 C_{\cl}(V)$.
Hence, $E_2^{p,q} = 0$ for $p\ne0$ and $E_2^{0,q} \simeq C^q_{\PV}(V)$.
This implies the collapse of the spectral sequence: $d_r=0$ for all $r\ge 2$ and
$$
\gr^p H^{p+q}_{\ch}(V) = E_2^{p,q} \,.
$$
Therefore, $H^q_{\ch}(V) \simeq C^q_{\PV}(V)$.
The explicit description of all cochains in the variational complex of $V$ follows from
\cite[Remark 3.16 and Lemma 3.17(b)]{BDSK19}.
\end{proof}

\subsection{Relation between VA cohomology and LCA cohomology}\label{sec:uelca}

Let $R$ be a Lie conformal algebra such that $R=\bar R \oplus \mb F C$ as an $\mb F[\partial]$-module,
where $\partial C=0$ and $\bar R$ is a free finitely-generated $\mb F[\partial]$-module. 
Note that $C$ is central in $R$ and it acts trivially on any $R$-module. 
Hence, $\bar R \simeq R/\mb FC$ is an LCA, and any $R$-module is naturally an $\bar R$-module.

Fix $c\in\mb F$, and consider $V=V^c(R)$ with its canonical filtration (see Section \ref{sec:4.15}):
\begin{equation}\label{canfilvcr}
\fil^{-1} V = \{0\} \subset \fil^0 V = \mb F\vac \subset \fil^1 V = \bar R + \mb F\vac \subset \fil^2 V \subset\cdots \,.
\end{equation}
Then $\fil^1 V$ is an LCA subalgebra of $(V,[\,{}_\la\,])$, which is isomorphic to $R$. From now on, we will identify $R=\fil^1 V$ as an LCA. All $\fil^p V$ are $R$-modules; hence also $\bar R$-modules.
The filtration \eqref{canfilvcr} is very good and the associated graded Poisson vertex algebra $\mc V=\gr V$ is isomorphic to the universal PVA $\mc V(\bar R)$ over the LCA $\bar R$ (see Lemma \ref{lem:uvg} and Section \ref{sec:3.1}). As an $\bar R$-module, $\mc V$ is isomorphic to the symmetric algebra $S(\bar R)$ over the adjoint representation, which is a direct sum of the $\bar R$-modules $\gr^q \mc V = S^q \bar R$ ($q\ge0$). Therefore, by Proposition \ref{prop:pv-lc}(a), we have isomorphisms of complexes
$$
C_{\PV}(\mc V) \simeq C_{\mc Chom}(\bar R,\mc V) 
\simeq \bigoplus_{q=0}^\infty C_{\mc Chom}(\bar R,S^q \bar R) \,.
$$

Recall that the complex $C_{\PV}(\mc V)$ is graded, so that
$Y \in C^n_{\PV}(\mc V)$ has degree $p$ if and only if, for all $t\in\mb Z$,
$$
Y(\gr^t (\mc V^{\otimes n})) \subset 
(\gr^{t-p} \mc V)[\la_1,\dots,\la_n] / \langle\partial+\la_1+\dots+\la_n \rangle
$$
(cf.\ \eqref{pclgrading}).
Restricting $Y$ to $\bar R^{\otimes n} \subset \gr^n (\mc V^{\otimes n})$
and using Remark \ref{rem:pva-vac},
we see that
$$
\gr^p C^{p+q}_{\PV}(\mc V) = \gr^p C^{p+q}_{\mc Chom}(\bar R,\mc V) 
= C^{p+q}_{\mc Chom}(\bar R,S^q \bar R) \,.
$$
We also have the corresponding grading of the cohomology:
$$
\gr^p H^{p+q}_{\PV}(\mc V) = \gr^p H^{p+q}_{\mc Chom}(\bar R,\mc V) 
= H^{p+q}_{\mc Chom}(\bar R,S^q \bar R) \,.
$$

\begin{proposition}\label{prop:hlcrv}
With the above notation, suppose that\/ 
$H_{\PV}(\mc V) \simeq H_{\mc Chom}(\bar R,\mb F)$, where the isomorphism is induced from the restriction \eqref{eq:rvm2} and the inclusion\/ $\mb F \simeq \mb F 1 \hookrightarrow \mc V$.
Then\/ $H_{\mc Chom}(\bar R,V) \simeq H_{\mc Chom}(\bar R,\mb F)$, where the isomorphism is induced from the inclusion\/ $\mb F \simeq \mb F \vac \hookrightarrow V$.
\end{proposition}
\begin{proof}
The canonical filtration of $V$ induces a filtration of the complex $C_{\mc Chom}(\bar R,V)$, so that
$Y \in\fil^p C^n_{\mc Chom}(\bar R,V)$ if and only if, for all $s\in\mb Z$,
$$
Y(\fil^s (\bar R^{\otimes n})) \subset 
(\fil^{s-p} V)[\la_1,\dots,\la_n] / \langle\partial+\la_1+\dots+\la_n \rangle
$$
(cf.\ \eqref{fil4-ref}, \eqref{eq:ch-chom} and Lemma \ref{lem:rvm}). 
Here $\bar R$ is equipped with the trivial filtration 
$$
\fil^0\bar R = \{0\} \subset \fil^1\bar R = \bar R \,,
$$
hence, $\bar R^{\otimes n} = \fil^n \bar R^{\otimes n}$. Therefore, 
$$
\fil^p C^{p+q}_{\mc Chom}(\bar R,V) = C^{p+q}_{\mc Chom}(\bar R, \fil^q V)
\,, \qquad p,q \in\mb Z \,.
$$
This filtration is separated and exhaustive, because so is the filtration of $V$.

Since each $\fil^q V$ is an $\bar R$-submodule of $V$, the differential in $C_{\mc Chom}(\bar R,V)$ sends
$\fil^p C^{p+q}_{\mc Chom}(\bar R,V)$ to $\fil^{p+1} C^{p+q}_{\mc Chom}(\bar R,V)$.
As in Appendix \ref{seg:ssfil}, we obtain a spectral sequence $\{(E_r,d_r)\}_{r\ge1}$, whose first term is
$$
E_1^{p,q} = \gr^p C^{p+q}_{\mc Chom}(\bar R,V) \simeq C^{p+q}_{\mc Chom}(\bar R, \gr^q V)
= C^{p+q}_{\mc Chom}(\bar R,S^q \bar R) \,.
$$
Hence
$$
E_2^{p,q} \simeq H^{p+q}_{\mc Chom}(\bar R,S^q \bar R) \,.
$$

But, by assumption, $H^{p+q}_{\mc Chom}(\bar R,S^q \bar R) = 0$ for $q>0$.
Therefore, $E_2^{p,q} = 0$ for $q>0$ and $E_2^{p,0} \simeq H^p_{\mc Chom}(\bar R,\mb F)$.
By Remark \ref{rem:col}, the spectral sequence collapses at the second term: all $d_r=0$ for $r\ge 2$.
Since the filtration of $C_{\mc Chom}(\bar R,V)$ is separated and exhaustive, we get
$$
\gr^p H^{p+q}_{\mc Chom}(\bar R,V) \simeq E_\infty^{p,q} = E_2^{p,q}
\,, \qquad p,q\in\mb Z \,.
$$
This implies that
$$
H^n_{\mc Chom}(\bar R,V) = \gr^n H^{n}_{\mc Chom}(\bar R,V) \simeq H^n_{\mc Chom}(\bar R,\mb F)
\,, \qquad n\ge0 \,,
$$
completing the proof.
\end{proof}

Next, we compare the cohomology of the LCA $\bar R$ to that of its central extension $R$.

\begin{lemma}\label{lem:cext}
Suppose that the LCA\/ $R=\bar R+\mb F C$ is a nontrivial central extension of an LCA\/ $\bar R$ by an element\/ $C$ such that\/ $\partial C=0$. Let\/ $M$ be an\/ $R$-module such that\/ $\dim\ker\partial|_M = 1$. Then we have an injective morphism of complexes\/ $C_{\mc Chom}(\bar R,M) \hookrightarrow C_{\mc Chom}(R,M)$, induced by the projection\/ $R \twoheadrightarrow \bar R$, which is an isomorphism on\/ $C^n_{\mc Chom}(\bar R,M)$ for\/ $n\ne 1$. It induces surjective linear maps
\begin{equation}\label{eq:cext}
H^n_{\mc Chom}(\bar R,M) \twoheadrightarrow H^n_{\mc Chom}(R,M)
\,, \qquad n\ge 0\,,
\end{equation}
which are isomorphisms for all\/ $n\ne2$. For\/ $n=2$, the kernel of the map \eqref{eq:cext} is\/ $1$-dimensional and is spanned by the image of the\/ $2$-cocycle $\al\in H^2_{\mc Chom}(\bar R,\mb F)$, giving the central extension\/ $R$, under the map induced by the inclusion\/ 
$\mb F \simeq \ker\partial|_M \hookrightarrow M$.
\end{lemma}
\begin{proof}
This follows from Proposition 2.11 and the proof of Proposition 3.15 from \cite{BDSK19}.
The reason why $H^2_{\mc Chom}(\bar R,\mb F)$ and $H^2_{\mc Chom}(R,\mb F)$ differ is that there is a $1$-cochain $\be \in H^1_{\mc Chom}(R,\mb F)$, defined by
$$
\be_\la(C)=1+\langle\partial+\la\rangle \,, \qquad
\be_\la(a)=\langle\partial+\la\rangle 
\,, \qquad a\in\bar R \,,
$$
such that the image of $\al$ in $H^2_{\mc Chom}(R,\mb F)$ is $d\be$.
On the other hand, for $n\ge2$, any $n$-cochain vanishes when one of its arguments is $C$, by the sesquilinearity \eqref{eq:chomses}.
\end{proof}

\begin{remark}\label{rem:cext}
Suppose that the assumptions of Lemma \ref{lem:cext} hold, except that the central extension $R=\bar R\oplus\mb F C$ is trivial. Then $H^n_{\mc Chom}(\bar R,M) \simeq H^n_{\mc Chom}(R,M)$ for all $n\ne 1$, and we have an inclusion $H^1_{\mc Chom}(\bar R,M) \hookrightarrow H^1_{\mc Chom}(R,M)$ whose image has codimension $1$.
\end{remark}

Under the above assumptions on the LCA, we can determine the VA cohomology of 
its universal enveloping vertex algebra $V$.
This result will be applied, in particular, to the universal Virasoro VA (see Section \ref{sec:vir} below).

\begin{theorem}\label{thm:ss4}
Suppose that the Lie conformal algebra\/ $R=\bar R+\mb F C$ is a non\-trivial central extension of a Lie conformal algebra\/ $\bar R$ by an element\/ $C$ such that\/ $\partial C=0$, where\/
$\bar R$ is a free finitely-generated\/ $\mb F[\partial]$-module. For\/ $c\in\mb F$, consider the vertex algebra\/ $V^c(R)$ and its LCA subalgebra\/ $\bar R+\mb F\vac \simeq R$. Assume that
\begin{equation}\label{eq:ss5}
H_{\PV}(\mc V(\bar R)) \simeq H_{\mc Chom}(\bar R,\mb F) \,,
\end{equation}
where the isomorphism is induced from the restriction \eqref{eq:rvm2} and the inclusion\/ $\mb F \simeq \mb F 1 \hookrightarrow \mc V(\bar R)$. Then
\begin{equation}\label{eq:ss6}
H_{\chb}(V^c(R)) \simeq H_{\mc Chom}(\bar R,\mb F) \,, 
\end{equation}
where the isomorphism is induced from the restriction \eqref{eq:rvm1} and the inclusion\/ $\mb F \simeq \mb F \vac \hookrightarrow V^c(R)$.
\end{theorem}
\begin{proof}
As above, we consider $V:=V^c(R)$ with its canonical filtration \eqref{canfilvcr}, which is very good and satisfies $\mc V:=\gr V \simeq \mc V(\bar R)$. Then we have the spectral sequence $\{(E_r,d_r)\}_{r\ge1}$ from Theorem \ref{thm:ss1}. 
Recall that, as a differential algebra, $\mc V \simeq S(\bar R)$ is isomorphic to an algebra of differential polynomials in finitely many (even or odd) variables. Hence, by Theorem \ref{thm:cl-pv}, 
\begin{equation}\label{eq:ss7}
E_2^{p,q} \simeq \gr^p H^{p+q}_{\cl}(\mc V) 
\simeq \gr^p H^{p+q}_{\PV}(\mc V) \,.
\end{equation}

As in the proof of Proposition \ref{prop:hlcrv}, we have, by assumption,
$\gr^p H^{p+q}_{\PV}(\mc V) 
= 0$ for $q>0$.
Therefore, $E_2^{p,q} = 0$ for $q>0$ and $E_2^{p,0} \simeq H^p_{\mc Chom}(\bar R,\mb F)$.
By Remark \ref{rem:col}, the spectral sequence collapses at the second term: all $d_r=0$ for $r\ge 2$.
Lemma \ref{lem:fdh2} then gives
\begin{equation}\label{eq:ss8}
\pi(\fil^p C^n_{\ch}(V) \cap \ker d) = \gr^p C^n_{\cl}(\mc V) \cap \ker\bar d \,,
\end{equation}
where $\pi$ denotes the canonical projection  
$\fil^p C^n_{\ch}(V) \twoheadrightarrow \gr^p C^n_{\ch}(V) \simeq \gr^p C^n_{\cl}(\mc V)$.

Recall that the LCA $R$ is identified as the subalgebra $\fil^1 V = \bar R+\mb F\vac$ of the LCA $(V,[\,{}_\la\,])$. By Lemma \ref{lem:rvm}, the restriction map \eqref{eq:rvm1} induces a morphism of complexes
$C_{\chb}(V) \to C_{\mc Chom}(R,V)$ and the corresponding linear maps $H^n_{\chb}(V) \to H^n_{\mc Chom}(R,V)$.
Due to Proposition \ref{prop:hlcrv} and Lemma \ref{lem:cext}, we have
$$
H^n_{\mc Chom}(R,V) \simeq H^n_{\mc Chom}(\bar R,V) \simeq H^n_{\mc Chom}(\bar R,\mb F) \simeq 
H^n_{\PV}(\mc V)
\,, \qquad n\ne 2 \,.
$$
Thus, we obtain linear maps $H^n_{\chb}(V) \to H^n_{\mc Chom}(\bar R,\mb F)$ for all $n\ne 2$, which are surjective by \eqref{eq:ss7} and \eqref{eq:ss8}.
On the other hand, by \eqref{eq:egrh1}, we have isomorphisms
$$
E_\infty^{p,q} = E_2^{p,q} \simeq \gr^p H^{p+q}_{\PV}(\mc V) 
\simeq \gr^p H^{p+q}_{\chb}(V) \,.
$$
We conclude that
$$
H^n_{\chb}(V) \simeq H^n_{\mc Chom}(R,\mb F) 
\,, \qquad n\ne2 \,.
$$

To finish the proof, it remains to show that $H^2_{\chb}(V) \simeq H^2_{\mc Chom}(\bar R,\mb F)$.
Recall from Lemma \ref{lem:cext} that $H^2_{\mc Chom}(R,\mb F)$ and $H^2_{\mc Chom}(\bar R,\mb F)$ only differ by the $2$-cocycle $\al \in C^2_{\mc Chom}(\bar R,\mb F)$, which gives the central extension $R$ of $\bar R$.
We know from \eqref{eq:ss8} that $\al$ can be lifted to a $2$-cocycle $Y \in \fil^2 C^2_{\ch}(V)$ such that
$$
Y_{\la_1,\la_2}^{z_1,z_2} (v \otimes 1) = \al_{\la_1,\la_2}(v) \,, \qquad v\in \bar R^{\otimes 2} \,.
$$
We need to check that $Y \ne dZ$ for any $Z \in C^1_{\chb}(V) \simeq \End_{\mb F[\partial]} V$.
Indeed, note that $Z(\vac) \in\mb \ker\partial = \mb F\vac$. If $Z(\vac)=0$, then the image of $Z$ under the
the restriction map \eqref{eq:rvm2} is in $C^1_{\mc Chom}(\bar R,\mb F)$, but $\al$ is nontrivial in 
$H^2_{\mc Chom}(\bar R,\mb F)$. Thus, $Z(\vac) = a\vac$ for some nonzero $a\in\mb F$.
By \cite[Eq.\ (7.6)]{BDSHK18}, we find for $v_1,v_2 \in V$ and
$h\in \mc O_2^{\star T} = \mb F[z_{12}^{\pm 1}]$:
\begin{align*}
(dZ)_{\la_1,\la_2}^{z_1,z_2} &(v_1 \otimes v_2 \otimes h)
= (-1)^{\bar p(v_1) \bar p(v_2)} X_{\la_2,\la_1}^{z_2,z_1} \bigl(Z(v_1) \otimes v_2 \otimes h \bigr)
\\
&+ X_{\la_1,\la_2}^{z_1,z_2} \bigl(v_1 \otimes Z(v_2) \otimes h \bigr)
- (-1)^{p(Z)} Z\bigl( X_{\la_1,\la_2}^{z_1,z_2}(v_1 \otimes v_2 \otimes h) \bigr) \,.
\end{align*}
Setting $v_1=v_2=\vac$ and $h=z_{21}^{-1}$, we obtain from \eqref{eq:Xint},
$$
(dZ)_{\la_1,\la_2}^{z_1,z_2} \bigl(\vac \otimes \vac \otimes z_{21}^{-1} \bigr) 
= - (-1)^{p(Z)} a\vac \ne 0 \,.
$$
On the other hand, since $\vac\in\fil^0 V$, $z_{21}^{-1} \in \fil^1 \mc O_2^{\star T}$ and $Y \in \fil^2 C^2_{\ch}(V)$, we have
$$
Y_{\la_1,\la_2}^{z_1,z_2} \bigl(\vac \otimes \vac \otimes z_{21}^{-1} \bigr) 
\subset (\fil^{-1} V)[\la_1,\la_2] / \langle\partial+\la_1+\la_2\rangle = 0 \,.
$$
This completes the proof of the theorem.
\end{proof}

\subsection{The case when the LCA $\bar R$ is abelian}\label{sec:rab}

Suppose that $\bar R$ is an abelian LCA (i.e., with zero $\la$-bracket), which is freely finitely-generated as an $\mb F[\partial]$-module, and let $R=\bar R+\mb F C$ be a nontrivial central extension of $\bar R$ by an element $C$ such that $\partial C=0$. Then the $\la$-bracket in $R$ has the form
$$
[C_\la R] = 0 \,, \qquad
[a_\la b] = \al_{\la,-\la-\partial}(a \otimes b) \, C \,, \qquad a,b \in\bar R \,,
$$
for some nonzero $2$-cocycle $\al\in H^2_{\mc Chom}(\bar R,\mb F) = C^2_{\mc Chom}(\bar R,\mb F)$.

Fix $c\in\mb F$, and consider $V=V^c(R)$ with its canonical filtration \eqref{canfilvcr}.
As before, we identify $R$ with the LCA subalgebra $\fil^1 V = \bar R+\mb F\vac$ of $(V,[\,{}_\la\,])$.
In this subsection, we will derive analogs of the results of Section \ref{sec:uelca}, with similar proofs.
Now the associated graded Poisson vertex algebra $\mc V=\gr V \simeq V(\bar R) \simeq S(\bar R)$ is a superalgebra of differential polynomials in finitely many (even or odd) variables, with the zero $\la$-bracket.
The PVA $\mc V^c := \mc V^c(R)$ is isomorphic to it as a differential algebra, but its $\la$-bracket is given by
\begin{equation}\label{eq:ss11}
[a_\la b] = \al_{\la,-\la-\partial}(a \otimes b) \, c \,, \qquad a,b \in\bar R
\end{equation}
(which is then extended to the whole PVA by the Leibniz rules L4 and L4').
Recall that, by Proposition \ref{prop:pv-lc}(b), we have an isomorphism of complexes
\begin{equation}\label{eq:ss12}
C_{\PV}(\mc V^c,\mc V^c) \simeq C_{\mc Chom}(\bar R,\mc V^c) \,.
\end{equation}
As an $\bar R$-module, $\mc V^c$ is filtered by its $\bar R$-submodules
$$
\fil^p \mc V^c := \sum_{q=0}^p S^q \bar R \,, \qquad p\ge -1 \,;
$$
however, $S^q \bar R$ are not $\bar R$-submodules of $\mc V^c$. In particular, $\fil^1 \mc V^c = \bar R+\mb F 1 \simeq R$ is the adjoint representation of $\bar R$ on $R$, given by \eqref{eq:ss11} and by $[a_\la 1] = 0$.

\begin{proposition}\label{prop:hlcrv2}
Let the LCA\/ $R=\bar R+\mb F C$ be a nontrivial central extension of an abelian LCA\/ $\bar R$, where\/ $\partial C=0$ and\/ $\bar R$ is freely finitely-generated as an\/ $\mb F[\partial]$-module.
Suppose that for some\/ $c\in\mb F$, we have\/ 
$H_{\PV}(\mc V^c,\mc V^c) \simeq H_{\mc Chom}(\bar R,R)$, where the isomorphism is induced from the restriction \eqref{eq:rvm2} and the inclusion\/ $R \simeq \fil^1 \mc V^c \hookrightarrow \mc V^c$.
Then\/ $H_{\mc Chom}(\bar R,V) \simeq H_{\mc Chom}(\bar R,R)$, where\/ $V=V^c(R)$ and the isomorphism is induced from the inclusion\/ $R \simeq \fil^1 V \hookrightarrow V$.
\end{proposition}
\begin{proof}
As in the proof of Proposition \ref{prop:hlcrv},
the canonical filtration \eqref{canfilvcr} of $V$ induces a 
separated exhaustive filtration of the complex $C_{\mc Chom}(\bar R,V)$.
This gives a spectral sequence $\{(E_r,d_r)\}_{r\ge1}$, whose first term is
$$
E_1^{p,q} = \gr^p C^{p+q}_{\mc Chom}(\bar R,V) \simeq C^{p+q}_{\mc Chom}(\bar R, \gr^q V)
= C^{p+q}_{\mc Chom}(\bar R,S^q \bar R) \,.
$$
Since the $\la$-bracket in $\bar R$ is zero, the differential $d_1=0$ and hence
$$
E_2^{p,q} = E_1^{p,q} \simeq C^{p+q}_{\mc Chom}(\bar R,S^q \bar R) \,.
$$
By construction, the differential 
$$
d_2 \colon E_2^{p,q} \to E_2^{p+2,q-1}
$$
is induced from the LCA differential
$$
d\colon C^{p+q}_{\mc Chom}(\bar R,\fil^q V) \to C^{p+q+1}_{\mc Chom}(\bar R,\fil^q V) \,.
$$
Indeed, the image of $C^{p+q}_{\mc Chom}(\bar R,\fil^q V)$ under $d$ lies in
$C^{p+q+1}_{\mc Chom}(\bar R,\fil^{q-1} V)$, because $[\bar R_\la\bar R]=0$.

Notice that $\fil^q V \simeq \fil^q \mc V^c$ as $\bar R$-modules, again due to the vanishing of the $\la$-bracket in $\bar R$. Therefore,
$$
E_3^{p,q} = \gr^p H^{p+q}_{\mc Chom}(\bar R,\fil^q V) \simeq \gr^p H^{p+q}_{\mc Chom}(\bar R,\fil^q \mc V^c)
$$
is the image of $H^{p+q}_{\mc Chom}(\bar R,\fil^q \mc V^c)$ under the linear map induced by the projection
$\fil^q \mc V^c \twoheadrightarrow S^q \bar R \simeq \fil^q \mc V^c / \fil^{q-1} \mc V^c$.
By assumption, $E_3^{p,q} = 0$ for $q>1$.
Hence, by Remark \ref{rem:col}, the spectral sequence collapses at the third term: all $d_r=0$ for $r\ge 3$.
As a result,
$$
\gr^p H^{p+q}_{\mc Chom}(\bar R,V) \simeq E_\infty^{p,q} = E_3^{p,q}
\,, \qquad p,q\in\mb Z \,.
$$
This implies, for every $n\geq0$, that
\begin{align*}
H^n_{\mc Chom}(\bar R,V) \simeq \bigoplus_{p=n-1}^n \gr^p H^n_{\mc Chom}(\bar R,V) 
\simeq E_3^{n,0} \oplus E_3^{n-1,1} 
\simeq H^n_{\mc Chom}(\bar R,\fil^1 V)
\,,
\end{align*}
which completes the proof.
\end{proof}

Under the above assumptions, we can determine the bounded VA cohomology of $V$.
This result will be applied to the free superfermion and free superboson vertex algebras (see Sections \ref{sec:ferm} and \ref{sec:bos} below).

\begin{theorem}\label{thm:ss5}
Suppose that the Lie conformal algebra\/ $R=\bar R+\mb F C$ is a non\-trivial central extension of an abelian Lie conformal algebra\/ $\bar R$ by an element\/ $C$ such that\/ $\partial C=0$, where\/
$\bar R$ is a free finitely-generated\/ $\mb F[\partial]$-module. For\/ $c\in\mb F$, consider the vertex algebra\/ $V^c(R)$ and its LCA subalgebra\/ $\bar R+\mb F\vac \simeq R$. Assume that
\begin{equation}\label{eq:ss9}
H_{\PV}(\mc V^c(R)) \simeq H_{\mc Chom}(\bar R,R) \,,
\end{equation}
where the isomorphism is induced from the restriction \eqref{eq:rvm2} and the inclusion\/ 
$R \simeq \fil^1 \mc V^c(R) \hookrightarrow \mc V^c(R)$. Then
\begin{equation}\label{eq:ss10}
H_{\chb}(V^c(R)) \simeq H_{\mc Chom}(\bar R,R) \,, 
\end{equation}
where the isomorphism is induced from the restriction \eqref{eq:rvm1} and the inclusion\/ 
$R \simeq \fil^1 V^c(R) \hookrightarrow V^c(R)$.
\end{theorem}
\begin{proof}
The proof is similar to the proofs of Theorem \ref{thm:ss4} and Proposition \ref{prop:hlcrv2}, so some details will be omitted.
As before, we consider $V:=V^c(R)$ with its canonical filtration \eqref{canfilvcr}, which is very good and satisfies $\mc V:=\gr V \simeq \mc V(\bar R) \simeq S(\bar R)$. Then we have the spectral sequence $\{(E_r,d_r)\}_{r\ge1}$ from Theorem \ref{thm:ss1}, and
$$
E_2^{p,q} \simeq \gr^p H^{p+q}_{\cl}(\mc V) 
\simeq \gr^p H^{p+q}_{\PV}(\mc V) = \gr^p C^{p+q}_{\PV}(\mc V) \,.
$$
We get
$$
E_3^{p,q} \simeq \gr^p H^{p+q}_{\PV}(\mc V^c) \,,
$$
where $\mc V^c :=\mc V^c(R)$.
By assumption, $E_3^{p,q} = 0$ for $q>1$ and 
$$
E_3^{n,0} \oplus E_3^{n-1,1} \simeq H^n_{\PV}(\mc V^c, \fil^1 \mc V^c)
\simeq H^n_{\mc Chom}(\bar R,R) \,.
$$
By Remark \ref{rem:col}, the spectral sequence collapses at the third term: all $d_r=0$ for $r\ge 3$.
The rest of the proof is as in the proof of Theorem \ref{thm:ss4}.
\end{proof}

\section{Examples of computations of VA cohomology}\label{sec:ex}

\subsection{Vertex algebra with $\partial=0$}\label{sec:comtriv}

Let $V$ be a vertex algebra with $\partial=0$. Then, by the sesquilinearity V2 (which implies L1),
the $\la$-bracket in $V$ is zero. By the quasicommutativity and quasiassociativity \eqref{eq:qc}, \eqref{eq:qa}, the normally ordered product is commutative and associative. 
We endow $V$ with the trivial increasing filtration \eqref{eq:trfg1}. Hence, $\gr V = \gr^0 V = V$, which is a PVA with the same product as $V$ and zero $\la$-bracket.

\begin{theorem}\label{thm:triv3}
Let\/ $V$ be a commutative associative\/ superalgebra and\/ $M$ be a\/ $V$-module.
Consider\/ $V$ as a (non-unital) vertex algebra with\/ $\partial=0$ and\/ $M$ with\/ $\partial=0$ as a module over the VA\/ $V$. Then the VA cohomology complex of\/ $V$ with coefficients in\/ $M$ is isomorphic to the Harrison cohomology complex of\/ $V$ with coefficients in\/ $M$. Explicitly,
a Harrison cochain $\ga\colon V^{\otimes n} \to M$ corresponds to a VA cochain\/
$Y \in C^n_{\ch}(V,M)$ such that
$$
Y_{\la_1,\dots,\la_n}^{z_1,\dots,z_n} 
\Bigl( v \otimes \frac1{z_{12} z_{23} \cdots z_{n-1,n}} \Bigr) 
= \ga(v) + \langle\la_1+\dots+\la_n\rangle
\,, \qquad v \in V^{\otimes n} \,.
$$
Consequently,
$
H^n_{\ch}(V,M) \simeq H^n_{\Har}(V,M) 
$
for all\/ $n\ge0$,
where the subscript\/ $\Har$ stands for the Harrison cohomology \cite{Har62}.
\end{theorem}
\begin{proof}
Note that, by considering the algebra $V\oplus M$ as in Sections \ref{sec:lieoper}, \ref{sec:4.2},
it is enough to prove the theorem for $M=V$.
In this case, by Theorems \ref{thm:triv} and \ref{thm:triv2},
we have an isomorphism of operads $\mc P^\ch \simeq \mc P^\cl$
and $\mc P^\cl(n) = \gr^{n-1} \mc P^\cl(n)$ for all $n\ge1$.
As a consequence, we get an isomorphism of complexes $C_{\ch}(V) \simeq C_{\cl}(V)$.
Explicitly, $Y \in C^n_{\ch}(V)$ corresponds to $Z \in C^n_{\cl}(V)$ such that
$$
Z^L(v) = Y(v \otimes p_L) \,,
$$
where $L$ is the line $1\to2\to3\to\cdots\to n$, and
$$
Z^\Ga(v) = Y(v \otimes p_\Ga) = 0 
\,, \qquad v \in V^{\otimes n} \,,
$$
for every graph $\Ga$ with $n$ vertices and $\le n-2$ edges.
For graphs $\Ga$ with $n-1$ edges, $Z^\Ga$ is determined by Remark \ref{rem:lines} and the symmetry of $Z$. 

Note that
$$
p_L = \frac1{z_{12} z_{23} \cdots z_{n-1,n}} \,,
$$
and $Y$ is uniquely determined by its values $Y(v \otimes p_L)$,
by Corollary \ref{cor:pgamma} and the symmetry conditions \eqref{20160629:eq5}.
Since $L$ is connected, $Z^L(v) \in V$ is independent of $\la_1,\dots,\la_n$.
Therefore,
$$
C^n_{\ch}(V) \simeq C^n_{\cl}(V) \simeq \Hom(V^{\otimes n},V) \,.
$$
Finally, the differential in $C_{\cl}(V)$ and the Harrison differential
coincide, as proved in \cite[Theorem 4.1]{BDSKV19}.
\end{proof}

\subsection{Virasoro VA}\label{sec:vir}

The \emph{universal Virasoro VA}  of central charge $c\in\mb F$ is defined as
$$
\Vir^c=V(R^{\vir})/ {:} V(R^{\vir})(C-c\vac) {:} \,,
$$
where $R^\vir$ is the Virasoro LCA from Example \ref{ex:virasoro-lca}.
It is freely generated by $\bar R^{\vir} = \mb F[\partial]L$,
and in $\Vir^c$ the $\lambda$-bracket of $L$ with itself is given by
$$
[L_\lambda L]=(\partial+2\lambda)L+\frac{c}{12}\lambda^3
\,.
$$
The canonical filtration of $\Vir^c$ is very good, and the associated graded is isomorphic to the PVA $\mc V^0(R^\vir) \simeq \mc V(\bar R^\vir)$ (see Lemma \ref{lem:uvg}).

By \cite[Theorem 4.17]{BDSK19}, $V=\Vir^c$ satisfies the conditions of Theorem \ref{thm:ss4}. 
Since $H_{\mc Chom}(\bar R^\vir,\mb F)$ is known by \cite{BKV99},
as an immediate consequence we obtain an explicit description of the bounded cohomology of $\Vir^c$ with coefficients in its adjoint representation.

\begin{theorem}\label{thm:vir}
For every central charge\/ $c\in\mb F$, we have
$$
\dim H^n_{\chb}(\Vir^c) = \begin{cases}
1, \quad\text{for} \;\; n=0,2,3, \\
0, \quad\text{otherwise.}
\end{cases}
$$
Explicitly, $H^n_{\chb}(\Vir^c,\Vir^c)$ for\/ $n=0,2,3$ is spanned over\/ $\mb F$ by, respectively, $\vac+\partial V \in V/\partial V$, and by unique (up to coboundary) cocycles\/ $Y$, $Z$ such that
$$
Y_{\la_1,\la_2}^{z_1,z_2}(L \otimes L \otimes 1) 
= \la_1^3 + \langle\partial+\la_1+\la_2\rangle
$$
and
$$
Z_{\la_1,\la_2,\la_3}^{z_1,z_2,z_3}(L \otimes L \otimes L \otimes 1) 
= (\la_1-\la_2)(\la_1-\la_3)(\la_2-\la_3) + \langle\partial+\la_1+\la_2+\la_3\rangle \,.
$$
\end{theorem}

\subsection{Free superfermion VA}\label{sec:ferm}

Let $\mf h$ be a finite-dimensional superspace, with parity $p$, and a super-skewsymmetric nondegenerate bilinear form $(\cdot|\cdot)$, as in Example \ref{ex:fermion-lca}.
The \emph{free superfermion VA} is defined as 
$$
F_{\mf h}
:= V^1(R^f_{\mf h})
= V(R^f_{\mf h}) / {:} V(R^f_{\mf h})(K-\vac) {:}
\,.
$$
It is freely generated by $\mb F[\partial]\mf h$, where the $\lambda$-bracket of $a,b\in\mf h$ is as in \eqref{eq:fermion} but with $K=1$. The canonical filtration of $F_{\mf h}$ is very good, and the associated graded is isomorphic to the PVA $\mc V^0(R^f_{\mf h})$ (see Lemma \ref{lem:uvg}). The latter is a superalgebra of differential polynomials with the zero $\la$-bracket.
The \emph{free superfermion PVA} is (cf.\ \eqref{eq:vcr}):
$$
\mc F_{\mf h}
:= \mc V^1(R^f_{\mf h})
= \mc V(R^f_{\mf h}) / \mc V(R^f_{\mf h})(K-1)
\,.
$$
We showed in \cite[Theorem 4.7]{BDSK19} that the cohomology of $\mc F_{\mf h}$ with coefficients in itself is trivial.
The same is true for $F_{\mf h}$, due to Theorem \ref{thm:ss5}.

\begin{theorem}\label{thm:fer}
We have
$$
H^0_{\ch}(F_{\mf h}) = \mb F \tint\vac \,, \qquad
H^1_{\ch}(F_{\mf h}) = 0 \,, \qquad
H^n_{\chb}(F_{\mf h}) = 0 \,, \qquad n\ge 2 \,.
$$
\end{theorem}
\begin{proof}
Recall that, by \eqref{eq:ss12},
$$
H_{\PV}(\mc F_{\mf h},\mc F_{\mf h}) \simeq H_{\mc Chom}(\bar R,\mc F_{\mf h}) \,,
$$
where $\bar R = R^f_{\mf h} / \mb F K = \mb F[\partial]\mf h$. 
Since, by \cite[Theorem 4.7]{BDSK19}, 
$$
H^0_{\PV}(\mc F_{\mf h},\mc F_{\mf h}) = \mb F \tint 1 \,, \qquad
H^n_{\PV}(\mc F_{\mf h},\mc F_{\mf h}) = 0 \,, \qquad n\ge 1 \,,
$$
all cocycles are equivalent to cocycles that take values in $\mb F 1\subset R \subset \mc F_{\mf h}$. Hence, \eqref{eq:ss9} holds, which implies \eqref{eq:ss10} and completes the proof.
\end{proof}

\subsection{Free superboson VA}\label{sec:bos}

Let $\mf h$ be a finite-dimensional superspace, with parity $p$, and a supersymmetric nondegenerate bilinear form $(\cdot|\cdot)$, as in Example \ref{ex:boson-lca}.
The \emph{free superboson} VA is defined as
$$
B_{\mf h}
:=
V^1(R^b_{\mf h})
= V(R^b_{\mf h}) / {:} V(R^b_{\mf h}) (K-\vac) {:}
\,.
$$
It is freely generated by $\mb F[\partial]\mf h$, where the $\lambda$-bracket of $a,b\in\mf h$ is as in  \eqref{eq:boson} with $K=1$. The canonical filtration of $F_{\mf h}$ is very good, and the associated graded is isomorphic to the PVA $\mc V^0(R^b_{\mf h})$ (see Lemma \ref{lem:uvg}). The latter is a superalgebra of differential polynomials with the zero $\la$-bracket.
The \emph{free superboson PVA} is
$$
\mc B_{\mf h}
:= \mc V^1(R^b_{\mf h})
= \mc V(R^b_{\mf h}) / \mc V(R^b_{\mf h})(K-1)
$$
 (cf.\ \eqref{eq:vcr}).
Its cohomology was found in \cite[Theorem 4.2]{BDSK19}. As a consequence, we determine the bounded cohomology of $B_{\mf h}$.
In order to state the result, fix a basis $\{u_1,\dots,u_N\}$ for $\mf h$ homogeneous with respect to parity, 
and let $\{u^1,\dots,u^N\}$ be its dual basis, so that $(u_i|u^j)=\delta_i^j$. 

\begin{theorem}\label{thm:bos}
For the free superboson VA\/ $B_{\mf h}$, we have
$$
H_{\chb}^n(B_{\mf h}) \simeq
H_{\PV}^n(\mc B_{\mf h}) \simeq (S^n(\Pi\mf h))^* \oplus (S^{n+1}(\Pi\mf h))^* 
\,, \qquad n\ge0 \,.
$$
Explicitly, 
an element $\alpha+\beta\in (S^n(\Pi\mf h))^* \oplus (S^{n+1}(\Pi\mf h))^*$
corresponds under this isomorphism to the unique (up to coboundary) $n$-cocycle\/ 
$Y \in C_{\ch}^{n}(B_{\mf h})$ such that
$$
Y_{\la_1,\dots,\la_n}^{z_1,\dots,z_n}(u \otimes 1)
= \al(u) + \sum_{j=1}^N \be(u\otimes u^j) u_j 
+ \langle\partial+\la_1+\dots+\la_n\rangle 
\,,\qquad u\in\mf h^{\otimes n}
\,.
$$\end{theorem}
\begin{proof}
By \eqref{eq:ss12} and \cite[Theorem 4.2]{BDSK19}, we have
$$
H_{\PV}(\mc B_{\mf h}) \simeq H_{\mc Chom}(\bar R,\mc B_{\mf h}) \simeq H_{\mc Chom}(\bar R,R) \,,
$$
where $\bar R = R^b_{\mf h} / \mb F K = \mb F[\partial]\mf h$ and $R$ is identified with $\bar R+\mb F 1 \subset \mc B_{\mf h}$.
Hence, we can apply Theorem \ref{thm:ss5} to conclude that
$H_{\chb}^n(B_{\mf h}) \simeq H_{\PV}^n(\mc B_{\mf h})$.
The explicit description of cocycles follows from \cite[Theorem 4.2]{BDSK19}.
\end{proof}

\subsection{Universal affine VA}\label{sec:aff}

Let $\mf g$ be a simple finite-dimensional Lie algebra.
The \emph{universal affine} VA at \emph{level} $k\in\mb F$ is defined as
$$
V^k_{\mf g} := V^k(\cur\mf g) = V(\cur\mf g) / {:} V(\cur\mf g)(K-k\vac) {:} \,,
$$
where $\cur\mf g$ is the affine LCA from Example \ref{ex:affine-lca}.
It is freely generated by $\mb F[\partial]\mf g$, where the $\lambda$-bracket of $a,b\in\mf g$ is as in  \eqref{eq:current} with $K=k$. The canonical filtration of $V^k_{\mf g}$ is very good, and the associated graded is isomorphic to the PVA 
$\mc V^0(\cur\mf g) \simeq \mc V(\overline{\cur}\,\mf g)$ 
(see Lemma \ref{lem:uvg}). 
\begin{conjecture}\label{thm:affine}
Let\/ $V^k_{\mf g}$ be the universal affine VA of\/ $\mf g$ at level\/ $k\neq -h^\vee$.
Then
$$
\dim H^n_{\chb}(V^k_{\mf g})
\,\leq\, 
\dim \Big(\bigwedge{}^n\mf g\oplus\bigwedge{}^{n+1}\mf g\Big)\,\,,\,\,\,\,n\geq0
\,.
$$
\end{conjecture}
%

\section{Application to integrability of evolution PDE}\label{sec:qint}

\subsection{Integrability via cohomology}

In this subsection we introduce a general cohomological framework which allows one to prove
integrability of both classical and quantum Hamiltonian systems
of evolution equations.
A cohomogical approach to integrability of classical Hamiltonian PDEs
was initiated in \cite{Kra88} and \cite{Ol87}, and further developed in \cite{DSK13}.
Let
$$
W=\bigoplus_{k\geq-1}W^k
$$
be a $\mb Z$-graded Lie superalgebra with parity $\bar p$.
The superspace $\Pi W^{-1}$, with the opposite parity $p=1-\bar p$, 
is called the space of \emph{Hamiltonian functionals},
and the space $W^1$ is called the space of \emph{structures}.

\begin{definition}\label{def:ham-str}
For an element $K\in W^1$, we define the bilinear product $\{\cdot\,,\,\cdot\}_K$ on $\Pi W^{-1}$
given by
\begin{equation}\label{eq:8.1}
\{f,g\}_K=(-1)^{p(f)}[[K,f],g]
\,.
\end{equation}
If $K\in W^1$ is odd (i.e., $\bar p(K)=\bar 1$) such that $[K,K]=0$,
we call $K$ a \emph{Poisson structure},
and the corresponding bilinear product \eqref{eq:8.1}
a \emph{Poisson bracket} on $\Pi W^{-1}$.
\end{definition}
Note that in the left-hand side of \eqref{eq:8.1} we view $f$ and $g$ as elements of $\Pi W^{-1}$,
while in the right-hand side we view them in $W^{-1}$,
and make computations in the Lie superalgebra $W$.
Also observe that $p(\{f,g\}_K)=p(f)+p(g)$, since $p(K)=\bar 0$.
Given a Poisson structure $K\in W^1$,
the associated to $f\in\Pi W^{-1}$
\emph{evolutionary vector field} is defined as $X_f=[K,f]\in W^0$.
\begin{lemma}\label{lem:8.2}
\begin{enumerate}[(a)]
\item
For every\/ $K\in W^1$, \eqref{eq:8.1} defines a super skewsymmetric bracket on\/ $\Pi W^{-1}$.
\item\medskip
If\/ $K$ is a Poisson structure,
\eqref{eq:8.1} defines a Lie superalgebra bracket on\/ $\Pi W^{-1}$.
\item\medskip
We have
\begin{equation}\label{eq:81a}
[X_f,X_g]=X_{\{f,g\}_K}
\,.
\end{equation}
\end{enumerate}
\end{lemma}
\begin{proof}
For $a,b\in W^{-1}$ we have,
by the Jacobi identity in the Lie superalgebra $W$,
$$
[[K,a],b]=(-1)^{(1+p(a))(1+p(b))}[[K,b],a]
\,,
$$
which is the same as the skewsymmetry condition
for the bracket \eqref{eq:8.1}.
This proves claim (a).
Next, we prove claim (b).
For $a,b,c\in W^{-1}$, we get, after a straightforward computation
in the Lie superalgebra $W$, using only the assumption that $K\in W^1$ is odd,
\begin{align*}
\{\{a,b\}_K,&c\}_K
-
\{a,\{b,c\}_K\}_K
+(-1)^{p(a)p(b)}
\{b,\{a,c\}_K\}_K
\\
&= \frac{1}{2}(-1)^{p(b)}
[[[[K,K],a],b],c]
\,.
\end{align*}
Hence, the Jacobi identity for the bracket \eqref{eq:8.1}
follows from the assumption that $[K,K]=0$.
Finally, claim (c) is just a restatement of the Jacobi identity for the bracket \eqref{eq:8.1}.
\end{proof}
By \eqref{eq:81a}, 
the map $f\mapsto X_f$ defines a Lie superalgebra homomorphism $\Pi W^{-1}\to W^0$,
whose image is the subalgebra of $W^0$ of evolutionary vector fields.

\begin{proposition}[Lenard--Magri scheme]\label{prop:8.3}
Let\/ $K,H\in W^1$ be odd and, for $N\geq1$,  
let\/ $h_0,h_1,\dots,h_N$ be even elements of\/ $\Pi W^{-1}$
such that
\begin{equation}\label{eq:8.2}
[H,h_n]=[K,h_{n+1}]
\,\;\text{ for all }\;\,
n=0,\dots,N-1
\,.
\end{equation}
Then
\begin{equation}\label{eq:8.3}
\{h_m,h_n\}_K=0=\{h_m,h_n\}_H
\,\;\text{ for all }\;\,
m,n=0,\dots,N
\,,
\end{equation}
where\/ $\{f,g\}_K$ and\/ $\{f,g\}_H$ are defined by \eqref{eq:8.1}.
\end{proposition}
\begin{proof}
We prove equation \eqref{eq:8.3} by induction on $|n-m|$.
Obviously, $\{h_n,h_n\}_K=\{h_n,h_n\}_H=0$,
since, by Lemma \ref{lem:8.2}(a),
both brackets are skewsymmetric on $\Pi W^{-1}$
and, by assumption, all elements $h_n$ are even in $\Pi W^{-1}$.
This proves the basis of the induction $m=n$.
For the inductive step, we may assume that $m>n\geq0$.
We have
$$
\{h_m,h_n\}_K
=
[[K,h_m],h_n]
=
[[H,h_{m-1}],h_n]
=
\{h_{m-1},h_n\}_H
\,,
$$
which vanishes by the inductive assumption.
Similarly,
\begin{align*}
\{h_m,h_n\}_H
&=
(-1)^{1+p(h_m)p(h_n)}
\{h_n,h_m\}_H
=
-
[[H,h_n],h_m] \\
& =
-
[[K,h_{n+1}],h_m]
=
-
\{h_{n+1},h_m\}_K
\,,
\end{align*}
which again vanishes by inductive assumption.
\end{proof}

\begin{theorem}\label{thm:8.4}
Suppose that\/ $K$ and\/ $H$ are compatible Poisson structures,
i.e., odd elements of\/ $W^1$ such that\/ $[H,H]=[K,K]=[K,H]=0$.
Assume, moreover, that
\begin{equation}\label{eq:8.5}
\ker\bigl(\ad K|_{W^0_{\bar 0}}\bigr)
\,\subset\,
[K,W^{-1}_{\bar 1}]
\end{equation}
$($i.e., the even part of the first cohomology of the complex\/ $(W,\ad K)$ vanishes$)$.
Then, if\/
$h_0,\dots,h_N\in \Pi W^{-1}$, $N\geq1$, are even elements
solving equations \eqref{eq:8.2},
there exists an even element\/ $h_{N+1}\in\Pi W^{-1}$
such that \eqref{eq:8.2} holds for\/ $N+1$. 
\end{theorem}
\begin{proof}
Since $H$ and $K$ are compatible Poisson structures, we have by the Jacobi identity
$$
[K,[H,h_N]]
=
-[H,[K,h_N]]
=
[H,[H,h_{N-1}]]
=
(\ad H)^2h_{N-1}
=0
\,.
$$
Hence, $[H,h_N]$ lies in the kernel of $\ad K$.
The claim follows by assumption \eqref{eq:8.5}.
\end{proof}
\begin{corollary}\label{cor:8.5}
If\/ $H,K$ are as in Theorem \ref{thm:8.4} and\/ $h$ is an even element of\/ $\Pi W^{-1}$
such that\/ $[K,h]=0$,
then there exists an infinite sequence of even elements
$$
h_0=h,h_1,h_2,\dots\in\Pi W^{-1}
$$ 
such that equation \eqref{eq:8.2}
holds for every\/ $n\in\mb Z_{\geq0}$.
\end{corollary}
\begin{proof}
Let $h_{-1}=0$ and $h_0=h$, and apply Theorem \ref{thm:8.4} recursively.
\end{proof}
\begin{theorem}[cf. \cite{BDSK09}]\label{thm:8.6}
Let\/ $H,K$ be as in Theorem \ref{thm:8.4}.
Consider two infinite sequences,
$h_0,h_1,h_2,\dots$ and $g_0,g_1,g_2,\dots$
of even elements of\/ $\Pi W^{-1}$ satisfying \eqref{eq:8.2},
and assume that\/ $[K,h_0]=0$.
Then
$$
\{h_m,h_n\}_{K,H}=
\{g_m,g_n\}_{K,H}=0,\quad
\{h_m,g_n\}_{K,H}=0
\,\,\text{ for all }\,\,
m,n\geq0
\,,
$$
where\/ $\{\cdot\,,\,\cdot\}_{H,K}$ denotes either of the two Poisson brackets.
\end{theorem}
\begin{proof}
The first equation holds by Proposition \ref{prop:8.3}.
We prove the second equation by induction on $m$.
We have
$$
\{h_0,g_n\}_K
=
[[K,h_0],g_n]
=
0
$$
since, by assumption, $[K,h_0]=0$.
Moreover,
$$
\{h_0,g_n\}_H
=
-\{g_n,h_0\}_H
=
-[[H,g_n],h_0]
=
-[[K,g_{n+1}],h_0]
=
[[K,h_0],g_{n+1}]
=0
\,,
$$
proving the base case of induction $m=0$.
For the inductive step, we have, for $m\geq1$,
$$
\{h_m,g_n\}_K
=
[[K,h_m],g_n]
=
[[H,h_{m-1}],g_n]
=
\{h_{m-1},g_n\}_H
\,,
$$
which vanishes by inductive assumption.
Similarly, we have
\begin{align*}
& \{h_m,g_n\}_H
=
-\{g_n,h_m\}_H
=
-[[H,g_n],h_m]
=
-[[K,g_{n+1}],h_m]
\\
& =
[[K,h_m],g_{n+1}]
=
[[H,h_{m-1}],g_{n+1}]
=
\{h_{m-1},g_{n+1}\}_H
\,,
\end{align*}
which again vanishes by induction.
\end{proof}
\begin{remark}\label{rem:8.7}
Note that a solution $h_{N+1}$ of equation \eqref{eq:8.2}
is unique up to adding an element from the kernel of $\ad K$.
Consider the increasing sequence of subspaces of $W^{-1}_{\bar 1}$,
$$
U_0\subset U_1\subset U_2\subset\dots\subset \widetilde U=\bigcup_{n\geq 0}U_n
\,,
$$
where $U_0=\ker(\ad K|_{W^{-1}_{\bar1}})$
and, recursively, $U_{n+1}=(\ad K)^{-1}[H,U_n]$ for every $n\geq 0$.
Then, by Proposition \ref{prop:8.3}, $\widetilde U$ is an abelian subalgebras
with respect to both $H$ and $K$-Poisson brackets:
$$
\{\widetilde U,\widetilde U\}_{H,K}=0
\,.
$$
Moreover, let $g_0,g_1,g_2,\dots\in W^{-1}_{\bar 1}$
be an infinite sequence satisfying \eqref{eq:8.2} for all $n\geq0$,
and let $V=\Span\{g_0,g_1,g_2,\dots\}\subset W^{-1}_{\bar 1}$.
Then, by Theorem \ref{thm:8.6}, we also have that
$$
\{\widetilde U+V,\widetilde U+V\}_{H,K}=0
\,.
$$
On the other hand, if $g_0^\prime,g_1^\prime,\dots\in W^{-1}_{\bar 1}$
is another infinite sequence satisfying \eqref{eq:8.2},
and $V^\prime=\Span\{g_0^\prime,g_1^\prime,\dots\}$,
we do not necessarily have that $\{V,V^\prime\}_{H,K}=0$.
\end{remark}

\subsection{Example 1: $W=W_{\PV}(\Pi\mc V)$}

Let $\mc V$ be a differential superalgebra, with parity denoted by $p$,
let $\partial$ be an even derivation on $\mc V$,
and consider the Lie superalgebra $W_{\PV}(\Pi\mc V)$
introduced in Section \ref{sec:clpva}.
Recall from \cite[Section 5.1]{DSK13} that $\Pi W_{\PV}^{-1}(\Pi\mc V)=\mc V/\partial\mc V$,
$W_{\PV}^0(\Pi\mc V)=\Der_{\partial}(\mc V)$ is the Lie superalgebra of all derivations of $\mc V$
commuting with $\partial$,
and odd elements $K\in W_{\PV}^1(\Pi\mc V)$ such that $[K,K]=0$
correspond bijectively to the PVA $\lambda$-brackets on $\mc V$,
via the map $K\mapsto\{\cdot\,_\lambda\,\cdot\}_K$ given by (cf. \eqref{eq:Xpva})
$$
\{f_\lambda g\}_K=(-1)^{p(f)}K_{\lambda,-\lambda-\partial}(f,g)
\,.
$$
Some of the commutators for the Lie superalgebra $W_{\PV}(\Pi\mc V)$
in low degrees are as follows.
Let $\tint f,\tint g\in\mc V/\partial\mc V$
(here and further $\tint:\,\mc V\to\mc V/\partial\mc V$ denotes the canonical
quotient map),
let $X,Y\in\Der_{\partial}(\mc V)$,
and let $K\in W_{\PV}^1(\Pi\mc V)$ be such that $[K,K]=0$.
We have
\begin{equation}\label{eq:brackets}
\begin{split}
& [\tint f,\tint h]=0
\,,\,\,
[X,\tint f]=\tint X(f)
\,,\,\,
[X,Y]=XY-YX
\,,\\
& [K,\tint f](g)=(-1)^{p(f)}\{f_\lambda g\}_K\big|_{\lambda=0}
\,.
\end{split}
\end{equation}
By the second and fourth equations in \eqref{eq:brackets},
the Poisson bracket \eqref{eq:8.1} on $\mc V/\partial\mc V$
associated to the Poisson structure $K\in W^1_{\PV}(\Pi\mc V)_{\bar 1}$
becomes
\begin{equation}\label{eq:8.6}
\{\tint f,\tint g\}_K
=
\tint \{f_\lambda g\}_K\big|_{\lambda=0}
\,.
\end{equation}
Furthermore, $\mc V$ is a left module over the Lie superalgebra $\mc V/\partial\mc V$
with the well-defined action
\begin{equation}\label{eq:8.7}
\{\tint f,g\}_K
=
\{f_\lambda g\}_K\big|_{\lambda=0}
\,,
\end{equation}
which is a derivation of both the $\lambda$-bracket and the product,
commuting with $\partial$.
Given a Hamiltonian functional $\tint h\in\mc V/\partial\mc V$
and a Poisson structure $K\in W^1_{\PV}(\Pi\mc V)_{\bar 1}$,
the corresponding \emph{Hamiltonian equation} is
\begin{equation}\label{eq:8.8}
\frac{du}{dt}
=
\{\tint h,u\}_K
\,\,,\,\,\,\,
u\in\mc V
\,.
\end{equation}
This equation is called \emph{integrable}
if $\tint h$ is contained in an infinite-dimensional abelian subalgebra
of the Lie algebra $\mc V/\partial\mc V$ (with the Poisson bracket \eqref{eq:8.6}).
Picking a basis $\tint h_0=\tint h,\tint h_1,\tint h_2,\dots$ of this abelian subalgebra,
we obtain a hierarchy of integrable equations
\begin{equation}\label{eq:8.9}
\frac{du}{dt_n}
=
\{\tint h_n,u\}_K
\,\,,\,\,\,\,
n\geq0
\,,
\end{equation}
which are compatible since the corresponding evolutionary vector fields $X_{h_n}$'s commute,
by \eqref{eq:81a}.

\begin{example}\label{ex:boson}
Let $\mc V=\mb F[u,u^\prime,u^{\prime\prime},\dots]$
be the algebra of differential polynomials in one differential variable $u$,
so that $\partial u^{(n)}=u^{(n+1)}$.
One has two compatible PVA $\lambda$-brackets on $\mc V$,
defined by
\begin{equation}\label{eq:8.10}
\{u_\lambda u\}_K=\lambda
\,\,,\,\,\,\,
\{u_\lambda u\}_H=(\partial+2\lambda)u+\frac{c}{12}\lambda^3\,.
\end{equation}
Note that condition \eqref{eq:8.5} holds by \cite{DSK12}.
Let $\tint h_0=\tint u$.
By the last equation in \eqref{eq:8.5},
it's easy to check, using sesquilinearity and the Leibniz rule,
that $[K,\tint h_0]=0$.
Hence we can apply Corollary \eqref{cor:8.5}
to construct an infinite sequence $\tint h_0,\tint h_1,\tint h_2,\dots$
such that \eqref{eq:8.2} holds.
Hence, by Proposition \ref{prop:8.3}
$$
\{\tint h_m,\tint h_n\}_{H,K}=0
\,\text{ for all }\,
m,n\geq0
\,.
$$
We can compute the first few integrals of motion using the recursive formula \eqref{eq:8.2}:
$$
\tint h_0=\tint u
\,,\,\,
\tint h_1=\frac12\tint u^2
\,,\,\,
\tint h_2=\frac12\tint (u^3-\frac{c}{12}u'^2)
\,,\dots
\,.
$$
The corresponding integrable hierarchy of classical Hamiltonian equations
is the \emph{classical KdV hierarchy}:
$$
\frac{du}{dt_0}
=
0
\,\,,\,\,\,\,
\frac{du}{dt_1}
=
u'
\,\,,\,\,\,\,
\frac{du}{dt_2}
=
3uu'+\frac{c}{12}u^{\prime\prime\prime}
\,,\dots
$$
\end{example}

\subsection{Example 2: $W=W_{\ch}(\Pi V)$}

Let $V$ be a vector superspace and let $\partial$
be an even endomorphism of $V$.
Consider the $\mb Z$-graded Lie superalgebra
$W_{\ch}(\Pi V)=\bigoplus_{k\geq-1}W_{\ch}^k$
defined in Section \ref{sec:4.2}.
This Lie superalgebra is described in detail in \cite{BDSHK18}.
We have
$W_{\ch}^{-1}=V/\partial V$,
$W_{\ch}^0$ is the Lie superalgebra of all endomorphisms
of the $\mb F[\partial]$-module $V$,
and the odd elements $K\in W_{\ch}^1$ such that $[K,K]=0$
correspond bijectively to the VA structures on $V$,
via the map $K\mapsto\int^\lambda\!d\sigma[\cdot\,_\sigma\,\cdot]_K$ given by \eqref{eq:Xint}.
Some of the commutators for the Lie superalgebra $W_{\ch}(\Pi V)$
in low degrees are given by formulas \eqref{eq:brackets}.
It follows that the Poisson bracket \eqref{eq:8.1} on $V/\partial V$
associated to the Poisson structure $K\in W^1_{\ch}(\Pi V)$
coincides with formula \eqref{eq:8.6}.

Again, $V$ is a left module over the Lie superalgebra $V/\partial V$
with the well-defined action \eqref{eq:8.7}.
It is a derivation of both the $\lambda$-bracket and the normally ordered product,
commuting with $\partial$.
Given a Hamiltonian functional $\tint h\in V/\partial V$
and a Poisson structure $K\in W^1_{\ch}(\Pi\mc V)_{\bar 1}$,
the corresponding \emph{quantum Hamiltonian equation} is
again \eqref{eq:8.8}.
The notions of integrability, etc., in the quantum case
are the same as in the classical case.

In a similar fashion as in Example \ref{ex:boson},
one obtains Hamiltonian equations of the quantum KdV hierarchy.
The details of this and other examples will be discussed in a subsequent publication.

\appendix
\section{The spectral sequence of a filtered complex}\label{seg:ssfil}

In this appendix, we recall the construction of a spectral sequence from a filtered cohomology complex in a slightly more general setting than is usually discussed in the literature
(see e.g.\ \cite{McL01}). 

Consider a cochain complex $C = \bigoplus_{n=0}^\infty C^n$, where each $C^n$ is a vector superspace over the field $\mb F$, equipped with a differential $d$, an even linear operator on $C$ such that $d^2=0$ and $d C^n \subset C^{n+1}$ for all $n$. We let $C^n=0$ for $n\le-1$.
We suppose that the complex $(C,d)$ has a decreasing filtration $\{\fil^p C\}_{p\in\mb Z}$, so that each space $C^n$ is filtered by subspaces:
\begin{equation}\label{eq:fil}
\cdots\supset \fil^{p-1} C^n \supset \fil^p C^n \supset \fil^{p+1} C^n \supset\cdots
\,, \qquad n\ge0 \,,
\end{equation}
and the differential $d$ is compatible with the filtration:
\begin{equation}\label{eq:dfil}
d \fil^p C^n \subset \fil^p C^{n+1}
\,, \qquad p\in\mb Z \,, \; n\ge0 \,.
\end{equation}
Furthermore, we will assume that the filtration is \emph{separated}, i.e.,
\begin{equation}\label{eq:sepfil}
\bigcap_{p\in\mb Z} \fil^p C^n = 0
\,, \qquad n\ge0 \,.
\end{equation}
We will denote by $H=\bigoplus_{n=0}^\infty H^n$ the cohomology of the complex $(C,d)$:
\begin{equation}\label{eq:hn}
H^n := H^n(C,d) = (C^{n} \cap \ker d) / d C^{n-1}
\,, \qquad n\ge 0 \,.
\end{equation}
The filtration of $C$ induces a decreasing filtration of its cohomology $H$:
\begin{equation}\label{eq:filhn}
\cdots\supset \fil^{p-1} H^n \supset \fil^p H^n \supset \fil^{p+1} H^n \supset\cdots
\,, \qquad n\ge 0 \,,
\end{equation}
where
\begin{equation}\label{eq:fphn}
\fil^p H^n = \frac{(\fil^p C^{n} \cap \ker d) + d C^{n-1}}{d C^{n-1}}
\simeq \frac{\fil^p C^{n} \cap \ker d}{\fil^p C^{n} \cap d C^{n-1}}
\end{equation}
is the image of $\fil^p C^{n} \cap \ker d \subset C^{n} \cap \ker d$ under the canonical projection
$C^{n} \cap \ker d \twoheadrightarrow H^n$.
In other words, $\fil^p H^n$ is the image of $H^n(\fil^p C,d)$ in $H^n(C,d)$ under the linear map induced by the inclusion
$\fil^p C \hookrightarrow C$.

\begin{remark}\label{rem:hfilsep}
If the filtration \eqref{eq:fil} of $C^n$ is separated, then the induced filtration \eqref{eq:filhn} of the cohomology $H^n$ is also separated. 
\end{remark}

Let
$$
\gr H = \bigoplus_{p\in\mb Z} \gr^p H = \bigoplus_{\substack{p\in\mb Z \\ n\ge0}} \gr^p H^n
\,, \qquad 
\gr^p H^n = \fil^p H^n / \fil^{p+1} H^n
\,,
$$
be the associated graded space.
Then by \eqref{eq:fphn}, we have
\begin{equation}\label{eq:grphn}
\gr^p H^n \simeq 
\frac{\fil^p C^{n} \cap \ker d}{(\fil^{p+1} C^{n} \cap \ker d) + (\fil^p C^{n} \cap d C^{n-1})} \,.
\end{equation}

The spectral sequence $\{(E_r,d_r)\}_{r\ge0}$ associated to the filtered complex $(C,d)$ is constructed as follows. For $p,q\in\mb Z$ and $r\ge-1$, let
\begin{equation}\label{eq:zpq}
\begin{split}
Z^{p,q}_r 
&= \{ \al\in \fil^p C^{p+q} \,|\, d\al \in \fil^{p+r} C^{p+q+1} \}
\\
&= \fil^p C^{p+q} \cap d^{-1} (\fil^{p+r} C^{p+q+1}) \,,
\end{split}
\end{equation}
and
\begin{equation}\label{eq:bpq}
\begin{split}
B^{p,q}_r 
&= \{ \al\in\fil^p C^{p+q} \,|\, \al=d\be \;\text{ for some }\;  \be\in \fil^{p-r} C^{p+q-1} \} 
\\
&= \fil^p C^{p+q} \cap d (\fil^{p-r} C^{p+q-1})
\\
&= d Z^{p-r,q+r-1}_r \,.
\end{split}
\end{equation}
Obviously, we have $d B^{p,q}_r = 0$ and
$$
B^{p,q}_r \subset B^{p,q}_s \subset Z^{p,q}_s \subset Z^{p,q}_r
\,, \qquad -1\le r\le s \,.
$$
Note also that $Z^{p+1,q-1}_{r-1} = Z^{p,q}_r \cap \fil^{p+1} C^{p+q}$. We define
\begin{equation}\label{eq:epq}
E^{p,q}_r := Z^{p,q}_r / (Z^{p+1,q-1}_{r-1} + B^{p,q}_{r-1}) 
\,, \qquad p,q \in\mb Z \,, \; r\in\mb Z_+ \,.
\end{equation}
Since
$$
d Z^{p,q}_r = B^{p+r,q-r+1}_r \subset Z^{p+r,q-r+1}_r 
$$
and
$$
d (Z^{p+1,q-1}_{r-1} + B^{p,q}_{r-1}) = B^{p+r+1,q-r}_{r-1} 
\subset Z^{p+r+1,q-r}_{r-1} + B^{p+r+1,q-r}_{r-1} \,,
$$
the differential $d$ induces linear maps
\begin{equation}\label{eq:drepq}
d_r \colon E^{p,q}_r \to E^{p+r,q-r+1}_r 
\end{equation}
such that $d_r^2=0$.
One checks that the cohomology of $(E_r,d_r)$ is isomorphic to $E_{r+1}$, i.e.,
\begin{equation}\label{eq:epqr}
(E^{p,q}_r \cap \ker d_r) / d_r E^{p-r,q+r-1}_r  \simeq E^{p,q}_{r+1} \,,
\end{equation}
so indeed we have a spectral sequence (see e.g.\ \cite{McL01}).

\begin{remark}\label{rem:epqr}
Due to \eqref{eq:epqr}, if $\dim E^{p,q}_s < \infty$ for some $p,q,s$, then
$\dim E^{p,q}_r \le \dim E^{p,q}_s$ for all $r\ge s$. In particular, $E^{p,q}_s = 0$
implies $E^{p,q}_r = 0$ for all $r\ge s$.
\end{remark}

Consider in more detail the $r=0$ term. Observe that, by \eqref{eq:fil} and \eqref{eq:dfil},
$$
Z^{p,q}_{-1} = Z^{p,q}_0 = \fil^{p} C^{p+q} 
\,, \qquad
B^{p,q}_{-1} = d \fil^{p+1} C^{p+q-1} \subset Z^{p+1,q-1}_{-1}
\,.
$$
Hence,
\begin{equation}\label{eq:epq0}
E^{p,q}_0 
= \fil^{p} C^{p+q} / \fil^{p+1} C^{p+q}
= \gr^p C^{p+q}
\,, \qquad p,q \in\mb Z \,.
\end{equation}
The differential $d_0\colon E^{p,q}_0 \to E^{p,q+1}_0$ is induced by the restriction
$d\colon \fil^{p} C^{p+q} \to \fil^{p} C^{p+q+1}$.

\begin{remark}\label{rem:dfil+1}
Suppose that the differential $d$ of $C$ satisfies the following stronger property than
\eqref{eq:dfil}:
\begin{equation}\label{eq:dfil+1}
d \fil^p C^n \subset \fil^{p+1} C^{n+1}
\,, \qquad p\in\mb Z \,, \; n\ge0 \,.
\end{equation}
Then we have $d_0=0$ and
$$
Z^{p,q}_0 = Z^{p,q}_1 = \fil^{p} C^{p+q} 
\,, \qquad
B^{p,q}_0 = d \fil^{p} C^{p+q-1} \subset Z^{p+1,q-1}_0
\,.
$$
Hence,
$E^{p,q}_1 = E^{p,q}_0 = \gr^p C^{p+q}$ and
$d_1\colon E^{p,q}_1 \to E^{p+1,q}_1$ is induced by the restriction
$d\colon \fil^{p} C^{p+q} \to \fil^{p+1} C^{p+q+1}$.
\end{remark}

The limit of the spectral sequence is defined as
\begin{equation}\label{eq:epqinf}
E^{p,q}_\infty := Z^{p,q}_\infty / (Z^{p+1,q-1}_\infty + B^{p,q}_\infty) 
\,, \qquad p,q \in\mb Z \,,
\end{equation}
where
\begin{equation*}
B^{p,q}_\infty := \bigcup_{r\ge0} B^{p,q}_r 
\subset
Z^{p,q}_\infty := \bigcap_{r\ge0} Z^{p,q}_r \,.
\end{equation*}

\begin{lemma}\label{lem:epqinf}
If\/ $\dim E^{p,q}_s < \infty$ for some\/ $p,q,s$, then\/
$\dim E^{p,q}_\infty \le \dim E^{p,q}_s$. In particular, $E^{p,q}_s = 0$
implies\/ $E^{p,q}_\infty = 0$.
\end{lemma}
\begin{proof}
Note that, by construction, $Z^{p,q}_r$ and $B^{p,q}_r$ are subspaces of $\fil^p C^{p+q}$ for all $r\ge-1$. Let $\pi\colon \fil^p C^{p+q} \twoheadrightarrow \gr^p C^{p+q}$
be the canonical projection. Then we have a tower of subspaces
\begin{equation*}
0 = \bar B^{p,q}_{-1} \subset
\bar B^{p,q}_{0} \subset \bar B^{p,q}_{1} \subset\cdots\subset 
\bar Z^{p,q}_{2} \subset \bar Z^{p,q}_{1} \subset \bar Z^{p,q}_{0} = \gr^p C^{p+q}
\,,
\end{equation*}
where $\bar Z^{p,q}_r := \pi(Z^{p,q}_r)$ and $\bar B^{p,q}_r := \pi(B^{p,q}_r)$.
We claim that
\begin{equation*}
E^{p,q}_r \simeq \bar Z^{p,q}_r / \bar B^{p,q}_{r-1}
\,, \qquad r\ge 0 \,.
\end{equation*}
Indeed, the composition of the restriction of $\pi$ to $Z^{p,q}_r$ with the canonical projection 
$\bar Z^{p,q}_r \twoheadrightarrow \bar Z^{p,q}_r / \bar B^{p,q}_{r-1}$ is clearly surjective,
and its kernel is equal to $Z^{p+1,q-1}_{r-1} + B^{p,q}_{r-1}$.
Similarly,
\begin{equation*}
E^{p,q}_\infty \simeq \bar Z^{p,q}_\infty / \bar B^{p,q}_\infty
\,,
\end{equation*}
where
\begin{equation*}
\bar B^{p,q}_\infty := \pi(B^{p,q}_\infty) = \bigcup_{r\ge 0} \bar B^{p,q}_r 
\,, \qquad
\bar Z^{p,q}_\infty := \pi(Z^{p,q}_\infty) = \bigcap_{r\ge 0} \bar Z^{p,q}_r \,.
\end{equation*}
The lemma then follows from the inclusions $\bar Z^{p,q}_\infty \subset \bar Z^{p,q}_s$
and $\bar B^{p,q}_\infty \supset \bar B^{p,q}_{s-1}$.
\end{proof}

Notice that
\begin{equation}\label{eq:zpqinf}
Z^{p,q}_\infty = \fil^p C^{p+q} \cap \ker d \,,
\end{equation}
since the filtration of $C$ is separated (see \eqref{eq:sepfil}).
On the other hand,
\begin{equation}\label{eq:bpqinf}
B^{p,q}_\infty = \fil^p C^{p+q} \cap d \Bigl( \bigcup_{m\in\mb Z} \fil^m C^{p+q-1} \Bigr) \,.
\end{equation}
If the filtration of $C$ is \emph{exhaustive}, i.e., if
\begin{equation}\label{eq:filex}
\bigcup_{m\in\mb Z} \fil^m C^n = C^n 
\,, \qquad n\ge0 \,,
\end{equation}
then $B^{p,q}_\infty$ will be equal to
\begin{equation}\label{eq:bpqex}
\widetilde B^{p,q}_\infty := \fil^p C^{p+q} \cap d C^{p+q-1}
\,.
\end{equation}
In this case, comparing \eqref{eq:grphn} and \eqref{eq:epqinf}, we obtain
\begin{equation}\label{eq:egrh1}
\gr^p H^{p+q} \simeq
Z^{p,q}_\infty / (Z^{p+1,q-1}_\infty + B^{p,q}_\infty)
= E^{p,q}_\infty 
\,, \qquad p,q\in\mb Z \,.
\end{equation}

\begin{remark}\label{rem:hfilex}
Assume that, for some fixed $n\ge0$, the filtration of $C^n \cap \ker d$ is exhaustive, i.e.,
\begin{equation}\label{eq:kerdex}
\bigcup_{m\in\mb Z} (\fil^m C^n \cap \ker d) = C^n \cap \ker d
\,.
\end{equation}
Then the induced filtration of $H^n=H^n(C,d)$ is exhaustive
(see \eqref{eq:filhn}, \eqref{eq:fphn}). This condition is weaker than
the filtration of $C^n$ being exhaustive.
\end{remark}

In general, without assuming the filtration of $C$ is exhaustive, we have
$$
B^{p,q}_\infty \subset \widetilde B^{p,q}_\infty \subset Z^{p,q}_\infty \,.
$$
Then from \eqref{eq:grphn} and \eqref{eq:epqinf}, instead of the isomorphisms \eqref{eq:egrh1}, we get
surjective linear maps
\begin{equation}\label{eq:egrh}
E^{p,q}_\infty \twoheadrightarrow \gr^p H^{p+q} \simeq
Z^{p,q}_\infty / (Z^{p+1,q-1}_\infty + \widetilde B^{p,q}_\infty)
\,, \qquad p,q\in\mb Z \,.
\end{equation}
As a consequence of \eqref{eq:egrh}, we obtain upper bounds on the dimension of the cohomology $H$.

\begin{lemma}\label{lem:fdh}
Let\/ $r_0=0$ or\/ $1$, and\/ $(C,d)$ be a cochain complex equipped with 
a decreasing separated filtration such that\/ $d\fil^p C \subset \fil^{p+r_0} C$
for all\/ $p\in\mb Z$.
Consider the associated graded complex\/ $\gr C$ with the differential\/
$\bar d \colon \gr^p C \to \gr^{p+r_0} C$ induced by\/ $d$.

\begin{enumerate}[(a)]
\item
Suppose that\/ $\dim H^{p,q}(\gr C,\bar d) < \infty$ for some\/ $p,q\in\mb Z$, 
where
$$
H^{p,q}(\gr C,\bar d) := \bigl(\gr^p C^{p+q} \cap \ker\bar d\bigr) \big/ \bar d \bigl(\gr^{p-r_0} C^{p+q-1}\bigr) \,.
$$
Then
$$
\dim \gr^p H^{p+q}(C,d) \le \dim H^{p,q}(\gr C,\bar d) \,.
$$

\item
Assume that, for some fixed\/ $n\ge0$, the filtration of\/ $C^n \cap \ker d$ is exhaustive.
Then\/ $\dim H^n(\gr C,\bar d) < \infty$ implies that\/
$\dim H^n(C,d) \le \dim H^n(\gr C,\bar d)$.
\end{enumerate}
\end{lemma}
\begin{proof}
(a)
By \eqref{eq:epq0} and Remark \ref{rem:dfil+1}, we have that
$E^{p,q}_{r_0} =  \gr^p C^{p+q}$ and
$\bar d = d_{r_0} \colon E^{p,q}_{r_0} \to E^{p+{r_0},q-{r_0}+1}_{r_0}$ is the corresponding differential. 
Thus, $H^{p,q}(\gr C,\bar d) \simeq E^{p,q}_{r_0+1}$,
and claim (a) follows from \eqref{eq:egrh} and Lemma \ref{lem:epqinf}.

(b)
As before, let us write $H^n=H^n(C,d)$ for short.
We have
$$
H^n(\gr C,\bar d) = \bigoplus_{p\in\mb Z} H^{p,n-p}(\gr C,\bar d) \,,
$$
and part (a) implies that $\gr^p H^n \ne 0$ only for finitely many $p\in\mb Z$.
Hence, the filtration of $H^n$ is finite, i.e., of the form
$$
\cdots = \fil^{k-1} H^n = \fil^{k} H^n \supset \fil^{k+1} H^n \supset\cdots\supset \fil^{\ell-1} H^n \supset \fil^{\ell} H^n = \fil^{\ell+1} H^n = \cdots
$$
for some integers $k\le\ell$. Since, by Remarks \ref{rem:hfilsep} and \ref{rem:hfilex}, the filtration of $H^n$ is separated and exhaustive, it follows that $\fil^{k} H^n = H^n$ and $\fil^{\ell} H^n = 0$. Thus,
$$
\dim H^n = \sum_{p=k}^{\ell-1} \dim\gr^p H^n \,,
$$
which together with part (a) completes the proof of (b).
\end{proof}

A spectral sequence $\{(E_r,d_r)\}$ is said to \emph{collapse} (or degenerate) at the $s$-th term if all differentials $d_r=0$ for $r\ge s$. We will use the following notation from the proof of Lemma \ref{lem:epqinf}. Let $\pi\colon\fil^p C \twoheadrightarrow \gr^p C$ be the canonical projection and
$\bar Z^{p,q}_r := \pi(Z^{p,q}_r)$, $\bar B^{p,q}_r := \pi(B^{p,q}_r)$. 
Then $E^{p,q}_r \simeq \bar Z^{p,q}_r / \bar B^{p,q}_{r-1}$ and we have short exact sequences
\begin{equation*}
0 \to \bar Z^{p,q}_{r+1} / \bar B^{p,q}_{r-1} \to \bar Z^{p,q}_r / \bar B^{p,q}_{r-1} 
\xrightarrow{\, d_r \,} \bar B^{p,q}_r / \bar B^{p,q}_{r-1} \to 0 \,.
\end{equation*}
If $d_r=0$ for $r\ge s$, we obtain
\begin{align}
\label{eq:col1}
\bar B^{p,q}_{s-1} &= \bar B^{p,q}_{s} = \bar B^{p,q}_{s+1} = \cdots = \bar B^{p,q}_\infty \,,
\\ \label{eq:col2}
\bar Z^{p,q}_{s} &= \bar Z^{p,q}_{s+1} = \bar Z^{p,q}_{s+2} = \cdots = \bar Z^{p,q}_\infty \,,
\\ \label{eq:col3}
E^{p,q}_s &\simeq E^{p,q}_{s+1} \simeq E^{p,q}_{s+2} \simeq\cdots\simeq E^{p,q}_\infty \,.
\end{align}

A common cause for collapse is given in the next remark.

\begin{remark}\label{rem:col}
Fix $s\ge 2$ and suppose that, for all $p\in\mb Z$, we have $E^{p,q}_s=0$ whenever $q<0$ or $q\ge s-1$. Then the spectral sequence $\{(E_r,d_r)\}$ collapses at the $s$-th term. This follows from 
\eqref{eq:drepq} and Remark \ref{rem:epqr}.
\end{remark}

The following lemma will be useful for us.

\begin{lemma}\label{lem:fdh2}
Let\/ $(C,d)$ be a cochain complex equipped with a decreasing separated filtration such that\/ $d\fil^p C \subset \fil^{p+1} C$ for all\/ $p\in\mb Z$. Consider the associated graded complex\/ $\gr C$ with the differential\/ $\bar d \colon \gr^p C \to \gr^{p+1} C$ induced by\/ $d$, and denote by\/
$\pi$ the canonical projection\/ $\fil^p C \twoheadrightarrow \gr^p C$. If the spectral sequence \eqref{eq:epq} collapses at the second term, then
$$
\pi(\fil^p C^{n} \cap \ker d) = \gr^p C^{n} \cap \ker\bar d
$$
for all\/ $p,n\in\mb Z$.
\end{lemma}
\begin{proof}
This follows from \eqref{eq:zpqinf} and \eqref{eq:col2}, since 
$\bar Z^{p,q}_2 = \gr^p C^{p+q} \cap \ker\bar d$.
\end{proof}




\end{document}